\documentclass{article}
\usepackage[utf8]{inputenc}
\usepackage{url}
\usepackage{hyperref}
\usepackage{graphicx}
\usepackage{amsmath,amssymb}
\allowdisplaybreaks
\usepackage{float}
\usepackage{graphicx}
\usepackage{appendix}
\usepackage{listings} 
\usepackage{subfigure}
\usepackage{cite}
\usepackage{bm}
\usepackage{makecell}
\usepackage{multirow}
\usepackage{algorithm}  
\usepackage{algorithmicx}
\usepackage{algpseudocode}  

\usepackage{multicol}
\newcommand\keywords[1]{\textbf{Keywords}:#1}
\newtheorem{theorem}{Theorem}[section]
\newtheorem{lemma}{Lemma}[section]
\newtheorem{remark}{Remark}[section]
\newtheorem{corollary}{Corollary}[section]

\newenvironment{proof}{\paragraph{Proof:}}{\hfill$\square$}
\makeatletter
\newenvironment{breakablealgorithm}
  {\begin{center}
     \refstepcounter{algorithm}
     \hrule height.8pt depth0pt \kern2pt
     \renewcommand{\caption}[2][\relax]{
       {\raggedright\textbf{\ALG@name~\thealgorithm} ##2\par}
       \ifx\relax##1\relax 
         \addcontentsline{loa}{algorithm}{\protect\numberline{\thealgorithm}##2}
       \else
         \addcontentsline{loa}{algorithm}{\protect\numberline{\thealgorithm}##1}
       \fi
       \kern2pt\hrule\kern2pt
     }
  }{\kern2pt\hrule\relax
   \end{center}
  }
\makeatother

\title{A New Reduced Basis Method for Parabolic Equations   Based on Single-Eigenvalue Acceleration}
\author{Qijia Zhai, Qingguo Hong, Xiaoping Xie}
\date{\today}

\begin{document}

\maketitle


\begin{abstract}
In this paper, we  develop  a new reduced basis (RB) method, named as Single Eigenvalue Acceleration Method (SEAM), for second order parabolic equations with homogeneous Dirichlet boundary conditions. The high-fidelity   numerical method  adopts the backward Euler scheme and  conforming simplicial finite elements  for the temporal and spatial discretizations, respectively.  Under the assumption that the time step size is sufficiently small and time steps   are not very large,  we show that the singular value distribution of the high-fidelity solution matrix $U$ is close to that of a rank one matrix.
We select the eigenfunction associated to the principal eigenvalue of the matrix $U^\top U$   as the basis of    Proper Orthogonal Decomposition (POD) method so as to obtain SEAM and a parallel SEAM. Numerical experiments confirm the efficiency of the new method.
\end{abstract}

\section{Introduction}
The Reduced Basis (RB) method is a type of    model order reduction   approach for numerical approximation of problems involving repeated solution of differential equations generated in engineering and applied sciences. It was first proposed in \cite{doi:10.2514/3.7539} for the analysis of nonlinear structures in the 1970s, and later extended to many other problems such as partial differential equations (PDEs) with multiple parameters or different initial conditions \cite{ eftang2011hp,https://doi.org/10.1002/zamm.19950750709, Rheinboldt1993OnTT,Fink1983OnTE, 10.1007/BF01391412,doi:10.1137/0910047,GUNZBURGER1989123,alfioquarteroni_2016_reduced},  
PDE-constrained parametric optimization and control problems \cite{iapichino2017multiobjective,manzoni2019certified,ito1998reduced, ito1998reduced1}, 
and inverse problems \cite{lieberman2010parameter, maday2015parameterized}.
It should be mentioned that the 
work in \cite{10.1115/1.1448332, doi:10.2514/6.2003-3847} has led to a decisive improvement in the computational aspects of RB methods, owing to an efficient criterion for the selection of the basis functions, a systematic splitting of the computational procedure into an offline (parameter-independent) phase and an online (parameter-dependent) phase, and the use of  posteriori error estimates that guarantee certified numerical solutions for the reduced problems. The  RB methods that include the offline and online phases  as their essential constituents have become the  most widely used ones.

The Proper Orthogonal Decomposition (POD) method, combined with the Galerkin projection method, is a typical RB method. 
It uses   a set of orthonormal bases, which can represent the known data in the sense of least squares,  to linearly approximate the target variables so as to   obtain a low-dimensional approximate model. Since it is optimal in the least square sense, the POD method has the property of completely relying,  without making any a priori assumption, on the data. 

The POD method, whose predecessor was initially presented by K.Pearson\cite{FRSLIIIOL} in 1901 as an eigenvector analysis method and was originally conceived in the framework of continuous second-order processes by Berkooz\cite{Berkooz1992}, 
has been widely used in many fields with different appellations. In the singular value analysis and sample identification, the method is called the Karhumen-Loeve expansion \cite{keinosuke_1990_introduction}. In statistics it is named as the principal component analysis (PCA) \cite{jolliffe_2011_principal}. In geophysical fluid dynamics and meteorological sciences, it is called  the  empirical orthogonal function method (EOF)\cite{majda_2003_systematic, selten_1997_baroclinic}. We can also see the widespread applications of POD method in fluid dynamics and viscous structures\cite{aubry_1988_the, crommelin_2004_strategies, holmes_2012_turbulence, LUMLEY1981215, moin_1989_characteristiceddy, rajaee_1994_lowdimensional, sirovich_1987_turbulence1, sirovich_1987_turbulence2, sirovich_1987_turbulence3}, optimal fluid control problems \cite{joslin_1997_selfcontained, ly_2002_proper}, numerical analysis of PDEs \cite{LUO2019xi, kunisch_2001_galerkin, kunisch_2002_galerkin, ahlman_2002_proper,cao_2006_reducedorder, 
luo_2007_an, luo_2007_proper, luo_2007_finite} and  machine learning \cite{Mainini20151612, doi:10.1080/21681163.2015.1030775, HESTHAVEN201855,doi:10.2514/1.J056161,SWISCHUK2019704}. 

In this paper, we  develop  a new RB/POD method, named as Single Eigenvalue Acceleration Method (SEAM), for a full discretization, using backward Euler scheme for temporal discretization and continuous simplicial finite elements for spatial discretization, of second order parabolic equations with homogeneous Dirichlet boundary conditions. 
 The idea of  SEAM   is inspired by  a POD numerical experiment for   a one-dimensional heat conduction  problem:
\begin{equation}\label{nml1}
\left\{\begin{array}{ll}
\frac{\partial u}{\partial t}=\frac{\partial^{2} u}{\partial x^{2}}, &  x \in (0,1), t \in (0,0.1],\\
u(x, 0)=\sin(4\pi x), & x \in (0,1), \\
u(0,t) =u(1,t)= 0,& t\in (0,0.1].
\end{array}\right.
\end{equation}
When we focus on the numerical solution matrix $U$ of (\ref{nml1}), where each column consists of the value of the numerical solution at a node, and the number of columns is the same as the number of time steps (We  divide $(0,T]$ into 1000 equal parts and $D$ into 99 equal parts to get the high-fidelity numerical solution; see Figure \ref{fig1}), we observe something interesting: the principal singular value is much larger than   other singular values of the numerical solution matrix! Here we recall  that  the singular values of $U$ are the arithmetic square roots of the eigenvalues of $U^\top U$. Notice that when we use the POD method   the choice of basis functions is often empirical.     The observation tells us that we only need to choose one basis function when solving this equation with POD. 

We show in    Figure  \ref{fig1} ,  Figure \ref{fig2}
and Table \ref{tab1}   the SEAM (POD with one basis function) solution and computational details, respectively.  We can see 
 that when SEAM is used, 
the computational  time can be saved greatly, and the relative error between the SEAM solution and the high-fidelity solution is very small. 

\begin{figure}[H]
    \centering
   \subfigure[High-fidelity numerical solution]{\includegraphics[width=5cm, height =4cm]{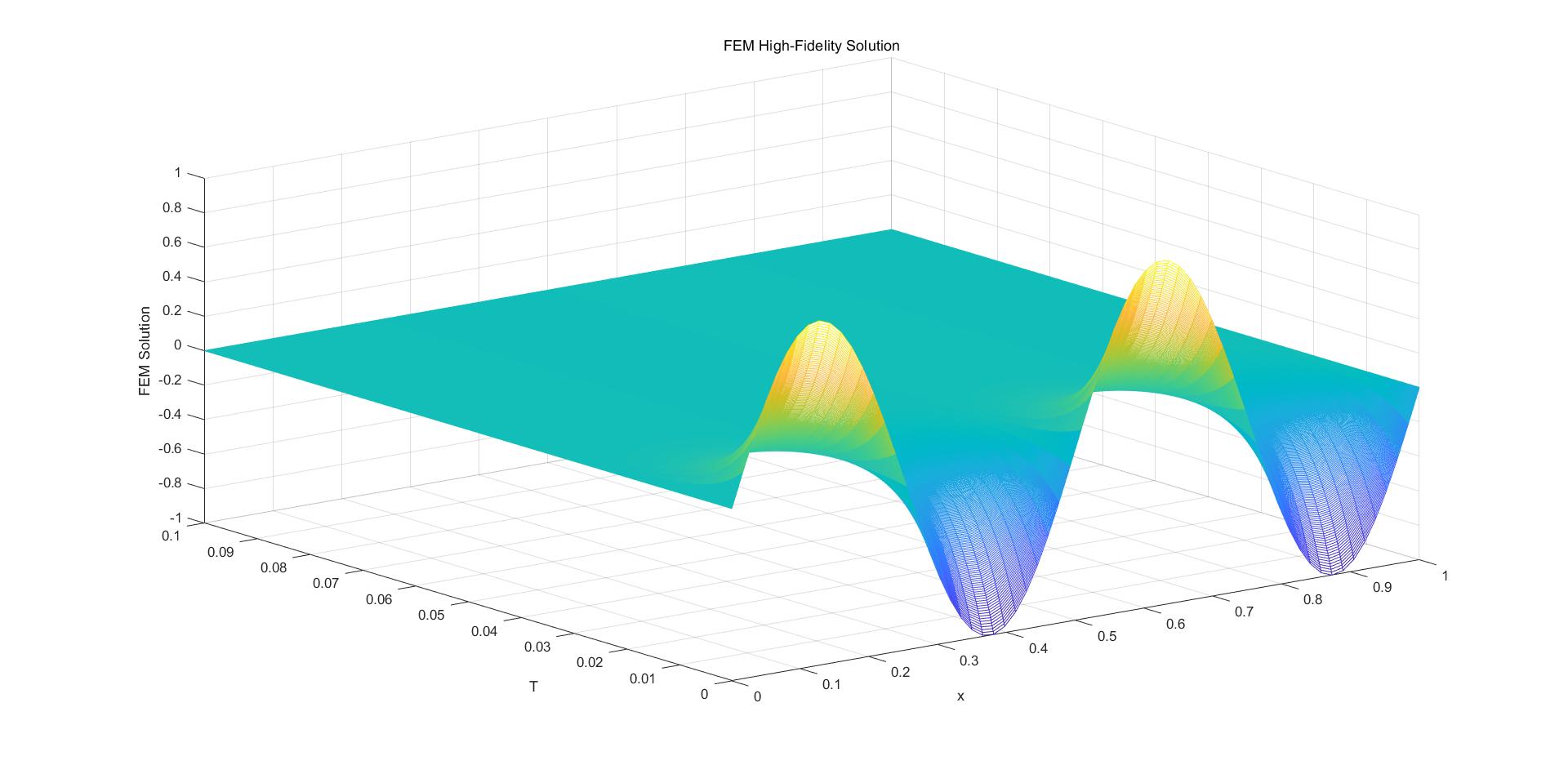}}\quad 
    \subfigure[SEAM solution]{\includegraphics[width=5cm, height =4cm]{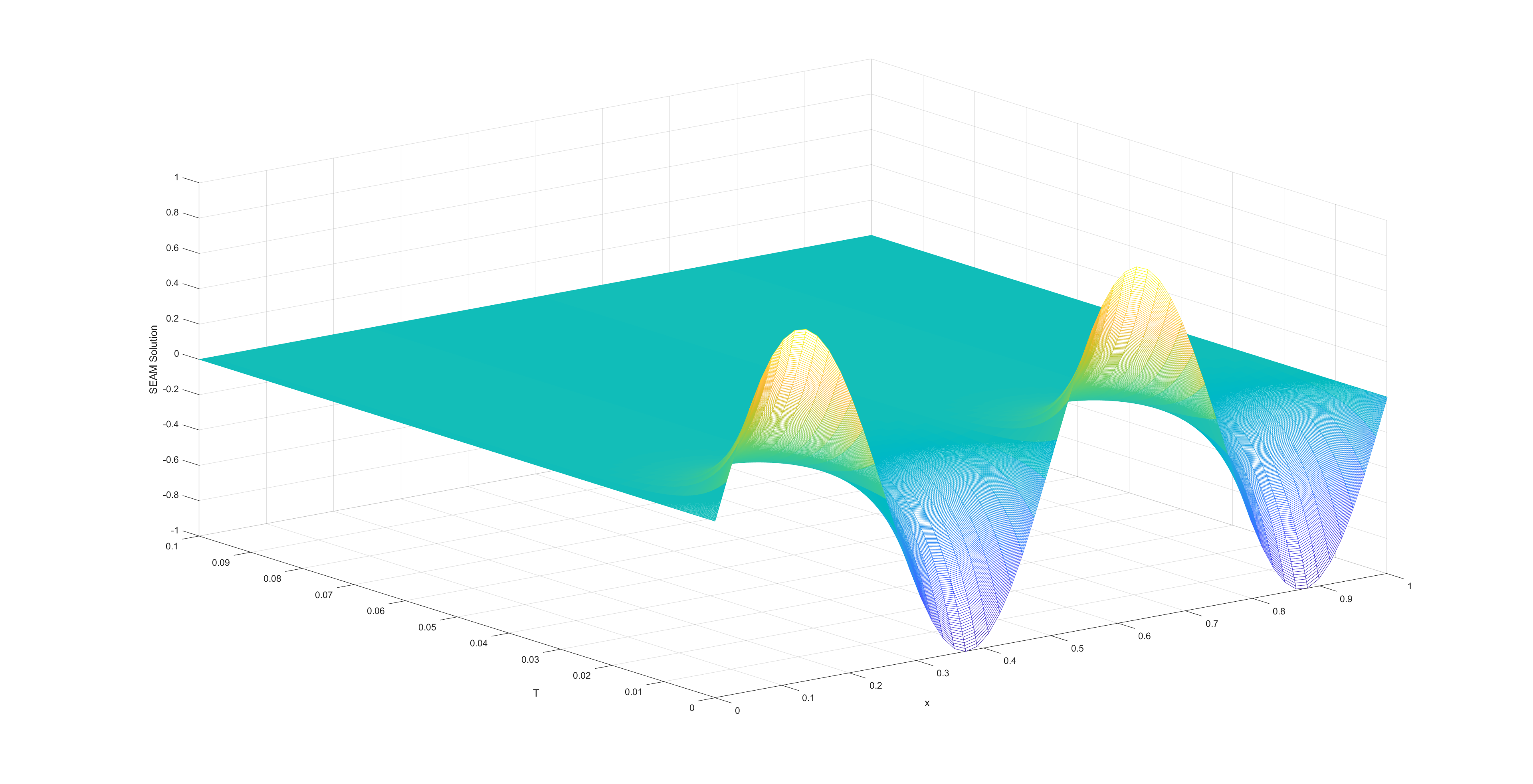}}
    \caption{High-Fidelity numerical solution and  SEAM solution for  (\ref{nml1}) }\label{fig1}

\end{figure}


\begin{figure}[H]
 \centering

    \subfigure[The first 20   eigenvalues of $U^\top U$ (from   big to small):
    the 1st one is 1007.63, and the 2nd one is 8.7e-9 ]{\includegraphics[width=5cm, height =4cm]{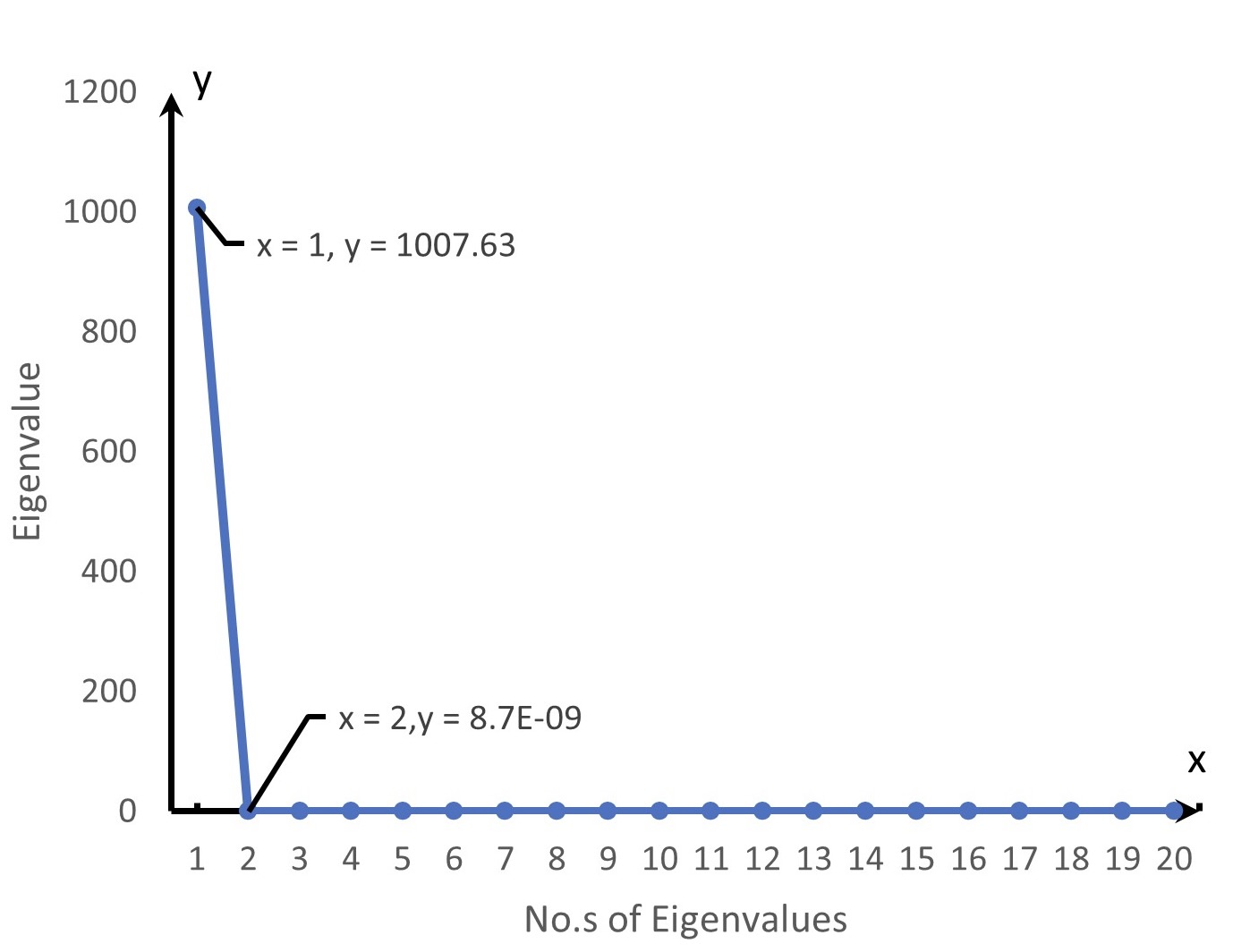}}\quad
    \subfigure[Relative error between high-fidelity solution and SEAM solution]{\includegraphics[width=5cm, height =4cm]{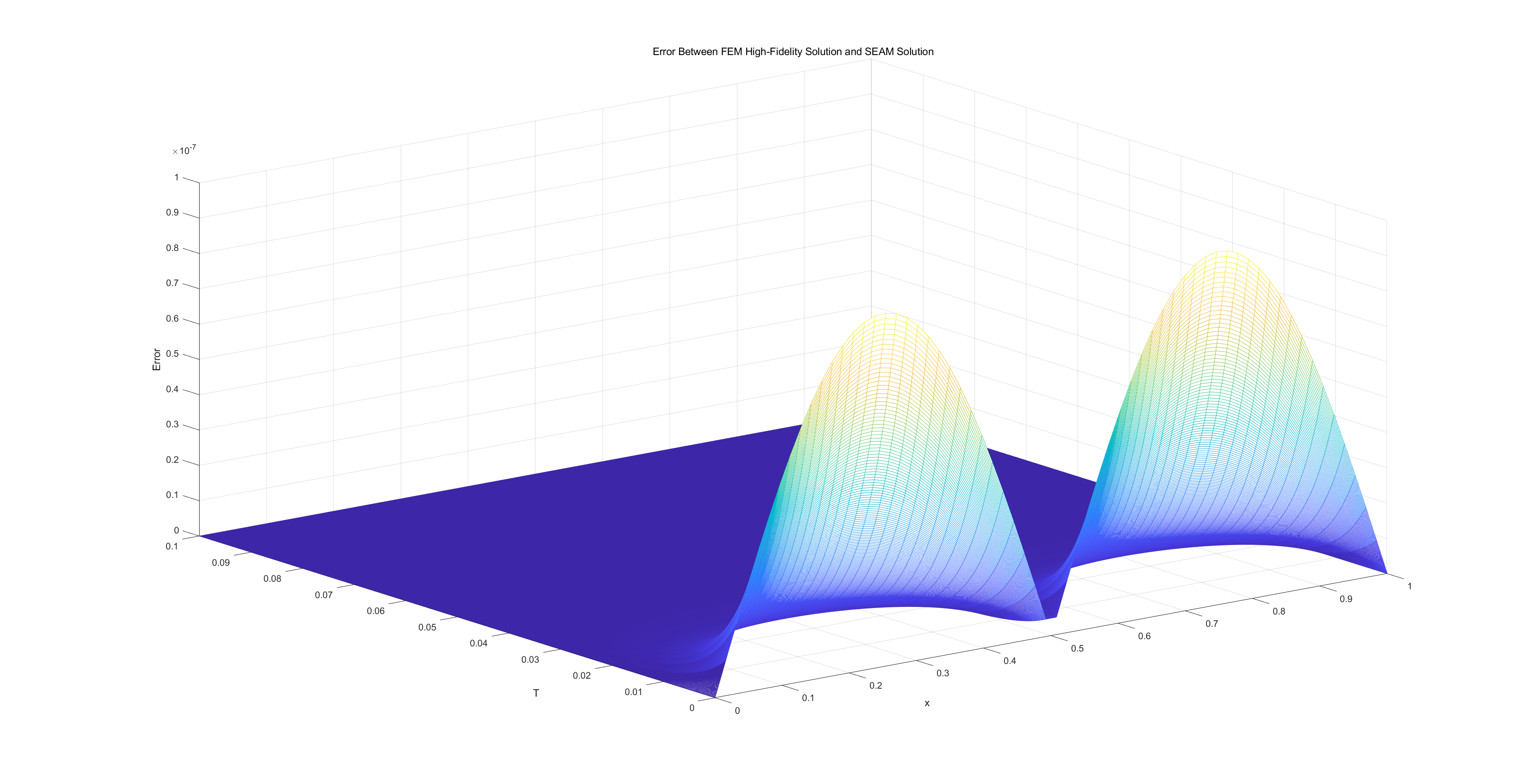}}
    \caption{ The first 20   eigenvalues of $U^\top U$ and relative error between  SEAM solution and high-fidelity solution   of (\ref{nml1}) }\label{fig2}
\end{figure}

\begin{table}[H]\label{tab1}
    \centering
    \begin{tabular}{lr|lr}
    \hline
    \multicolumn{2}{l}{High-fidelity model}  &\multicolumn{2}{l}{SEAM}\\
    \hline
    \multirow{2}{*}{\makecell[l]{Number of FEM d.o.fs \\ (each time step)}}& \multirow{2}{*}{ 98}     &   \multirow{2}{*}{\makecell[l]{Number of SEAM d.o.fs  \\ (each time step)}} & \multirow{2}{*}{1}\\
    &&&\\
     \hline
             &      &     d.o.fs reduction & 98:1\\ 
     \hline
    FE solution time &   1.73s   &     SEAM solution time  &  762ms\\
    \hline
    \multicolumn{2}{l}{SEAM-relative error in $L^{2}$-norm} &  \multicolumn{2}{l}{1.2e{-8}}\\
    \hline
    \end{tabular}
    \caption{Computational details for  the high-fidelity / SEAM solutions of (\ref{nml1}).}
     \label{1dcase1}
\end{table}

%
%
%

Based on the above observation from the  numerical experiment of 1d case,   we   investigate  SEAM   for d-dimensional  (d=2,3) second order parabolic problems   with  homogeneous Dirichlet boundary conditions. 
Under the assumption that the time step size is sufficiently small and  the time steps are not very large, we show 
%
    that the singular value distribution of the high-fidelity solution matrix  is very close to a reference matrix  of  rank one.
Then we select the eigenfunction associated to the principal eigenvalue as the POD basis  so as to obtain SEAM. 
 
%
%
 
The rest of the  paper is arranged as follows. Section 2   introduces the model problem and its full discretization. Section 3 gives the SEAM algorithm and its parallel version.  Section 4 discusses the singular value distribution of numerical solution matrix, 
 and Section 5 reports some numerical experiments.  Finally, Section 6 gives concluding remarks.

\section{Model problem and fully discrete scheme}
Let $D\subset \mathcal{R}^{d}$ ($d = 1, 2,3$)  be an open, bounded domain, and set $D_{T}:=D \times(0, T]$ for constant $T>0$.  We consider the following second-order parabolic problem: find $u=u(x,t)$ such that
\begin{equation}\label{modelproblem}
\begin{cases}
u_{t}- \nabla \cdot (\alpha(x)\nabla u)+c(x) u=f, & \text { in } D_{T}, \\ u=u_{0}, & \text { on } D \times\{0\}, \\ u=0, & \text { on } \partial D \times[0, T],
\end{cases}
\end{equation}
where $\alpha(x) = \left(\alpha_{k,j}(x)\right)_{d\times d}$ with $\alpha_{k,j}(x) = \alpha_{j,k}(x)  \in W^{1,\infty}(D)$,  $c(x) \in L^{\infty}(D)$,  and $f=f(x, t): D_{T} \rightarrow \mathcal{R}$ and $  u_{0}(x): D \rightarrow \mathcal{R}$ are given functions. 

Let   $\mathcal{T}_h=\bigcup\{T\}$ be a conforming shape-regular simplicial decomposition of the domain $D$, and let   $\mathbb{V}_h \subset H_0^1(D)$ be a space of standard finite elements of   arbitrary order  with respect to $\mathcal{T}_h$.  Denote by 
$ \left\{ x_{k}: \   k=1, 2, \cdots, M\right\} $
the set of  all  interior nodes    of $\mathcal{T}_h$, and by $\{\varphi_{k}(x):  \ k=1, 2, \cdots, M\}$   the nodal basis of   $\mathbb{V}_h$ with
$$     \varphi_{k}(x_j)=\delta_{kj} \quad  \text{for } k,j=1, 2, \cdots, M.$$
 Then we have 
$$
\mathbb{V}_h=\text{span} \left\{\varphi_{k}: \  k=1, \cdots, M\right\}.
$$

We divide the time interval $(0, T)$ into a uniform grid with nodes $t_n=n \tau$ and step size $\tau =T / N$ for $n = 0, 1, \ldots, N$. For convenience we set   $g^n:=g(t_n)$ for any   function $g(t)$. 

Let   $\pi_hu_{0}  \in \mathbb{V}_{h}$ be the piecewise   interpolation  of $u_{0}$. We consider the following fully discrete scheme for \eqref{modelproblem}: 

For $n=1,2, ..., N$, find $u_{n,h} \in \mathbb{V}_{h}$ such that
\begin{equation}\label{femscheme}
   \frac{1}{\tau}\left(u_{n,h}-u_{n-1,h}, v\right)_D+a \left(u_{n,h},   v\right)
   =\left( f_{n}, v\right)_D, \quad \forall v \in \mathbb{V}_{h},  
\end{equation}
with initial data $u_{0,h}  =\pi_h u^{0}$. Here  $(\cdot, \cdot)_D$ denotes the $L^2$-inner product on $D$, and  
$a \left(u,   v\right):=\left(\alpha \nabla u, \nabla v\right)_D + \left(c u_{n,h},v\right)_D $. 

Introduce   the numerical solution vector  $U_{n}$, the mass matrix  $\mathcal{M}$, the total stiffness matrix $\mathcal{S}$, and the load vector $F$, defined respectively by
\begin{equation*}\label{solution vector}
    U_{n} := \left(u_{n,h}(x_1), \cdots, u_{n,h}(x_M)\right)^\top, \quad n=0,1, \cdots, N, 
\end{equation*}
$$ 
\mathcal{M} := \left((\varphi_{k}, \varphi_{j})_D\right)_{M\times M},  \quad \mathcal{S}:=\left(a(\varphi_{k}, \varphi_{j})\right)_{M\times M}, \quad F:=\left(\left(f, \varphi_1\right), \cdots, \left(f, \varphi_M\right)\right)^{\top}.
$$
%
%
 Notice that   $\mathcal{M}$ and   $\mathcal{S}  $ are  both symmetric and positive definite.  
 Since   
$ 
    u_{n,h}(x)=  \sum\limits_{k=1}^{M} u_{n,h}(x_k)\varphi_{k}(x),
$ 
for a given $U_{0}$ the system (\ref{femscheme}) is of the  matrix form 
\begin{equation}\label{unwithfinq}
     (\mathcal{M}+\tau \mathcal{S}) U_{n}=  \mathcal{M}  U_{n-1} +  \tau F.
\end{equation}



\section{Single-Eigenvalue Acceleration Method}
In this section, we   give the SEAM algorithm, based on POD,  and a parallel version. 
Let $U_n$ be the high-fidelity  numerical solution vector defined by (\ref{unwithfinq}) for each $n$, and let \begin{equation}\label{fsu}
U:=\left(U_{0}, U_{1},\cdots, U_{n}\right),\quad n=0,1,\cdots,N,
\end{equation} 
be the numerical   solution matrix of (\ref{modelproblem}).  Denote
\begin{equation}\label{UTU}
X:=U^\top U,
\end{equation}
 and let  
$$  \lambda_0(X)\geq \lambda_{1}(X)\geq \cdots \geq   \lambda_{n}(X)\geq 0$$ be the eigenvalues of $X$, then the singular values of $U$ are 
$$\sqrt{\lambda_0(X)}, \sqrt{\lambda_1(X)}, \cdots, \sqrt{\lambda_{n}(X)}.$$  
 
 Set
$$
\widehat{\mathcal{U}}:=\operatorname{span}\left\{ U_{0}, U_{1}, \cdots, U_{n}\right\}, \quad \hat d := \operatorname{dim} \widehat{\mathcal{U}}.
$$
Without loss of generality, we assume that $1\leq \hat d <<M$.
 The POD method is to find the  standard orthogonal basis functions $\beta_{j}$ $(j=1,2, \cdots, p )$ of $ \widehat{\mathcal{U}}$
 such 
that for any $p $, $1 \leq$ $p \leq \hat d$, the mean square error
between $U_n$  $(n=0,1,\cdots)$   and its orthogonal projection onto the space $\text{span} \{\beta_1, \cdots, \beta_p\}$  is minimized on average, i.e.
 \begin{equation}\label{26}
    \min _{\left\{\beta_{j}\right\}_{j=1}^{p} \subset \widehat{\mathcal{U}}}
    \sum_{k=0}^{n}\left\|{U}_{k}-\sum_{j=1}^{p}\langle{U}_{k}, \beta_{j}\rangle \beta_{j}\right\|_{2}^{2}.
\end{equation}
 A solution $\left\{\beta_{j}\right\}_{j = 1}^{p}$ to (\ref{26}) is called a POD-basis of rank $p$. 

We will show later (cf. Theorem \ref{SEAMinhparabolicequation} and Corollary  \ref{corr1})  that the non-zero eigenvalues of $X$ except for the principal eigenvalue are close to zero  if   $n$ is not very large and the  time step size $\tau$ is small enough.      Based upon the   theoretical result,  we choose $p=1$ in the POD method \eqref{26}. This means that we only need to use the biggest eigenvalue $\lambda_0(X)$   and the associated eigenvector $\mathbf{b}_0=(b_{0},b_{1}, \cdots, b_{n})^{T}$  to obtain the rank 1 POD basis (cf. \cite{sirovich_1987_turbulence1})
$$
    \beta_{1}=\frac{1}{\sqrt{ \lambda_{0}}} \sum_{k=0}^{n} b_{k} {U}_{k}.
$$

 For convenience,  we name the POD method with $p=1$ as   Single Eigenvalue Acceleration Method (SEAM), and the corresponding SEAM scheme for the full discretization \eqref{unwithfinq} is as follows: 
 
 Given $ \alpha_0=\beta_1^\top U_0\in \Re,$ for  $k = 1, 2, \cdots, n,$ the SEAM approximation solution $U^{seam}_k$ at $k$-th time step is obtained by 
   \begin{equation}\label{RB-SEAM}
\left\{\begin{array}{l}
\beta_1^\top  (\mathcal{M}+\tau \mathcal{S}) \beta_1 \alpha_{k} =\beta_1^\top  \mathcal{M} \beta_1 \alpha _{k-1} + \tau  \beta_1^\top F, \\
U^{seam}_k= \alpha_{k}\beta_1.
\end{array}\right.
\end{equation}
Note that  the   coefficients $\beta_1^\top  (\mathcal{M}+\tau \mathcal{S}) \beta_1$ and $\beta_1^\top  \mathcal{M} \beta_1$  in \eqref{RB-SEAM}  are   positive  constants. 

By Theorem \ref{SEAMinhparabolicequation}   the model-order-reduction error of SEAM \eqref{RB-SEAM} is (cf. \cite{alfioquarteroni_2016_reduced})
$$
       \sum_{k=0}^{n}\left\|{U}_{k}-U^{seam}_k\right\|_{2}^{2}=\sum_{k=1}^{n} \lambda_k(X) \leq   C n^2\tau,
$$
where $C$ is a   positive constant  independent of $n$ and $\tau$.  
This means that the SEAM algorithm  may be inaccurate when $n^2\tau$  is not   small enough. In other words, we can not expect the desired accuracy if $n$ is large.   
To overcome such a weakness,  we will give a  parallel SEAM below (cf. Remark \ref{remark2}). 

For simplicity of notation,     we   assume    $N+1=(\tilde{n}+1)(n+1) $. 
We divide the total numerical solution matrix as 
$$\left(U_{0}, U_{1},\cdots, U_{N}\right)=\left(\mathfrak{U}_{0}, \mathfrak{U}_{1},\cdots, \mathfrak{U}_{\tilde{n}}\right), $$
with  the submatrices 
\begin{equation}\label{U_i}
\mathfrak{U}_{k} :=(U_{k(n+1)},U_{k(n+1)+1},\cdots,U_{k(n+1) + n}), \quad k =0, 1,  \cdots, \tilde{n}.
\end{equation}
 Then the   parallel SEAM can be described as following:
\begin{breakablealgorithm}
 \caption{Parallel SEAM}\label{seamalginhpe}
\begin{algorithmic}[1]
\For{$k = 0:\tilde{n}$}
     \State $\beta_{k} \gets POD(\mathfrak{U}_{k})$
\EndFor
\State $\alpha_{k,0}=\beta_k^\top U_{k(n+1)}, \quad   k = 0,1, \cdots,\tilde{n}$
\For{$j = 1:n$}
    \State $\beta_{k}^\top  (\mathcal{M}+\tau \mathcal{S}) \beta_{k} \alpha_{k,j} = \beta_k^\top  \mathcal{M} \beta_k \alpha _{k,j-1} + \tau  \beta_k^\top F, \quad   k = 0,1, \cdots,\tilde{n}$
\EndFor
\State \textbf{output } 
$\mathfrak{U}^{seam}_{k}=  \beta_{k} \left(\alpha_{k,0},\alpha_{k,1},\cdots, \alpha_{k,n}\right), \quad  k = 0,1, \cdots,\tilde{n}.$
\end{algorithmic}
\end{breakablealgorithm}

\section{Singular Value Distribution of Numerical Solution Matrix }

In this section, we   apply a perturbation technique to  investigate the singular value distribution of the numerical solution matrix $U$ in (\ref{fsu}). 
 We need the following  result due to Hoffman and Wielandt  \cite{doi:10.1142/9789812796936_0011}:
\begin{lemma}\label{whlemma}
If $A$ and $E$ are two $(n+1) \times (n+1)$ symmetric matrices, then there hold
$$
\sum_{k=0}^{n}\left(\lambda_{k}(A+E)-\lambda_{k}(A)\right)^{2} \leq\|E\|_{\mathfrak{F}}^{2} 
$$
and 
$$
\lambda_{n}(E) \leq \lambda_{k}(A+E)-\lambda_{k}(A) \leq \lambda_{0}(E) , \quad 0\leq k\leq n,
$$
where $\lambda_n(E)\leq \lambda_{n-1}(E)\leq \cdots \leq  \lambda_0(E)$ are the eigenvalues of $E$, 
and    $\|E\|_{\mathfrak{F}}
=\sqrt{\sum\limits_{k=0}^{n} \sum\limits_{j=0}^{n} E_{k j}^{2}}$ is   the Frobenius norm of $E$. 
\end{lemma}

 Let us turn back to the matrix equation \eqref{unwithfinq}. Since 
$$\mathcal{M}+\tau \mathcal{S}=\mathcal{M}^{1/2}(I+\tau\mathcal{M}^{-1/2}\mathcal{S} \mathcal{M}^{-1/2})\mathcal{M}^{1/2},$$
we   rewrite \eqref{unwithfinq} as 
\begin{equation}\label{unwithfinq0}
     U_{n}=
     \mathcal{M}^{-1/2} \left(I + \tau\bm{\mathcal{A}}\right)^{-1} \mathcal{M}^{1/2} U_{n-1} +  \tau\mathcal{M}^{-1/2} \left(I + \tau\bm{\mathcal{A}}\right)^{-1} \mathcal{M}^{-1/2}  F
\end{equation}
with $\bm{\mathcal{A}} = \mathcal{M}^{-1/2}\mathcal{S} \mathcal{M}^{-1/2}$, for a given $U_0$ and  $n=1,2, ..., N$.  

In what follows we  assume that the time step $\tau \in (0,1)$ is small enough such that 
\begin{equation}\label{time-step}
   \max\{ \tau\|\bm{\mathcal{A}}\|_2, \|I - \tau  \bm{\mathcal{A}}\|_2\}<1.
 \end{equation}
By noticing that $\bm{\mathcal{A}} $ is SPD, this assumption is reasonable. 

  Let $\rho(\cdot) $ represent  the spectral radius of a matrix. As \eqref{time-step} implies $\rho(\tau  \bm{\mathcal{A}}) < 1$, we have 
\begin{equation}\label{Matrixseries}
\left(I + \tau  \bm{\mathcal{A}}\right)^{-1} = \sum_{k = 0}^{\infty}(-1)^{k}(\tau  \bm{\mathcal{A}})^{k} = I - \tau  \bm{\mathcal{A}}    + O\left(\tau^{2}\right).
\end{equation}

In the   case of $F = 0$,    (\ref{unwithfinq0}) is  reduced to 
\begin{equation}\label{Un-f=0}
\left\{\begin{array}{l}
\tilde U_0=U_0,\\
\widetilde U_{n}  
 =  \mathcal{M}^{-1/2} \left(I + \tau\bm{\mathcal{A}}\right)^{-1} \mathcal{M}^{1/2} \widetilde U_{n-1}, \quad   1 \leq n\leq N.
 \end{array}
 \right.
\end{equation}
Inspired by \eqref{Matrixseries}, 
for $n=0,1,2,\cdots$ we define $U_{n}^{*}$ by
\begin{equation}\label{uapp1}
\left\{\begin{array}{l}
U_{0}^{*}=U_{0},\\
U_{n}^{*} = \mathcal{M}^{-1/2}\left(I - \tau  \bm{\mathcal{A}} \right)\mathcal{M}^{1/2} U_{*}^{n-1}=\mathcal{M}^{-1/2}\left(I - \tau  \bm{\mathcal{A}} \right)^n\mathcal{M}^{1/2}U_{0}^{*},
\end{array}
 \right.
 \end{equation}
and define  $\bar{U}_{n}^{*}$  by 
\begin{equation}\label{uapp2}
\left\{\begin{array}{l}
\bar{U}_{0}^{*}=U_{0},\\
\bar{U}_{n}^{*} =\mathcal{M}^{-1/2}\left(I - \tau  \|\bm{\mathcal{A}}\|_{2} I \right)\mathcal{M}^{1/2} \bar{U}_{*}^{n-1}= \left(1 - \tau \|\bm{\mathcal{A}}\|_{2}   \right)^n\bar{U}_{0}^{*}. 
\end{array}
 \right.
 \end{equation}
Denote 
$$
\widetilde U := (\widetilde U_{0},\widetilde U_{1},\cdots, \widetilde U_{n}), \quad U_{*} := \left(  U_0^{*}, {U}_{1}^{*},  \cdots,  {U}_{n}^{*}\right),
\quad \bar{U}_{*} := \left(\bar{U}_{0}^{*},\bar{U}_{1}^{*}, ,\cdots, \bar{U}_{n}^{*}\right) ,
 $$ 
 $$
\widetilde{X}: = \widetilde{U}^{\top}\widetilde{U}, \quad  \bar{X}_{*} := \bar{U}_{*}^{\top}\bar{U}_{*}.
$$
\begin{remark}\label{remk4.1}
We easily see that the rank of  $\bar{U}_{*}$ is at most 1, which implies that  $\bar{U}_{*}$ has at most one non-zero singular value, and thus  $\bar{X}_{*}$ has at most one nonzero eigenvalue. 
\end{remark}

In the sequel   we  use $a \lesssim b$ to denote  $a \le C  b$ for simplicity, where  $C$ is a generic  positive constant  independent of $n$ and $\tau$ and may   be  different at its each occurrence.

The following lemma shows that $\bar{X}_{*}$ can be viewed as a perturbation of 
$\widetilde{X}$ in some sense.

\begin{lemma}\label{SEAMsecondorderparabolicequation}
Assume that 
$\tau$ is small enough such that   \eqref{time-step} holds, then   we have
\begin{equation}\label{eigen-X}
\begin{aligned}
\|\widetilde X-\bar{X}_{*}\|_{\mathfrak{F}}^{2} 
\lesssim   n^4\tau^2.
\end{aligned}
\end{equation} 
\end{lemma}
 
 \begin{proof}
From  \eqref{Un-f=0}-\eqref{uapp2} we   get
\begin{eqnarray*}\label{T_*}
\widetilde T_*:=\widetilde U - U_{*} = \left(\bm{0},\tilde \theta_{1}U_{0},\tilde\theta_{2}U_{0}\cdots,\tilde\theta_{n}U_{0}\right),\\
\bar T_*:=U_{*} - \bar{U}_{*} = \left(\bm{0},\bar{\theta}_{1}U_{0},\bar{\theta}_{2}U_{0},\cdots,\bar{\theta}_{n}U_{0}\right),\label{barT_*}
\end{eqnarray*}
where \begin{align*}
\tilde\theta_{n} &:=\mathcal{M}^{-1/2}\left[\left(I - \tau  \bm{\mathcal{A}}    +O\left(\tau^{2}\right) \right)^n-\left(I - \tau  \bm{\mathcal{A}}  \right)^n\right]\mathcal{M}^{1/2}\\
&=
\mathcal{M}^{-1/2}\left[\sum\limits_{k = 1}^{n}\binom{n}{k} \left(I - \tau  \bm{\mathcal{A}}  \right)^{n-k} O\left(\tau^{2k}\right)\right]\mathcal{M}^{1/2},\\
 \bar{\theta}_{n} &:=\mathcal{M}^{-1/2}\left[\left(I - \tau  \bm{\mathcal{A}} \right)^n
- \left(I - \tau \|\bm{\mathcal{A}}\|_{2} I   \right)^n\right]\mathcal{M}^{1/2}\\
&=\tau \mathcal{M}^{-1/2}\left[ \left(\|\bm{\mathcal{A}}\|_{2} I-\bm{\mathcal{A}}\right)\sum\limits_{k = 1}^{n}  \left(1 - \tau\|\bm{\mathcal{A}}\|_{2}     \right)^{k-1}\left(I - \tau\bm{\mathcal{A}}\right)^{n-k} \right]\mathcal{M}^{1/2} . 
\end{align*}
Thus, 
\begin{eqnarray}\label{E*bar}
\widetilde X -\bar{X}_{*}&=&(\widetilde U^{\top}\widetilde U -U_{*}^{\top}U_{*})+(U_{*}^{\top}U_{*} -\bar{U}_{*}^{\top}\bar{U}_{*}) \nonumber \\
&=&\left((U_{*} + \widetilde T_{*})^{\top}(U_{*} + \widetilde T_{*}) -U_{*}^{\top}U_{*}\right)+\left((\bar{U}_{*} + \bar{T}_{*})^{\top}(\bar{U}_{*} + \bar{T}_{*})-\bar{U}_{*}^{\top}\bar{U}_{*}\right)
\nonumber\\
&= & (U_{*}^{\top}\widetilde T_{*} + \widetilde T_{*}^{\top}U_{*} + \widetilde T_{*}^{\top}\widetilde T_{*})+( \bar{U}_{*}^{\top}\bar{T}_{*} + \bar{T}_{*}^{\top}\bar{U}_{*} + \bar{T}_{*}^{\top}\bar{T}_{*})\nonumber\\
&=:& E_*+\bar{E}_{*}.
\end{eqnarray}
From the assumption \eqref{time-step} and the definitions of $U_{*}$,  $\bar{U}_{*}$,  $\widetilde T_{*}$ and  $\bar{T}_{*}$,  we have the following estimates:
\begin{eqnarray*}
\|U_{*}\|_{\mathfrak{F}}^{2} &\le& \|\mathcal{M}^{-1/2}\|_{\mathfrak{F}}^2\sum_{k = 0}^{n}\|I - \tau  \bm{\mathcal{A}}\|_{2}^{2k}\|\mathcal{M}^{1/2}\|_2^2\|U_{0}\|_{2}^{2} \\
&\le&(n+1) \|\mathcal{M}^{-1/2}\|_{\mathfrak{F}}^2 \|\mathcal{M}^{1/2}\|_2^2\|U_{0}\|_{2}^{2},\\
\|\bar{U}_{*}\|_{\mathfrak{F}}^{2} &\le& \sum_{k = 0}^{n}(1 - \tau \|\bm{\mathcal{A}}\|_{2} )^{2k}\|U_{0}\|_{2}^{2} \le  (n+1)
\|U_{0}\|_{2}^{2},\\
\|\widetilde T_{*}\|_{\mathfrak{F}}^{2} &\le& \|\mathcal{M}^{-1/2}\|_{\mathfrak{F}}^2   \sum_{k = 1}^{n} \left[\sum\limits_{j = 1}^{k}\binom{k}{j} \| I - \tau  \bm{\mathcal{A}} \|_2 ^{k-j} O\left(\tau^{2j}\right)\right]^2\|\mathcal{M}^{1/2}\|_2^2\|U_{0}\|_{2}^{2}\\
 &\le&\|\mathcal{M}^{-1/2}\|_{\mathfrak{F}}^2 \sum_{k = 1}^{n} \left[\sum\limits_{j = 1}^{k}\binom{k}{j}  O\left(\tau^{2j}\right)\right]^2\|\mathcal{M}^{1/2}\|_2^2\|U_{0}\|_{2}^{2}\\
 &\le&\|\mathcal{M}^{-1/2}\|_{\mathfrak{F}}^2 \sum_{k = 1}^{n} \left[\sum\limits_{j = 1}^{k} O\left(\tau^{j}\right)\right]^2\|\mathcal{M}^{1/2}\|_2^2\|U_{0}\|_{2}^{2}\\
&\lesssim&    n\tau^2 \|\mathcal{M}^{-1/2}\|_{\mathfrak{F}}^2\|\mathcal{M}^{1/2}\|_2^2 \|U_{0}\|_{2}^{2},
\end{eqnarray*}
and
\begin{equation*}
\begin{aligned}
\|\bar{T}_{*}\|_{\mathfrak{F}}^{2} & \leq \tau^2 \|\mathcal{M}^{-1/2}\|_{\mathfrak{F}}^2\|\|\bm{\mathcal{A}}\|_{2} I - \bm{\mathcal{A}}\|_{2}^{2}\sum_{k = 1}^{n}\left[\sum\limits_{j = 1}^{k}  \left(1 -\tau  \|\bm{\mathcal{A}}\|_{2}    \right)^{j-1}\|I - \tau\bm{\mathcal{A}}\|^{k-j} \right]^2\|\mathcal{M}^{1/2}\|_2^2 \|U_{0}\|_{2}^{2} 
\\
& \leq \tau^2 \|\mathcal{M}^{-1/2}\|_{\mathfrak{F}}^2\|\|\bm{\mathcal{A}}\|_{2} I - \bm{\mathcal{A}}\|_{2}^{2}(\sum_{k = 1}^{n}k^{2} )\|\mathcal{M}^{1/2}\|_2^2 \|U_{0}\|_{2}^{2}
\\
& \lesssim  n^3 \tau^2  \|\mathcal{M}^{-1/2}\|_{\mathfrak{F}}^2\|\|\bm{\mathcal{A}}\|_{2} I - \bm{\mathcal{A}}\|_{2}^{2} \|\mathcal{M}^{1/2}\|_2^2 \|U_{0}\|_{2}^{2},
\end{aligned}
\end{equation*}
 Consequently, we obtain
\begin{align*}
\|E_{*}\|_{\mathfrak{F}}^{2} & =\|U_{*}^{\top}\widetilde T_{*} + \widetilde T_{*}^{\top}U_{*} + \widetilde T_{*}^{\top}\widetilde T_{*}\|_{\mathfrak{F}}^2\\
& \leq  8\|U_{*}\|_{\mathfrak{F}}^{2}\|\widetilde T_{*}\|_{\mathfrak{F}}^{2}  + 2\|\widetilde T_{*}\|_{\mathfrak{F}}^{4} \\ 
& \lesssim  n^2\tau^2  \|\mathcal{M}^{-1/2}\|_{\mathfrak{F}}^4 \|\mathcal{M}^{1/2}\|_2^4\|U_{0}\|_{2}^{4}\\
& \lesssim  n^2\tau^2,\\
 \|\bar{E}_{*}\|_{\mathfrak{F}}^{2} &=\|\bar{U}_{*}^{\top}\bar{T}_{*} + \bar{T}_{*}^{\top}\bar{U}_{*} + \bar{T}_{*}^{\top}\bar{T}_{*}\|_{\mathfrak{F}}^2 \\
 &\le  8\|\bar{U}_{*}\|_{\mathfrak{F}}^{2}\|\bar{T}_{*}\|_{\mathfrak{F}}^{2}  +2 \|\bar{T}_{*}\|_{\mathfrak{F}}^{4} \\ 
& \lesssim \left(n^4\tau^2  \|U_{0}\|_{2}^{2}+n^6\tau^4\|\mathcal{M}^{-1/2}\|_{\mathfrak{F}}^2\|\|\bm{\mathcal{A}}\|_{2} I - \bm{\mathcal{A}}\|_{2}^{2} \|\mathcal{M}^{1/2}\|_2^2 \|U_{0}\|_{2}^{2}\right)\\
&\quad \cdot
  \|\mathcal{M}^{-1/2}\|_{\mathfrak{F}}^2\|\|\bm{\mathcal{A}}\|_{2} I - \bm{\mathcal{A}}\|_{2}^{2} \|\mathcal{M}^{1/2}\|_2^2 \|U_{0}\|_{2}^{2}\\
  & \lesssim  n^4\tau^2,
\end{align*}
 which, together with \eqref{E*bar}, yields
 $$
 \|\widetilde X-\bar{X}_{*}\|_{\mathfrak{F}}^{2}\leq \|E_{*}\|_{\mathfrak{F}}^{2}+\|\bar{E}_{*}\|_{\mathfrak{F}}^{2}\lesssim n^2\tau^2+n^4\tau^2\lesssim n^4\tau^2.
$$
\end{proof}


Based on Lemma \ref{SEAMsecondorderparabolicequation},  
 we  can get an error estimate   between $ {X}=U^\top U$ $\bar X= \bar{U}^{\top}\bar {U}$ and :

\begin{lemma}\label{SEAMsecondorderparabolicequation2}
Assume that 
$\tau$ is small enough such that   \eqref{time-step} holds, then   we have
\begin{equation}\label{eigen-X2}
\begin{aligned}
\|X-\bar{X}_{*} \|_{\mathfrak{F}}^{2} 
\lesssim   n^4\tau^2.
\end{aligned}
\end{equation} 
\end{lemma}
\begin{proof}
From (\ref{unwithfinq0}) 
and \eqref{Un-f=0} it follows
 \begin{equation}\label{unfinq01}
  U_{n} -\widetilde U_n=\tau \mathcal{M}^{-1/2}\sum\limits_{k = 1}^{n} \left(I + \tau\bm{\mathcal{A}}\right)^{-k} \mathcal{M}^{-1/2}  F, \quad n=1,2, \cdots.
\end{equation}
Recalling that
$
    \widetilde U = (\widetilde U_{0},\widetilde U_{1},\cdots, \widetilde U_{n}) 
$
and 
$ U  = \left(  U_0 , {U}_{1} ,  \cdots,  {U}_{n} \right) ,$
  we have
$$  T_*:=U -\widetilde  U= \tau \left(  \bm{0} ,  \theta_{1}F ,  \cdots,    \theta_{n}F \right),$$
where 
\begin{align*}
 {\theta}_{n} &:=\mathcal{M}^{-1/2} \sum_{k = 1}^{n} \left(I + \tau  \bm{\mathcal{A}}    \right) ^{-k} \mathcal{M}^{-1/2}. 
\end{align*}
Then we obtain 
\begin{align*} 
{X}-\widetilde X&=(\widetilde U+  T_* )^{\top}(\widetilde U+  T_*)-\widetilde U^\top \widetilde U\\
&=\widetilde U^{\top}  T_* +   T_* ^{\top}\widetilde U+ T_*^{\top}  T_*.
\end{align*}
By \eqref{time-step} it is easy to show that  
$$ \|\widetilde U\|_{\mathfrak{F}}^2\lesssim n, \quad  \|   T_* \|_{\mathfrak{F}}^2\lesssim n\tau^2\| F \|_{\mathfrak{F}}^2.$$
As a result, we get
\begin{align} 
\|{X}-\widetilde X\|_{\mathfrak{F}}^2&\lesssim n^2\tau^2\| F \|_{\mathfrak{F}}^2+n^2\tau^4\| F \|_{\mathfrak{F}}^4\lesssim n^2\tau^2, 
\nonumber
\end{align}
which, together with the triangle inequality and Lemma \ref{SEAMsecondorderparabolicequation}, yields the desired estimate.
\end{proof}

Using Remark \ref{remk4.1} and Lemmas \ref{whlemma} and \ref{SEAMsecondorderparabolicequation}, we  immediately obtain the following main conclusion for the full discretization (\ref{unwithfinq0}): %
\begin{theorem}\label{SEAMinhparabolicequation}
Assume   that   $\tau$ is small enough such that   \eqref{time-step} holds, then  we have
\begin{equation}\label{xperum}
\begin{aligned}
 \left(\lambda_{0}( {X})-\lambda_{0}(\bar{X}_{*})\right)^2 +  \sum\limits_{ k=1}^{n} \lambda_{k}( {X})^2  \lesssim  n^4\tau^2 . 
\end{aligned}
\end{equation} 
\end{theorem}


This theorem directly yields the following result: 
\begin{corollary}\label{corr1}
Assume   that   $\tau$ is small enough such that   \eqref{time-step} holds,  and  that   $n = O(\tau^{-\frac{1}{2}+ \epsilon} )$ for some $\epsilon\in (0 ,\frac{1}{2})$,  then  we have
\begin{align}\label{eigen-X11}
&\left(\lambda_{0}(X)-\lambda_{0}(\bar{X}_{*})\right)^{2} + \sum_{k = 1}^{n}\lambda_{k}^{2}(X) \lesssim  \tau^{4\epsilon}. 
%
\end{align} 
\end{corollary}

\begin{remark}\label{remark1}
From   Corollary \ref{corr1}, we can see that each singular value of the high-fidelity solution matrix ${U}$ is a small perturbation of the corresponding singular value of the reference matrix $\bar{U}_{*}$, provided that $n = O(\tau^{-\frac{1}{2}+ \epsilon} )$. Since $\bar{U}_{*}$  has at most one non-zero singular value, all the singular values  of $U$, except for the largest one, are  then  close to zero. 
\end{remark}

\begin{remark}\label{remark2}
From the analysis we know that Corollary \ref{corr1} holds not only  for $U=(U_{0},U_{1},\cdots,U_{ n})$, but also for  $U =(U_{k},U_{k+1},\cdots,U_{k + n})$
 with $k=0,1, \cdots$. This motivates the parallel SEAM in Section 3.
%
%
\end{remark}

\section{Numerical Results}
This section provides some numerical experiments to  verify the performance of SEAM for the fully discrete scheme (\ref{unwithfinq0}) using continuous  linear finite elements.  All the tests are performed on a 12th-Gen Intel 3.20 GHz Core i9 computer.

\subsection{Two-dimensional tests}
We consider the following two-dimensional problem: 
\begin{equation}\label{nml2}
\left\{\begin{array}{ll}
\frac{\partial u}{\partial t}=\alpha_{1}(x,y)\frac{\partial^{2} u}{\partial x^{2}}+\alpha_{2}(x,y)\frac{\partial^{2} u}{\partial y^{2}} - c(x,y)u + f, & \text{ in } D\times (0,T], \\
u=u_{0}, & \text { on } D \times\{0\} ,\\
u = 0,&   \text { on } \partial D \times[0, T],
\end{array}\right.
\end{equation}
where   $D = (0,1)\times (0,1)$ and $T = 1$. To obtain  the high-fidelity numerical solutions,  we first divide   the spatial domain $D$  uniformly   into $32\times 32=1024$ squares, then divide each  square into  $2$  isosceles right triangles. For the temporal subdivision, the coefficients $\alpha_{i}$, the initial data $u_{0}(x,y)$ and the source term $f$, we consider three situations:
\begin{itemize}
\item[S1:]$\tau= 1\times 10^{-4}$, $\alpha_{1}(x,y) = \alpha_{2}(x,y) = c(x,y) = 1$, $u_{0}(x,y) = \sin (\pi xy)$, and $f=0, \  xy$;

\item [S2:]$\tau= 1\times 10^{-4}$,  $\alpha_{1}(x,y) = x^{2}$, $\alpha_{2}(x,y) = y^{2}$,  $c(x,y) = \pi^{2}(1 - 2x^{2}y^{2})$, $u_{0}(x,y) = \sin (\pi x)\sin(\pi y)$, and $f=0,\ 10$;

\item [S3:]$\tau= 2.5\times 10^{-3}$,  $\alpha_{1}(x,y) = x^{2}$, $\alpha_{2}(x,y) = y^{2}$,  $c(x,y) = \pi^{2}(1 - 2x^{2}y^{2})$, $u_{0}(x,y) = \sin (\pi x)\sin(\pi y)$, and $f=0$.
 \end{itemize}
 
In S1 and S2, we take $n = \tilde n=100$. In S3, we take $n = \tilde n = 20$. Figures \ref{2D-eigenvaluechanges1f0} to  \ref{2D-eigenvaluechanges3f0} show  the distribution of eigenvalues of  $\mathfrak{U}_{i}^\top \mathfrak{U}_{i}$ for different $f$ in situations  S1,S2 and S3, where  the numerical solution matrix $\mathfrak{U}_{i}$ is defined by \eqref{U_i} for $i=1,2,\cdots, \tilde n$. In each figure, the label 'No.s of time interval' means the $\tilde n$ time intervals, and the  label 'No.s of eigenvalues' means the first 5 eigenvalues in each time interval.

Figures \ref{seams1f0} to \ref{seams3f0} demonstrate the  high-fidelity numerical solutions and the corresponding SEAM solutions of (\ref{nml2}). 

Tables \ref{details-2D} gives some computational details as well as the relative errors between the  the  SEAM solution $u_{SEAM}$ and the high-fidelity numerical solution $u_{fem}$, with $$\text{SEAM-Error}_{L^{2}} := \left(\int_{D \times (0,T]}(u_{fem} - u_{SEAM})^{2}d{\bf x}dt\right)^{\frac{1}{2}}/ \left(\int_{D \times (0,T]} u_{fem}  ^{2}d{\bf x}dt\right)^{\frac{1}{2}}.$$

From Figures \ref{2D-eigenvaluechanges1f0} to  \ref{2D-eigenvaluechanges3f0}, we have the following observations on the distribution of eigenvalues:
\begin{itemize}
\item   In each case the principal eigenvalue of $\mathfrak{U}_{i}^\top \mathfrak{U}_{i}$  is much larger than other eigenvalues. For example, in the case S2 with $f=xy$, for the first 100 steps, the principal eigenvalue is 2373.15 and the second largest one is 0.000256;
\item  The eigenvalues of $\mathfrak{U}_{i}^\top \mathfrak{U}_{i}$ except the principal one are close to zero when $f=0$ and, though being not so close to zero, are still very small. These are   consistent with Corollary \ref{corr1};
\item  As $i$ increases from $1$ to $\tilde n$,   the principal eigenvalues of $\mathfrak{U}_{i}^\top \mathfrak{U}_{i}$ gradually decrease. For example, in the case S1 with $f=0$, the principal eigenvalues decrease from 3376.2 to 38.0. In the case S2 with $f=10$, the principal eigenvalues decrease from 2389.5 to 1258.0. And in the case S3, the principal eigenvalues decrease from 4495.9 to 328.3.

 \end{itemize}
 
 From     Figures \ref{seams1f0} to \ref{seams3f0} and Tables \ref{details-2D}, we have the following observations on the SEAM solutions:
\begin{itemize}
\item     SEAM gives accurate numerical solutions for   all cases;

\item  SEAM is much faster than the original high-fidelity numerical method, due to the remarkable reduction of the size of the discrete model from $M=961$ to $p=1$.
\end{itemize}

\begin{figure}[htbp]
    \centering
   \subfigure[The 1th-20th time intervals]{\includegraphics[width=5cm, height =3cm]{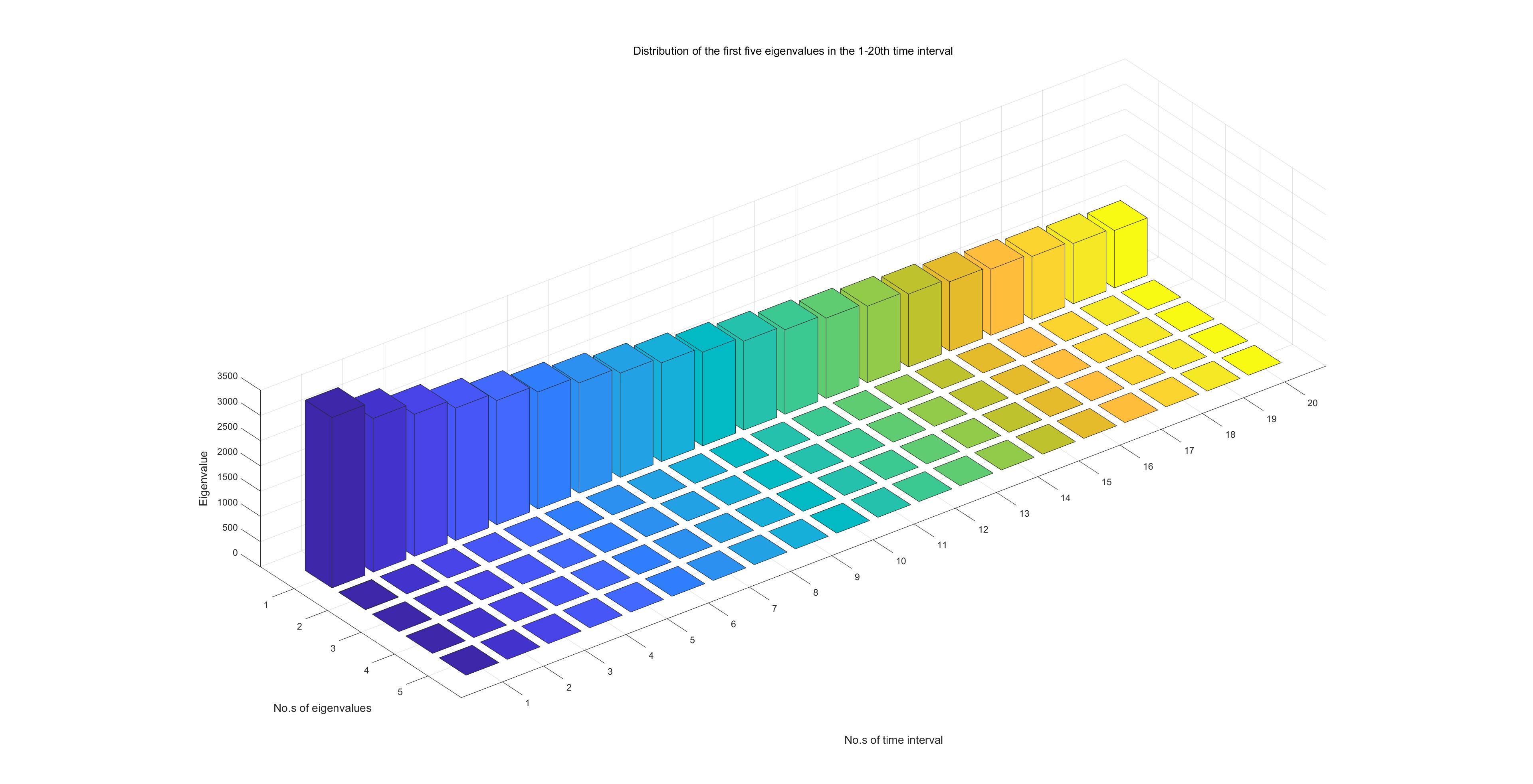}} 
   \subfigure[ The 21th-40th time intervals]{\includegraphics[width=5cm, height =3cm]{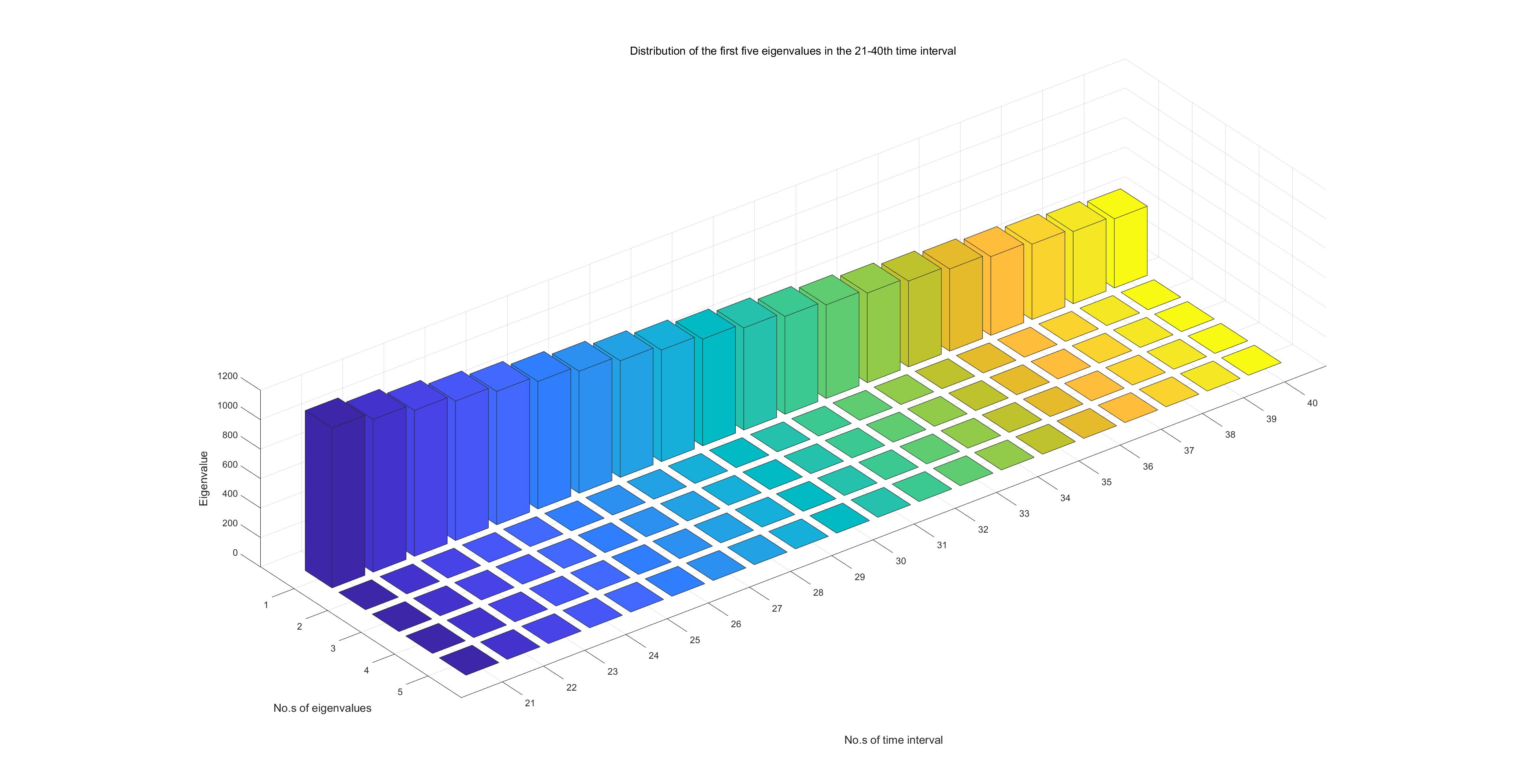}}

    \subfigure[  The 41th-60th time intervals]{\includegraphics[width=5cm, height =3cm]{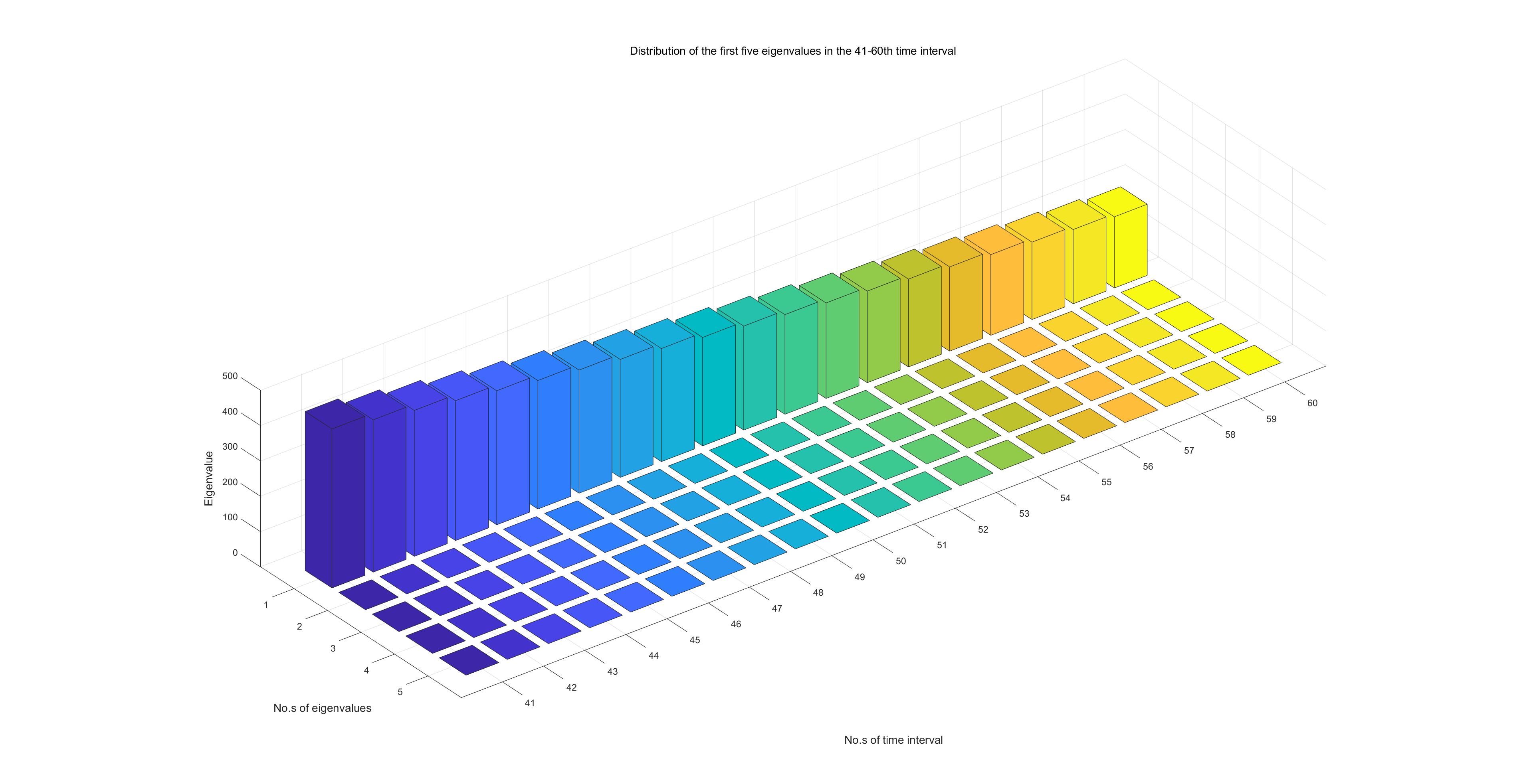}}
    \subfigure[  The 61th-80th time intervals]{\includegraphics[width=5cm, height =3cm]{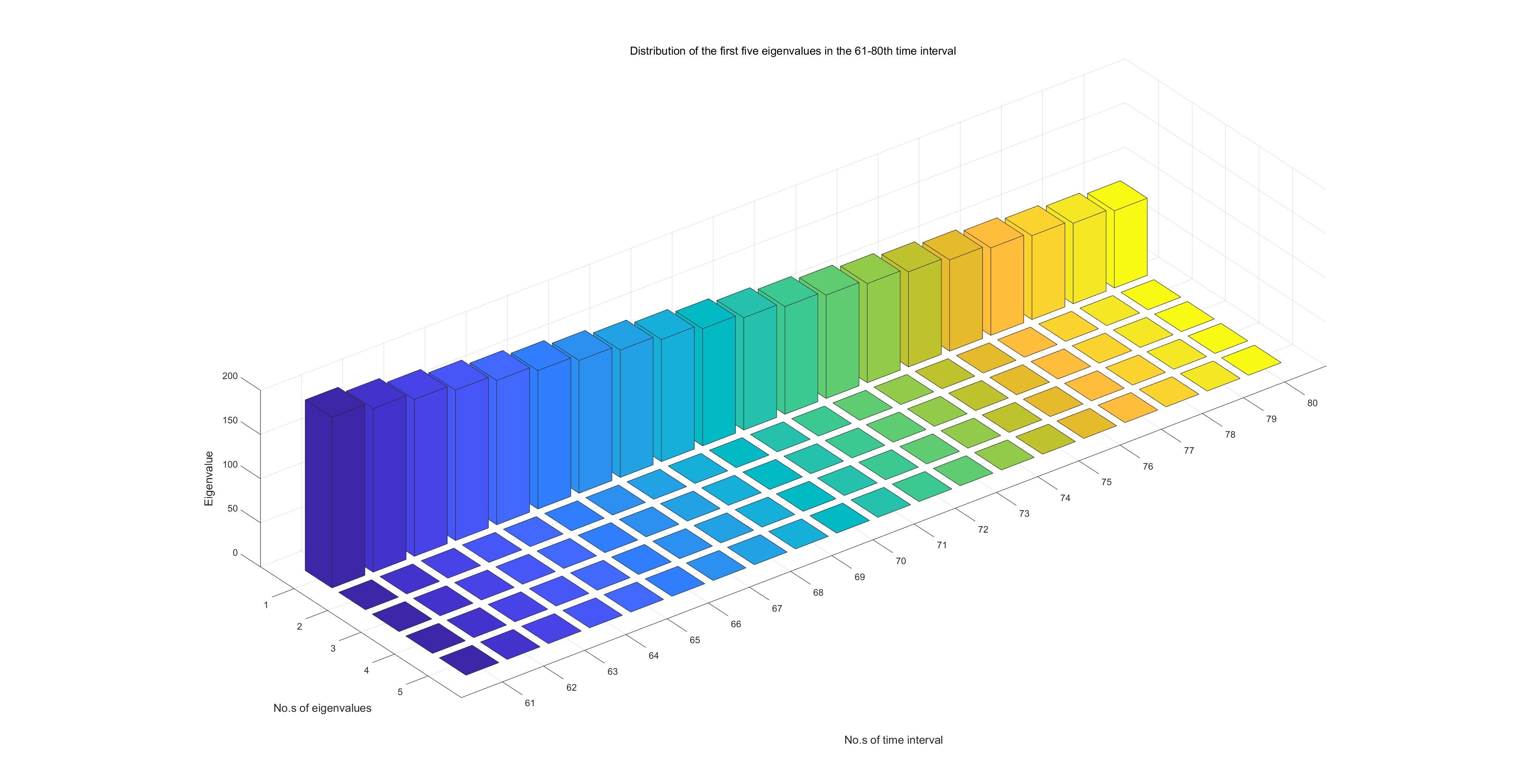}}

     \subfigure[  The 81th-100th time intervals]{\includegraphics[width=5cm, height =3cm]{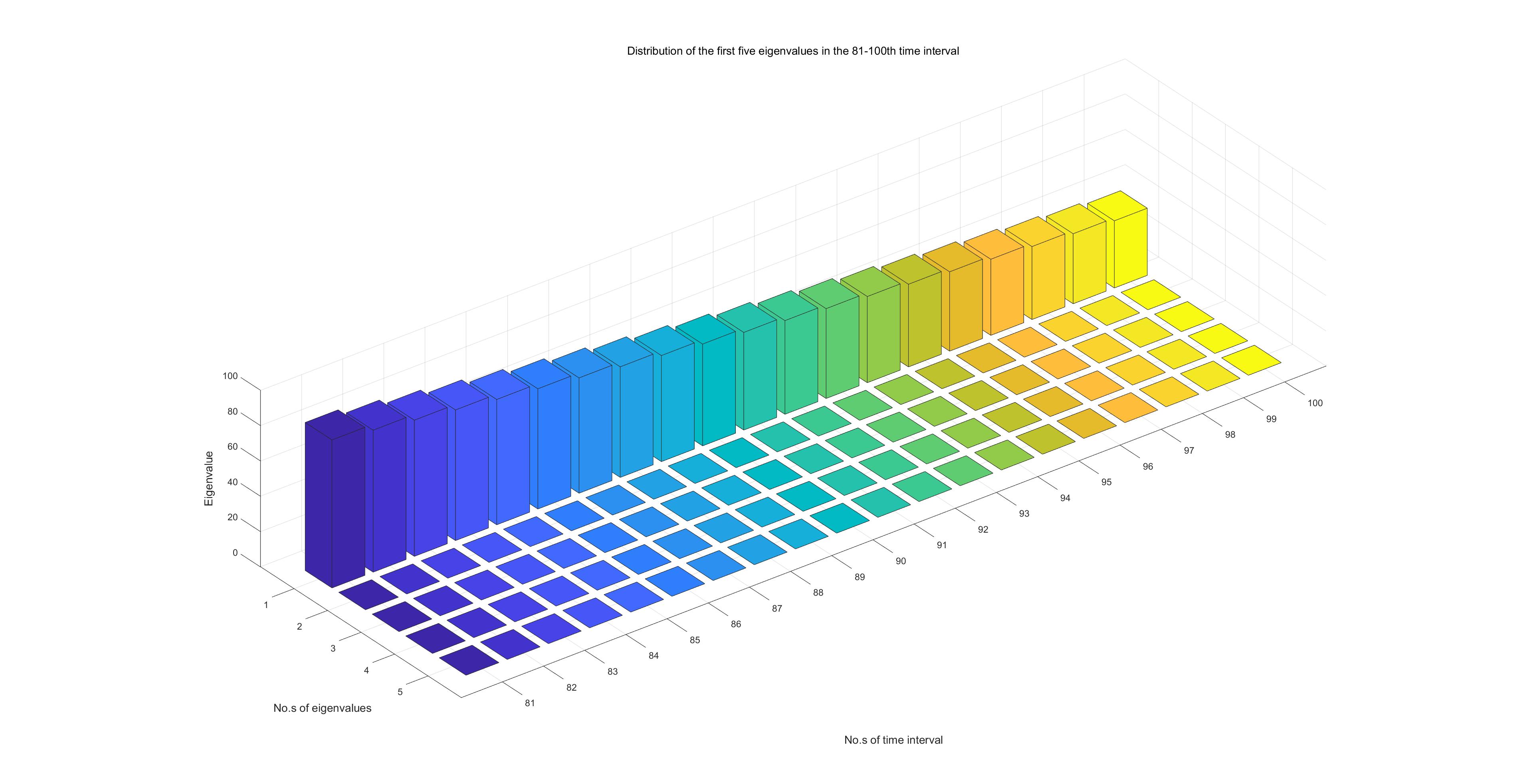}}
    \caption{The first 5 eigenvalues of $\mathfrak{U}_{i}^\top \mathfrak{U}_{i}$ (from big to small) for S1, $f=0$.}\label{2D-eigenvaluechanges1f0}
\end{figure}
\begin{figure}[htbp]
    \centering
   \subfigure[The 1th-20th time intervals]{\includegraphics[width=5cm, height =3cm]{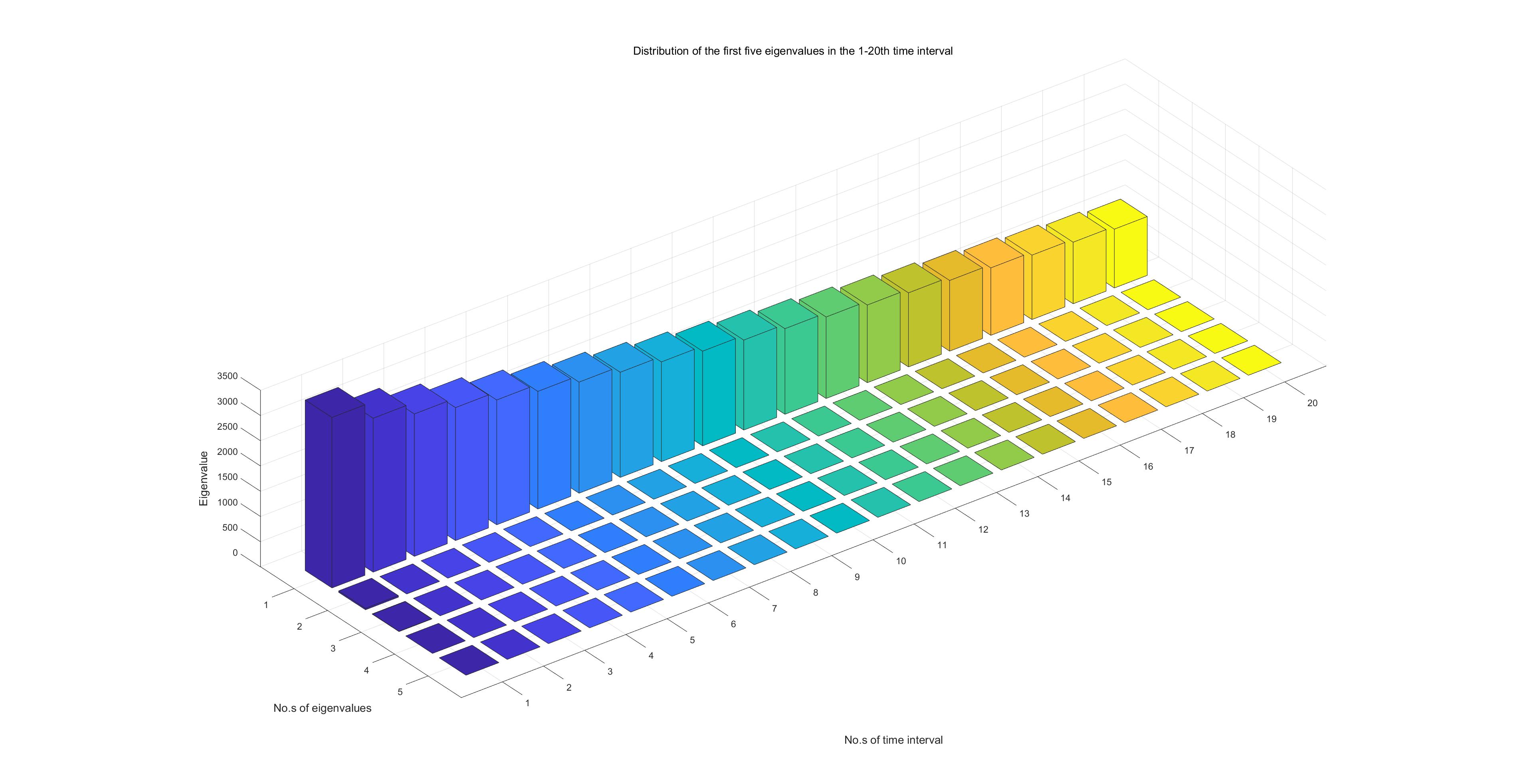}} 
   \subfigure[ The 21th-40th time intervals]{\includegraphics[width=5cm, height =3cm]{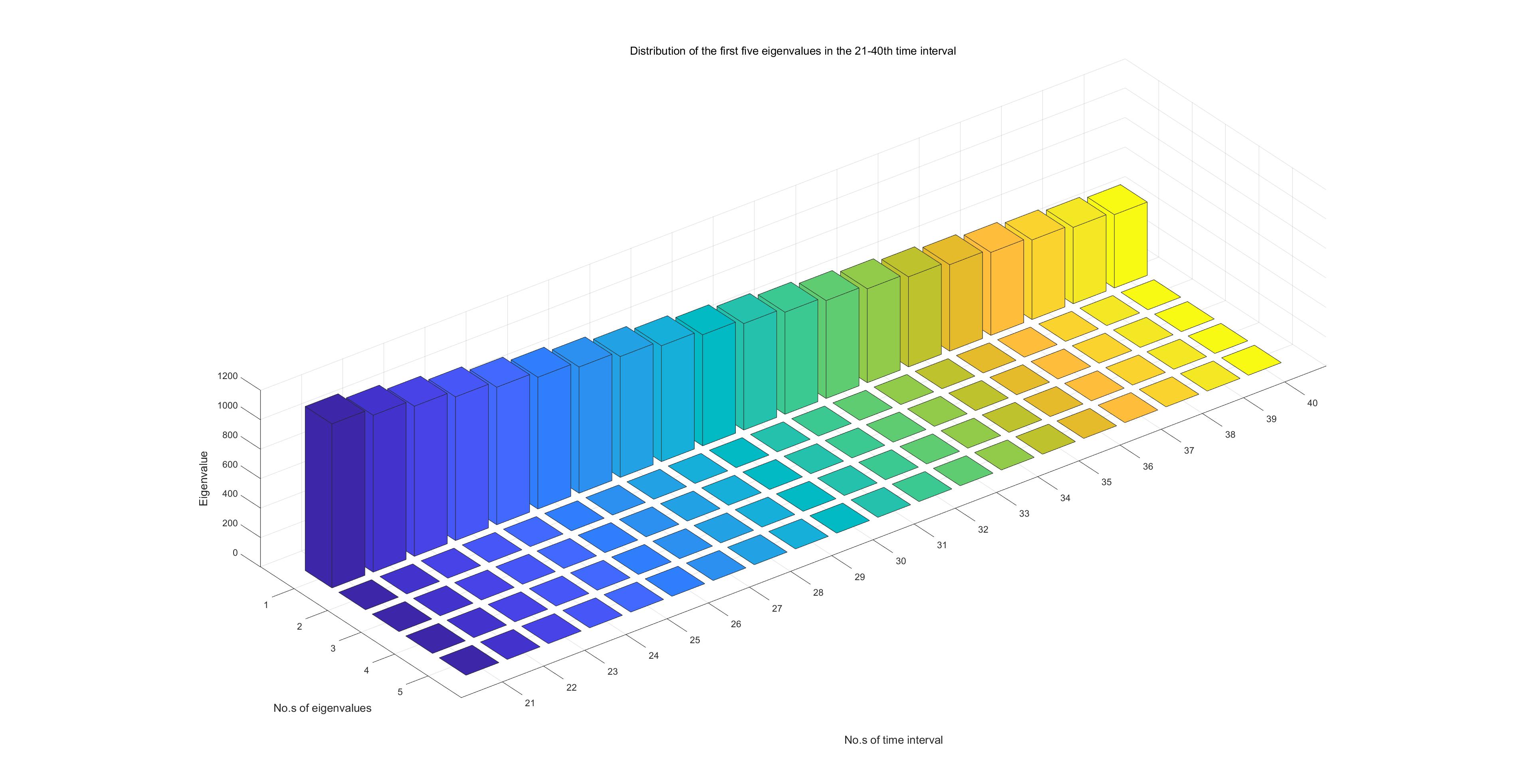}}

    \subfigure[  The 41th-60th time intervals]{\includegraphics[width=5cm, height =3cm]{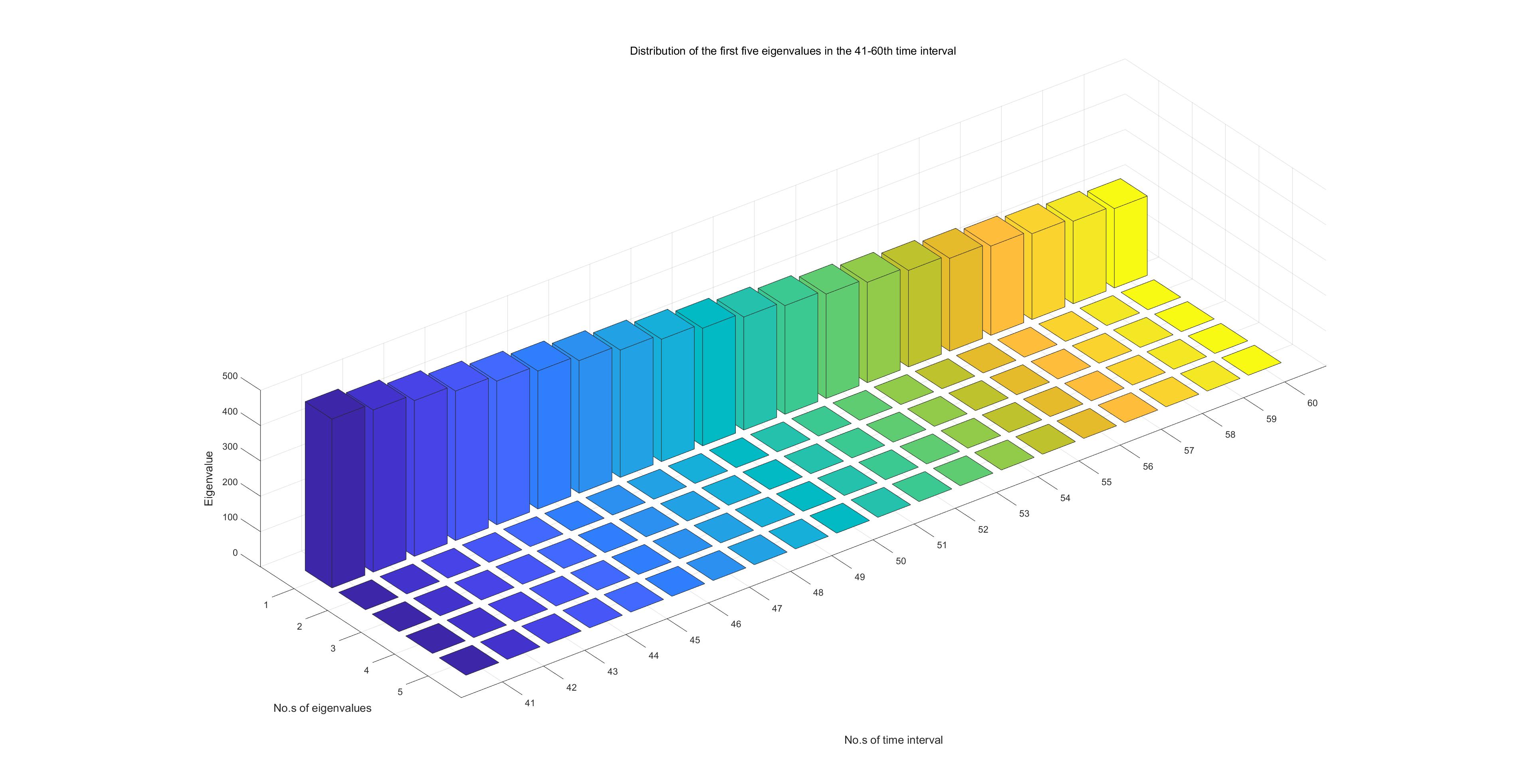}}
    \subfigure[  The 61th-80th time intervals]{\includegraphics[width=5cm, height =3cm]{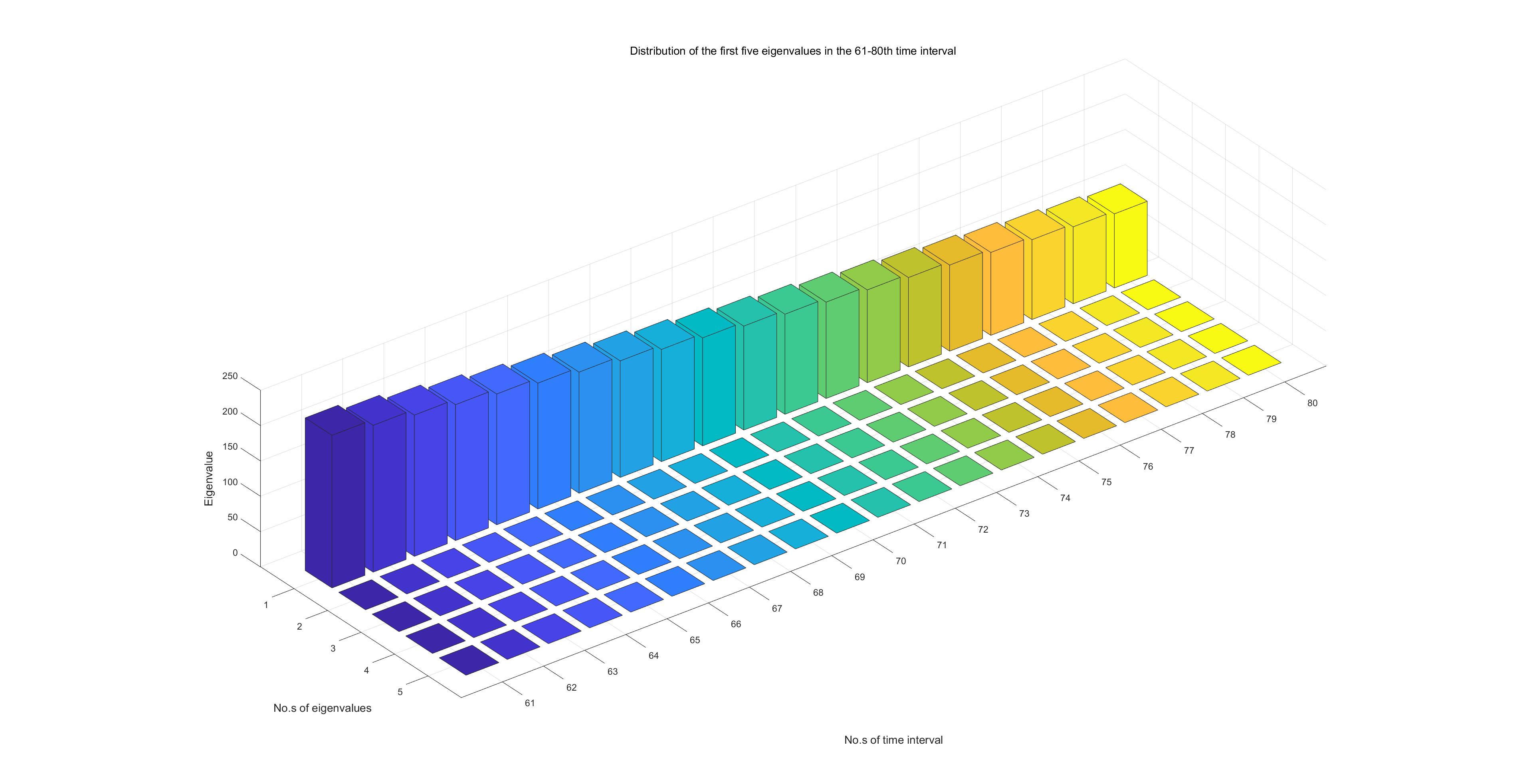}}

     \subfigure[  The 81th-100th time intervals]{\includegraphics[width=5cm, height =3cm]{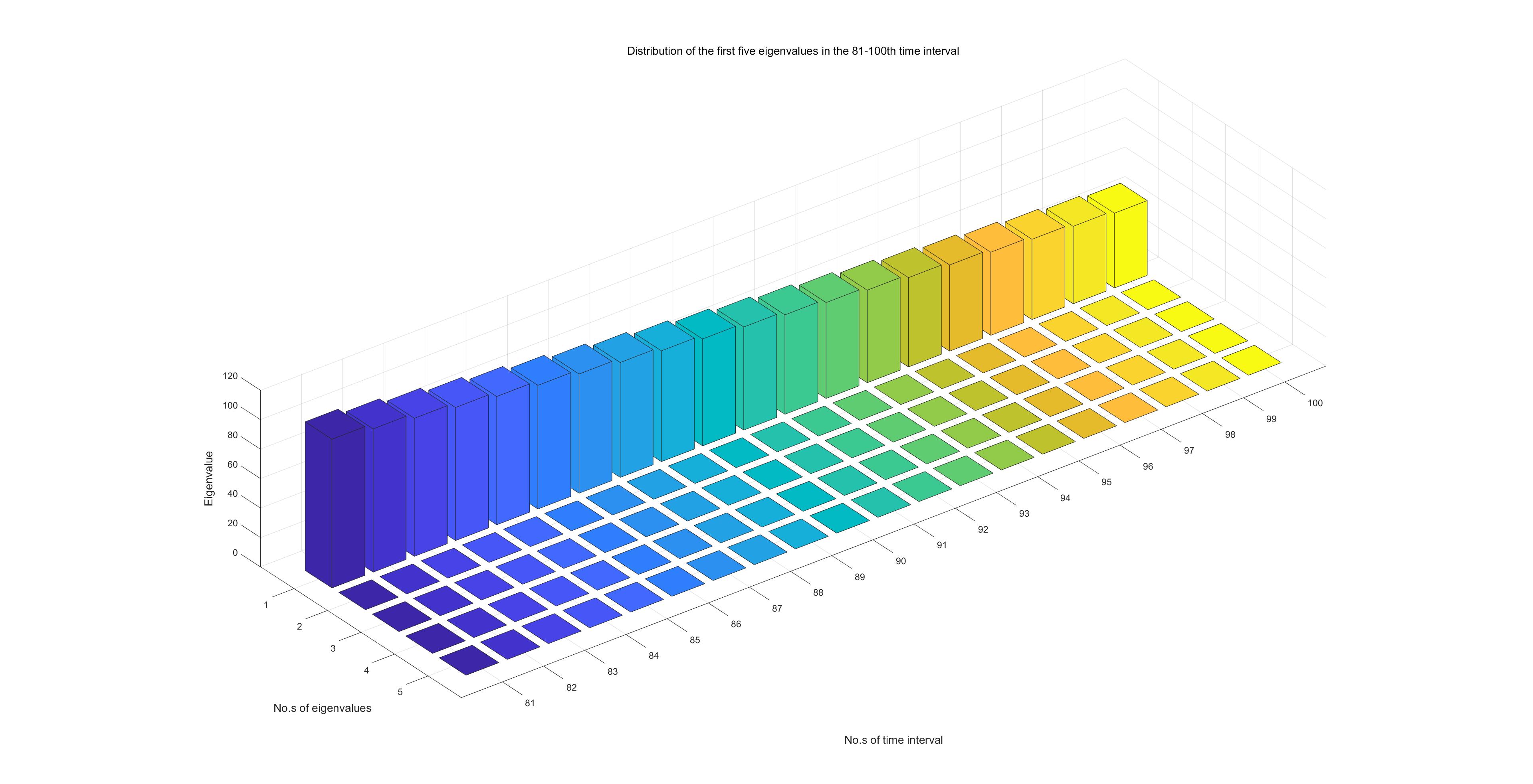}}
    \caption{The first 5 eigenvalues of $\mathfrak{U}_{i}^\top \mathfrak{U}_{i}$ (from big to small) for S1 with $f=xy$.}\label{2D-eigenvaluechanges1fxy}
\end{figure}
\begin{figure}[htbp]
    \centering
   \subfigure[The 1th-20th time intervals]{\includegraphics[width=5cm, height =3cm]{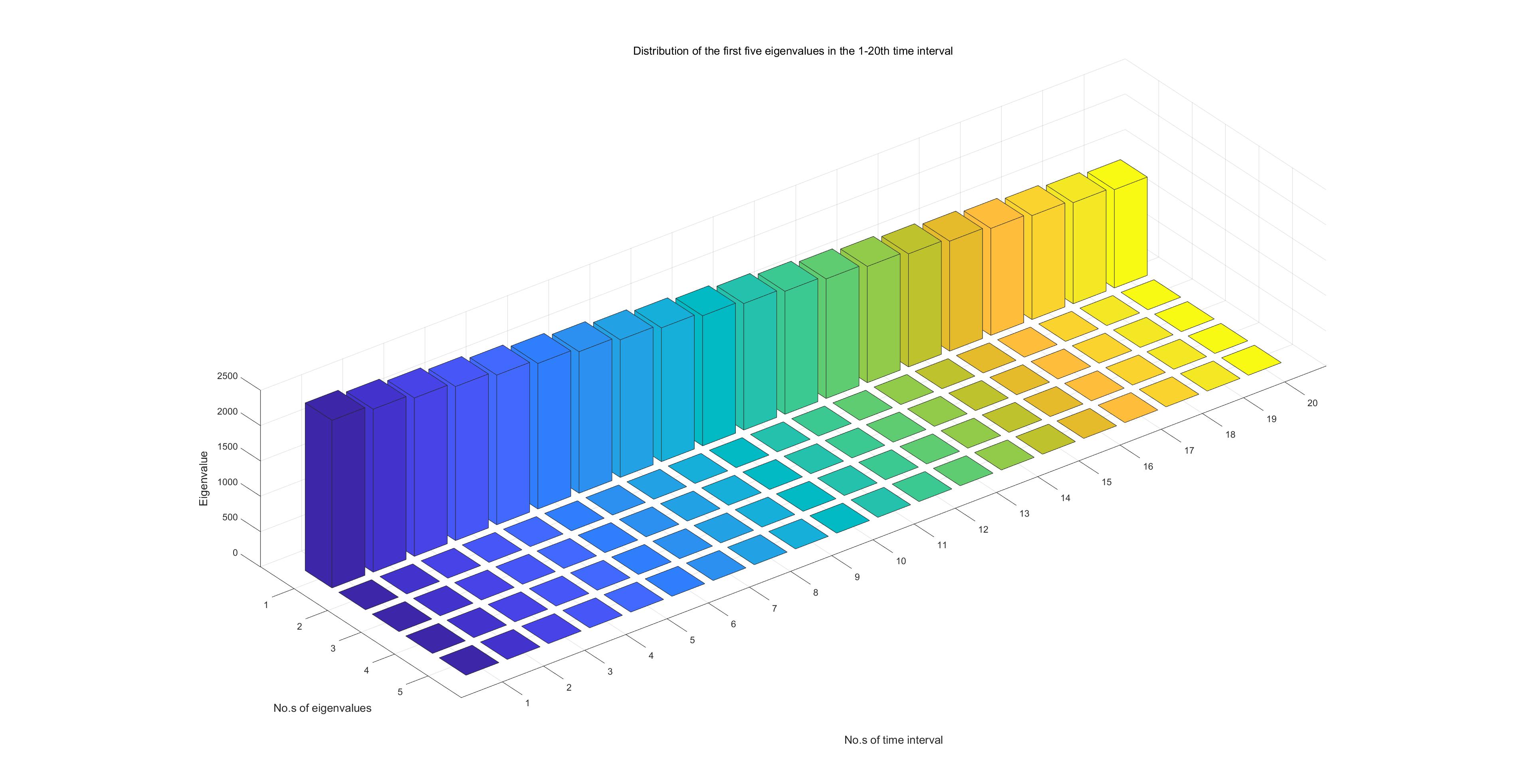}} 
   \subfigure[ The 21th-40th time intervals]{\includegraphics[width=5cm, height =3cm]{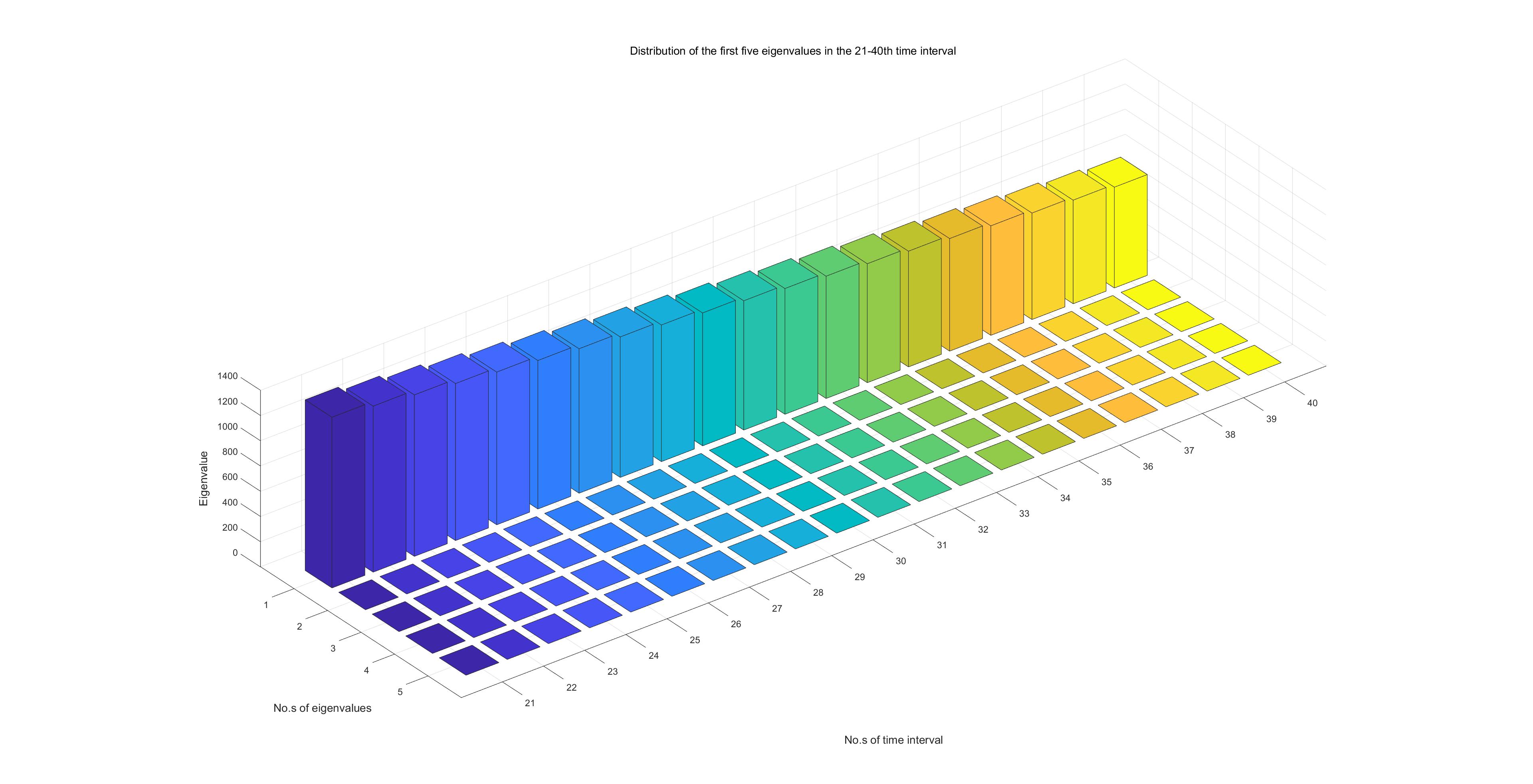}}

    \subfigure[  The 41th-60th time intervals]{\includegraphics[width=5cm, height =3cm]{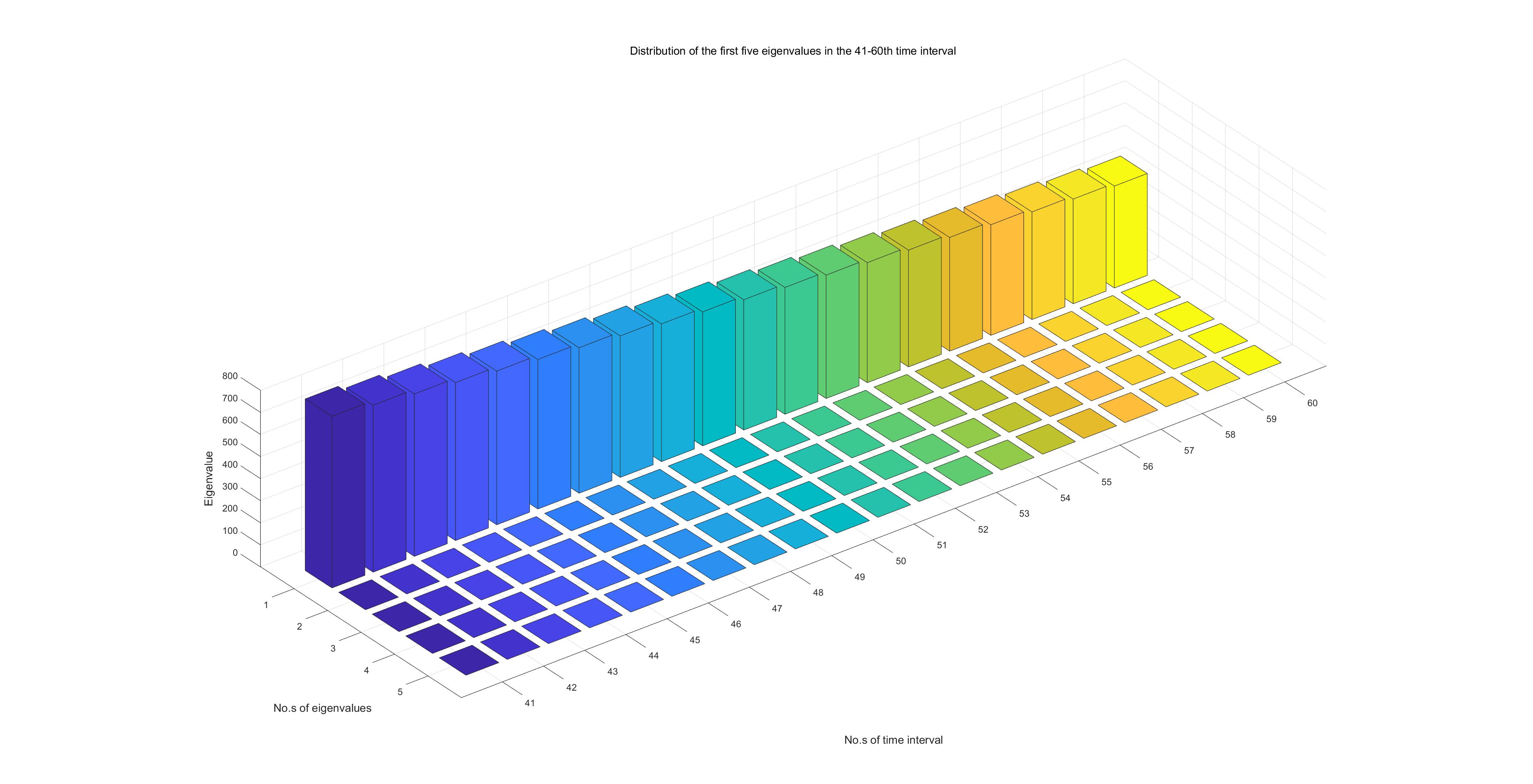}}
    \subfigure[  The 61th-80th time intervals]{\includegraphics[width=5cm, height =3cm]{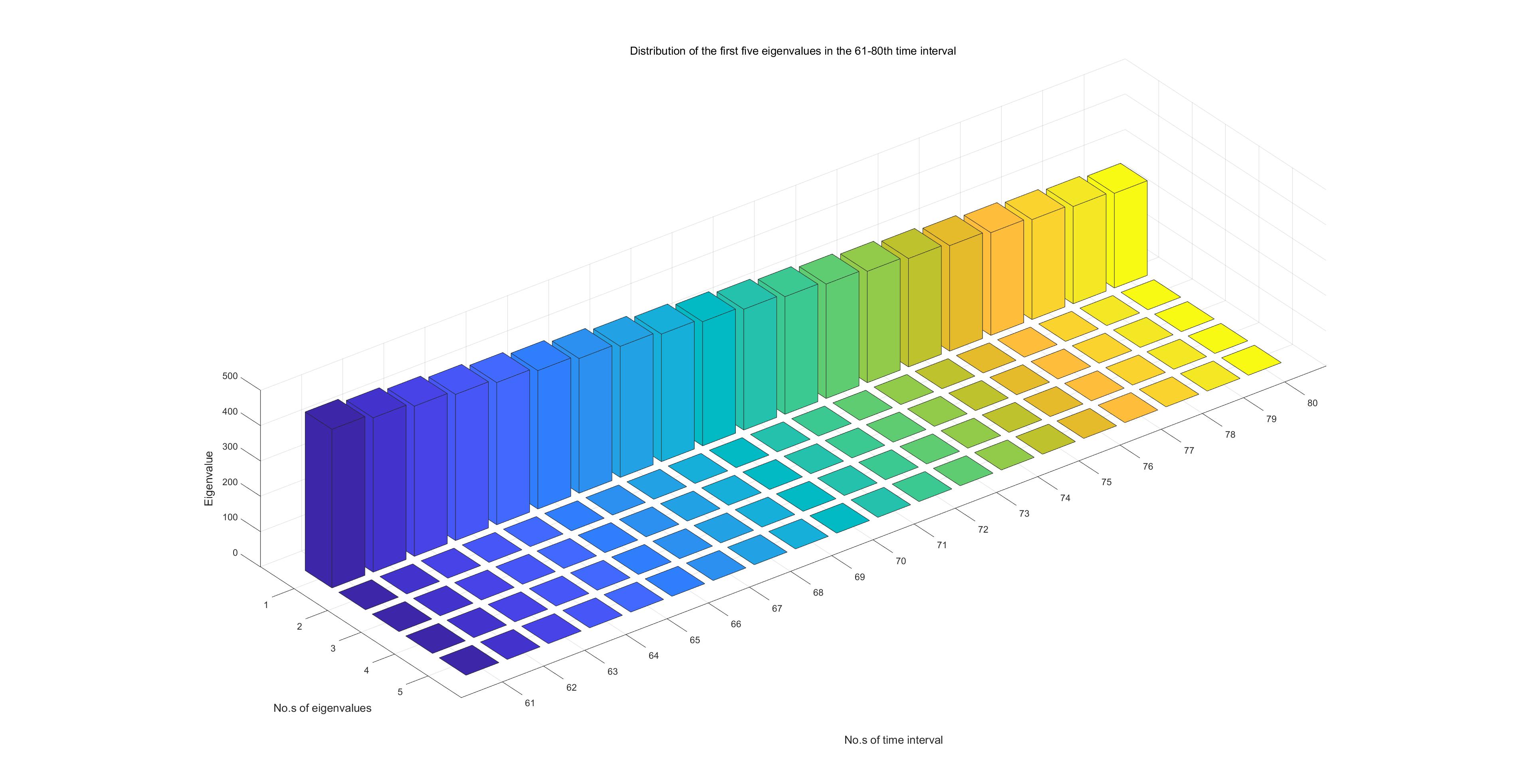}}

     \subfigure[  The 81th-100th time intervals]{\includegraphics[width=5cm, height =3cm]{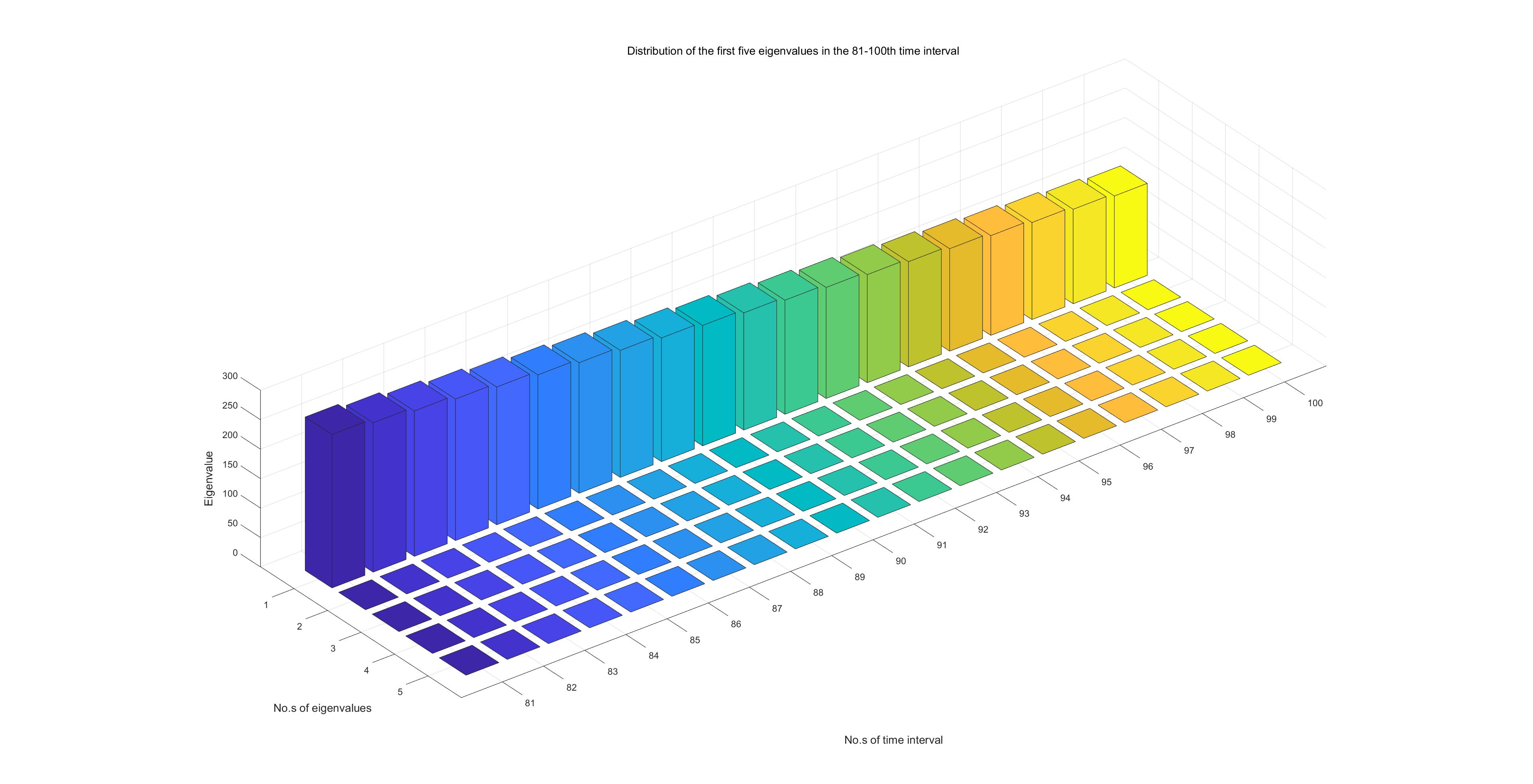}}
    \caption{The first 5 eigenvalues of $\mathfrak{U}_{i}^\top \mathfrak{U}_{i}$ (from big to small) for S2 with $f=0$.}\label{2D-eigenvaluechanges2f0}
\end{figure}
\begin{figure}[htbp]
    \centering
   \subfigure[The 1th-20th time intervals]{\includegraphics[width=5cm, height =3cm]{picnew/eig_s2_f0_1_20.jpg}} 
   \subfigure[ The 21th-40th time intervals]{\includegraphics[width=5cm, height =3cm]{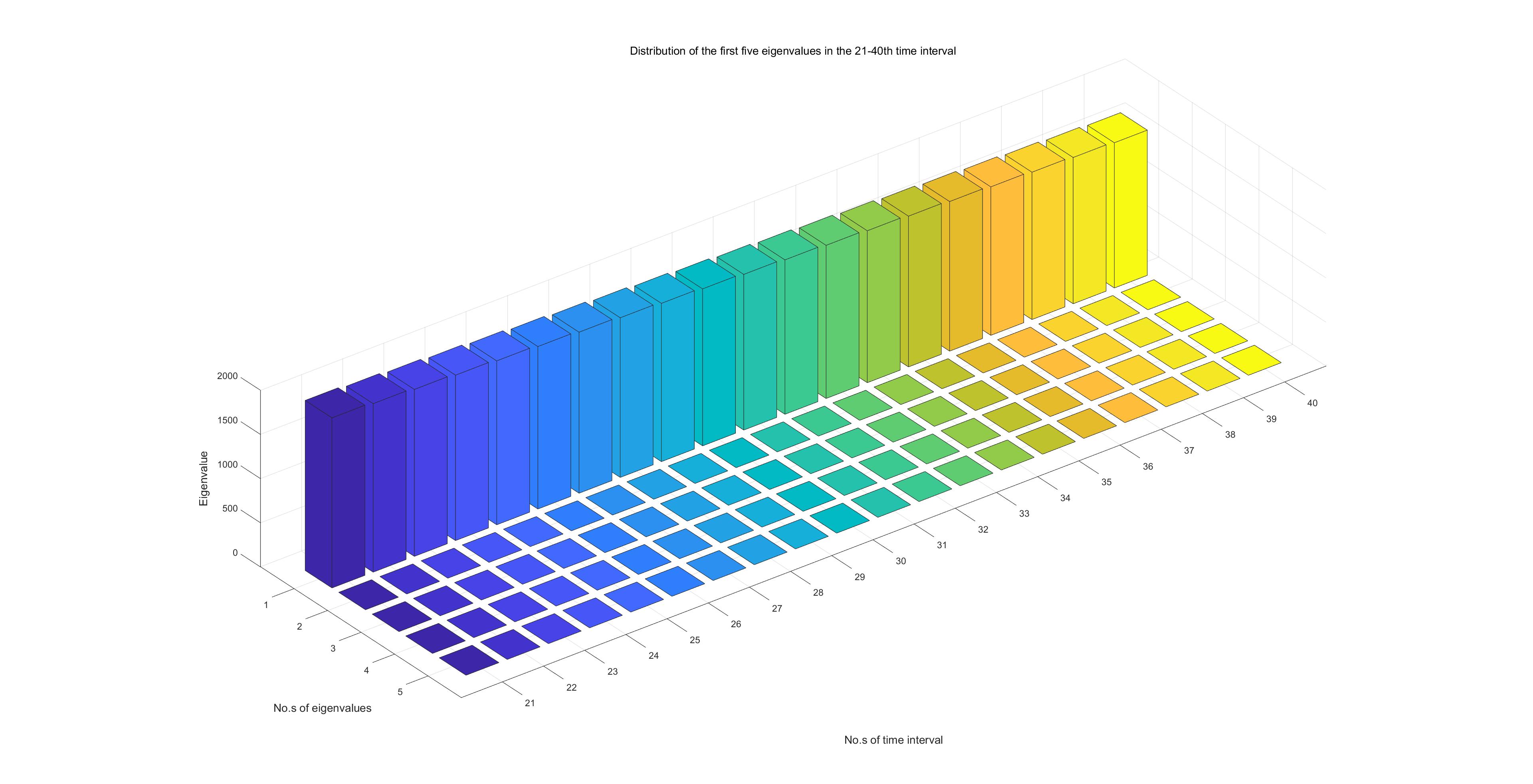}}

    \subfigure[  The 41th-60th time intervals]{\includegraphics[width=5cm, height =3cm]{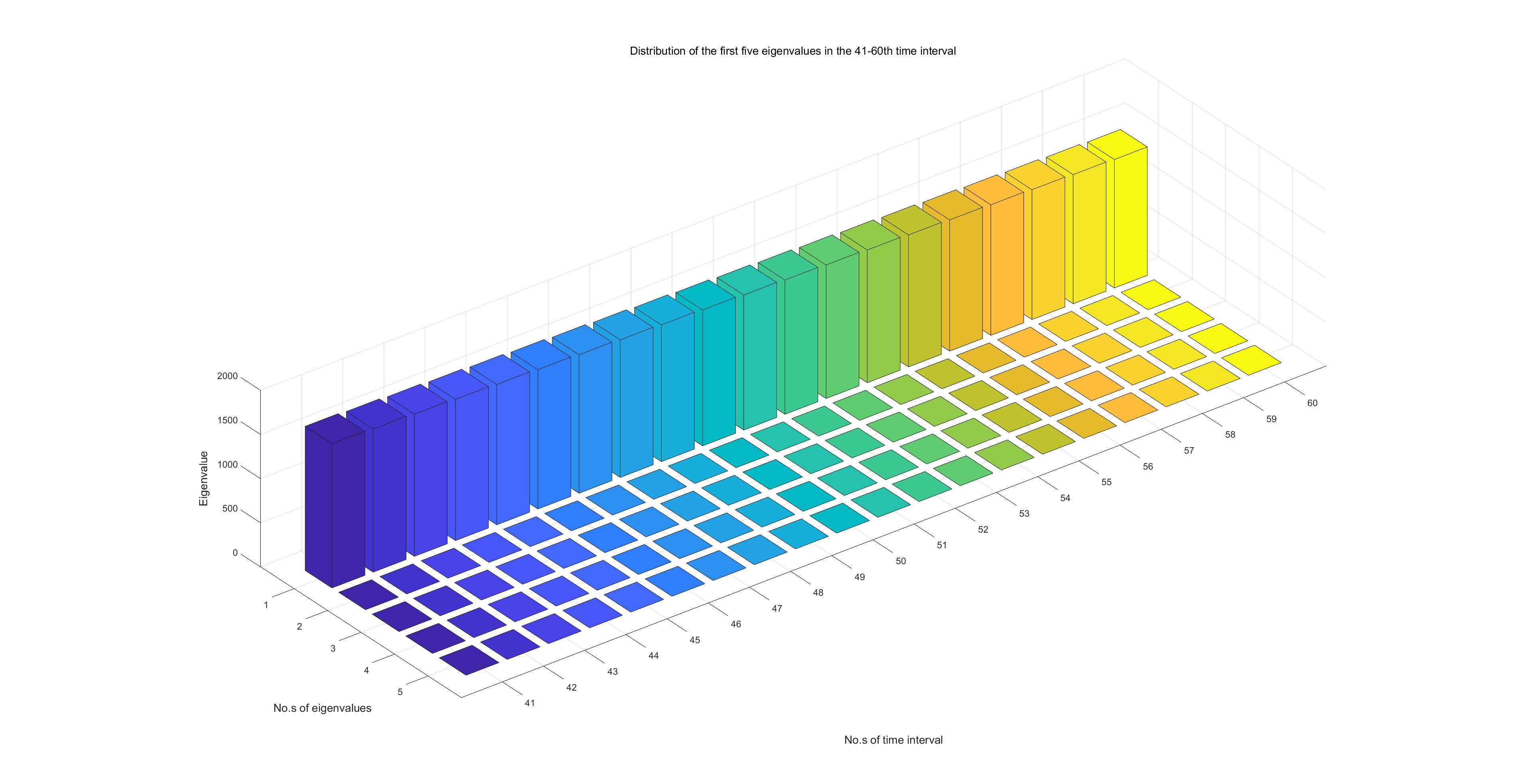}}
    \subfigure[  The 61th-80th time intervals]{\includegraphics[width=5cm, height =3cm]{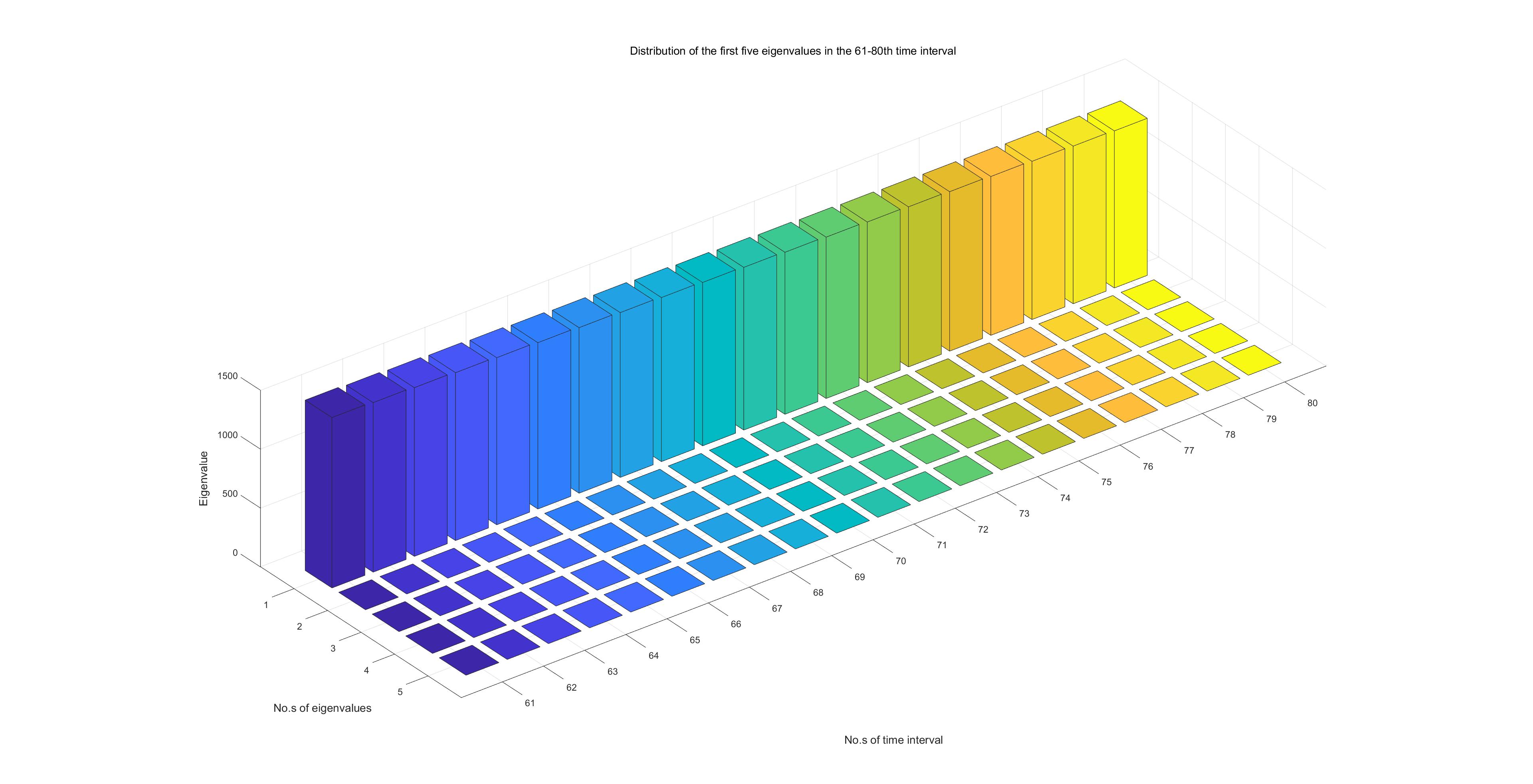}}

     \subfigure[  The 81th-100th time intervals]{\includegraphics[width=5cm, height =3cm]{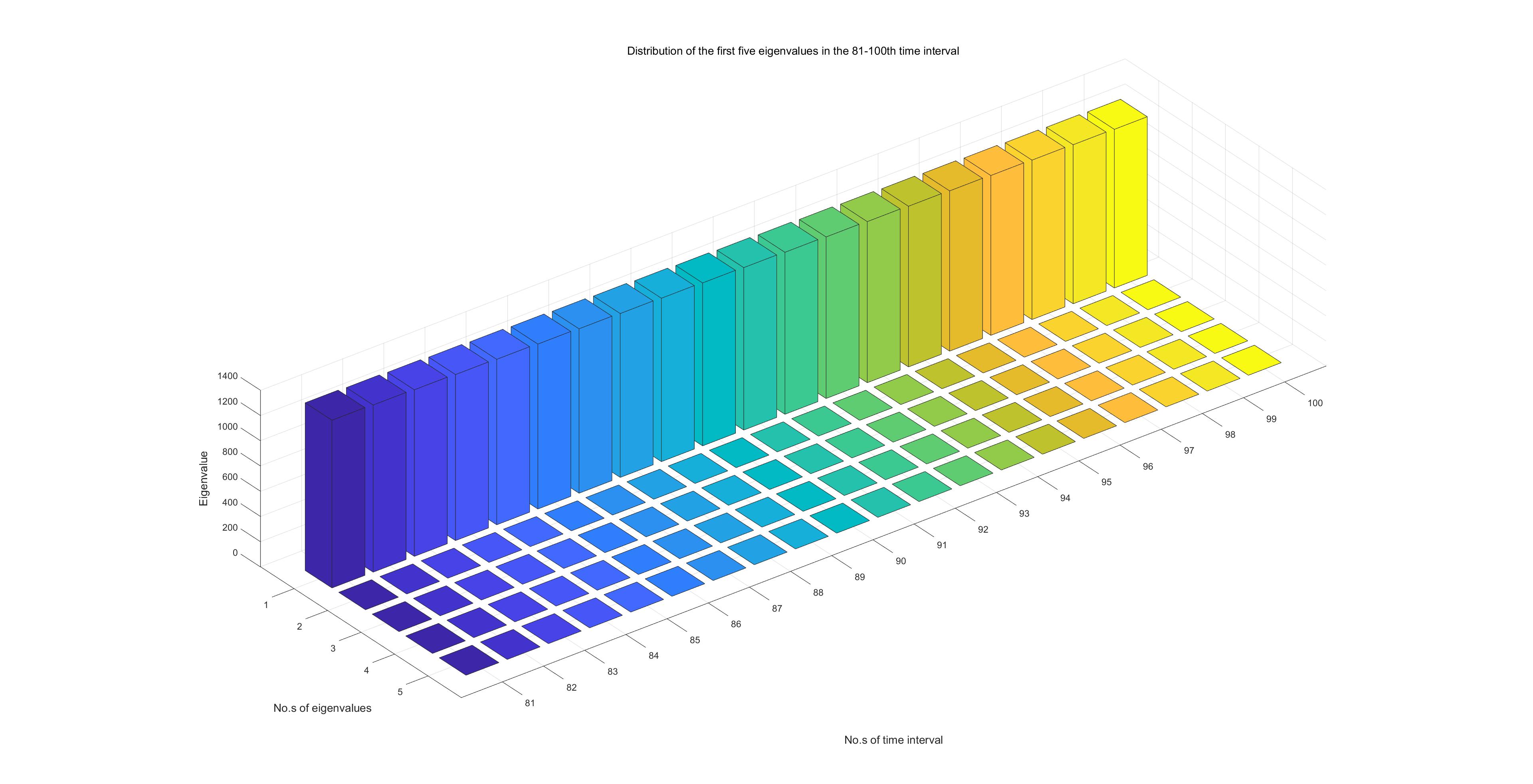}}
    \caption{The first 5 eigenvalues of $\mathfrak{U}_{i}^\top \mathfrak{U}_{i}$ (from big to small) for S2 with $f=10$.}\label{2D-eigenvaluechanges2f10}
\end{figure}
\begin{figure}[htbp]
    \centering
   \subfigure[The 1th-20th time intervals]{\includegraphics[width=5cm, height =3cm]{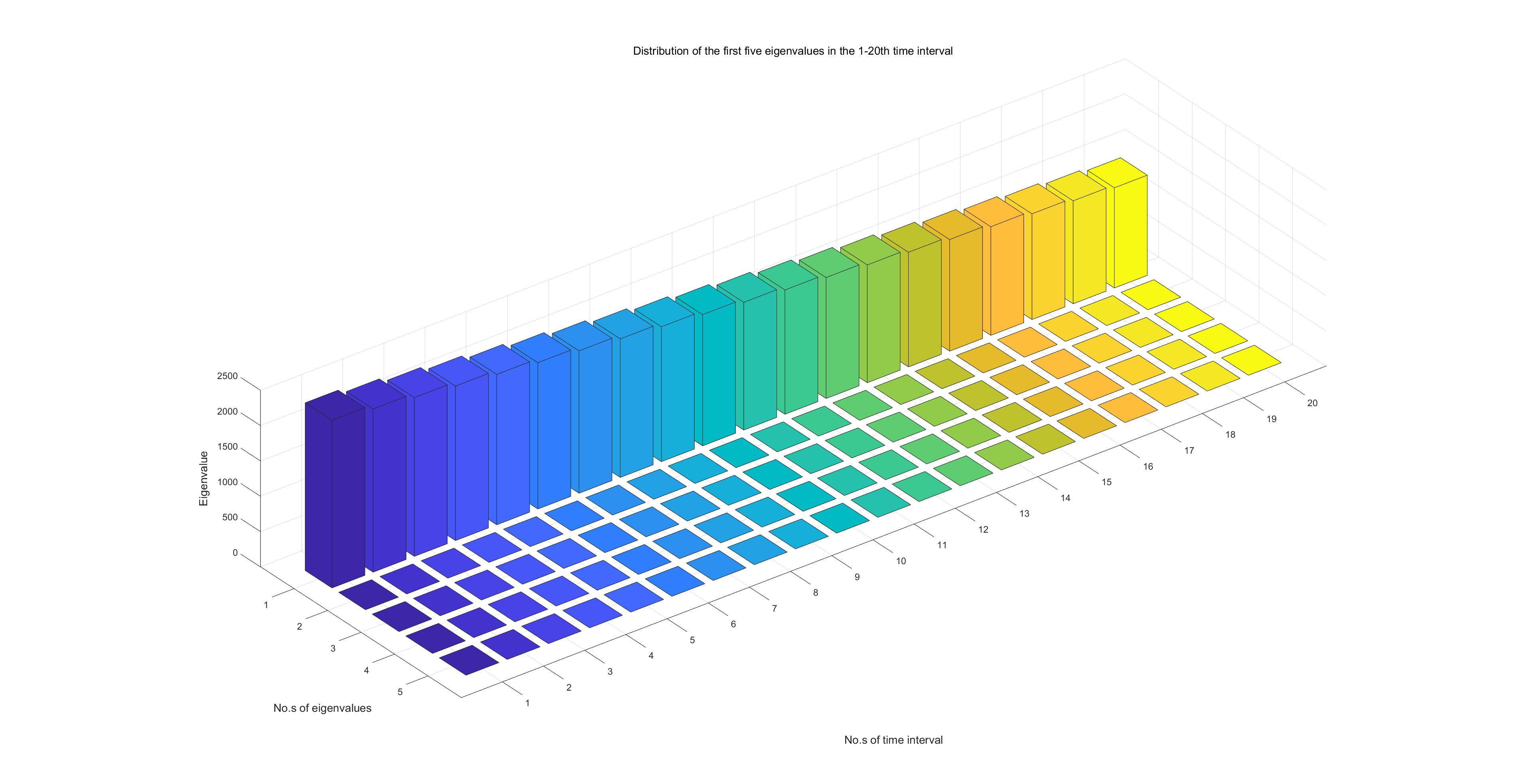}} 
   \subfigure[ The 21th-40th time intervals]{\includegraphics[width=5cm, height =3cm]{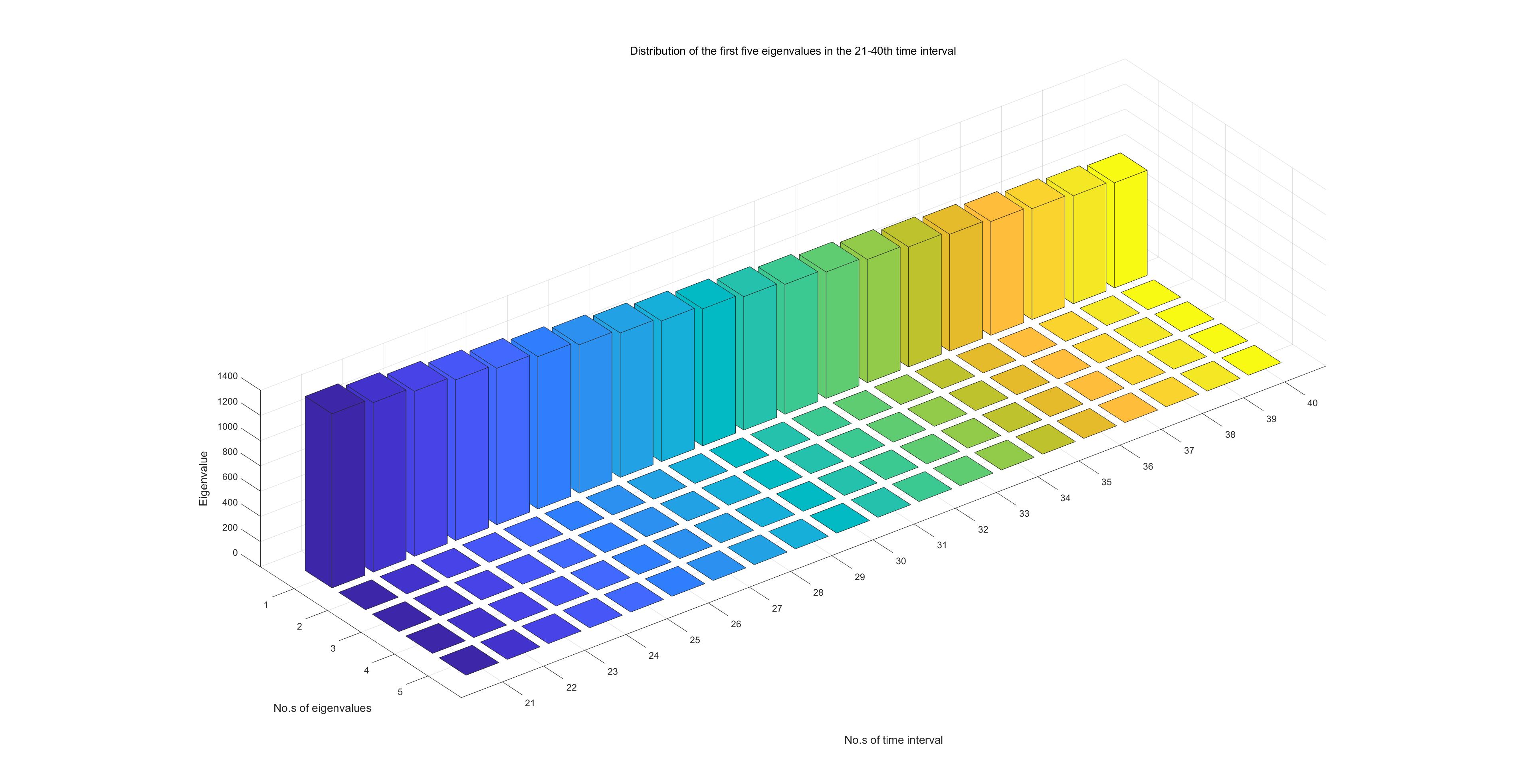}}

    \subfigure[  The 41th-60th time intervals]{\includegraphics[width=5cm, height =3cm]{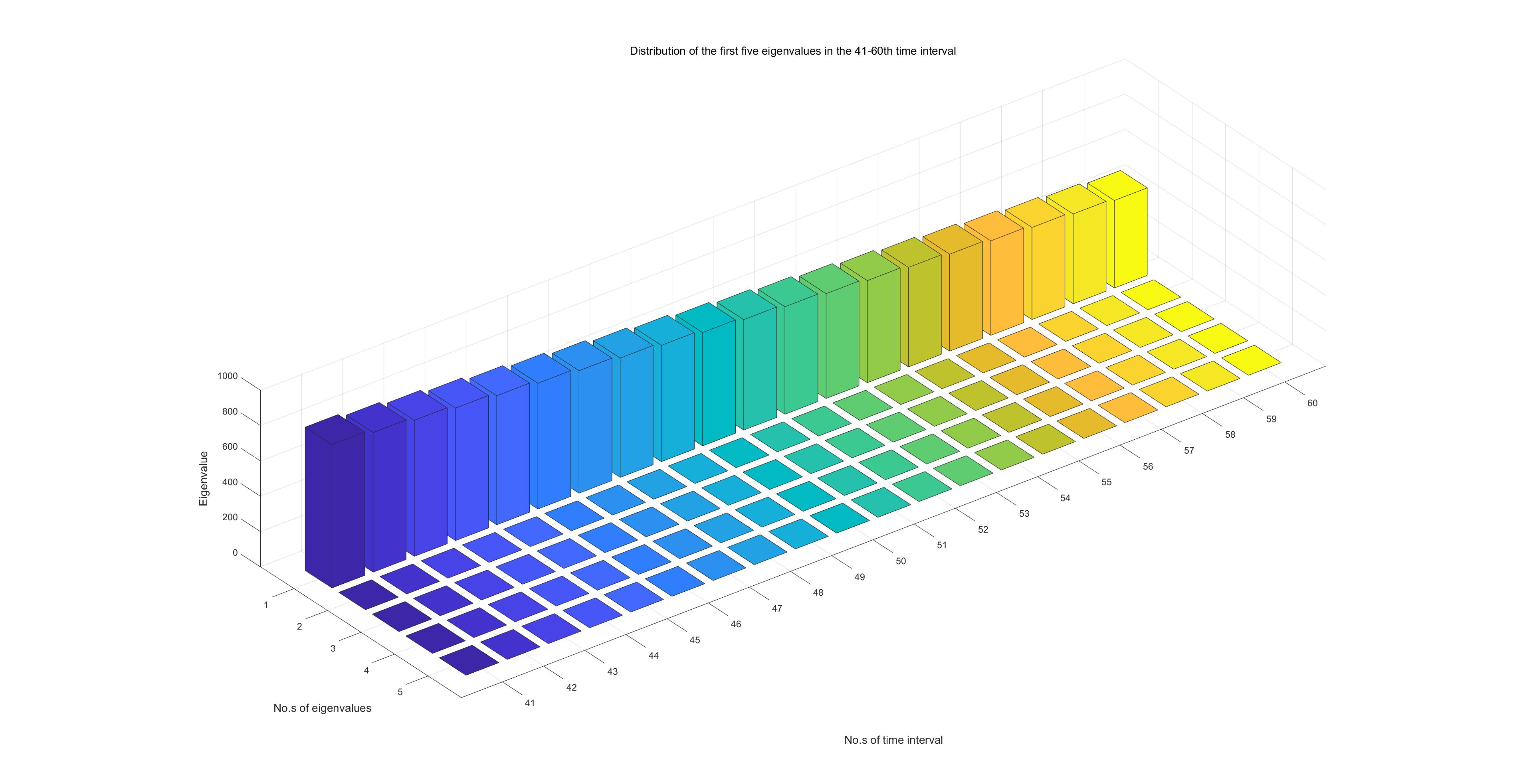}}
    \subfigure[  The 61th-80th time intervals]{\includegraphics[width=5cm, height =3cm]{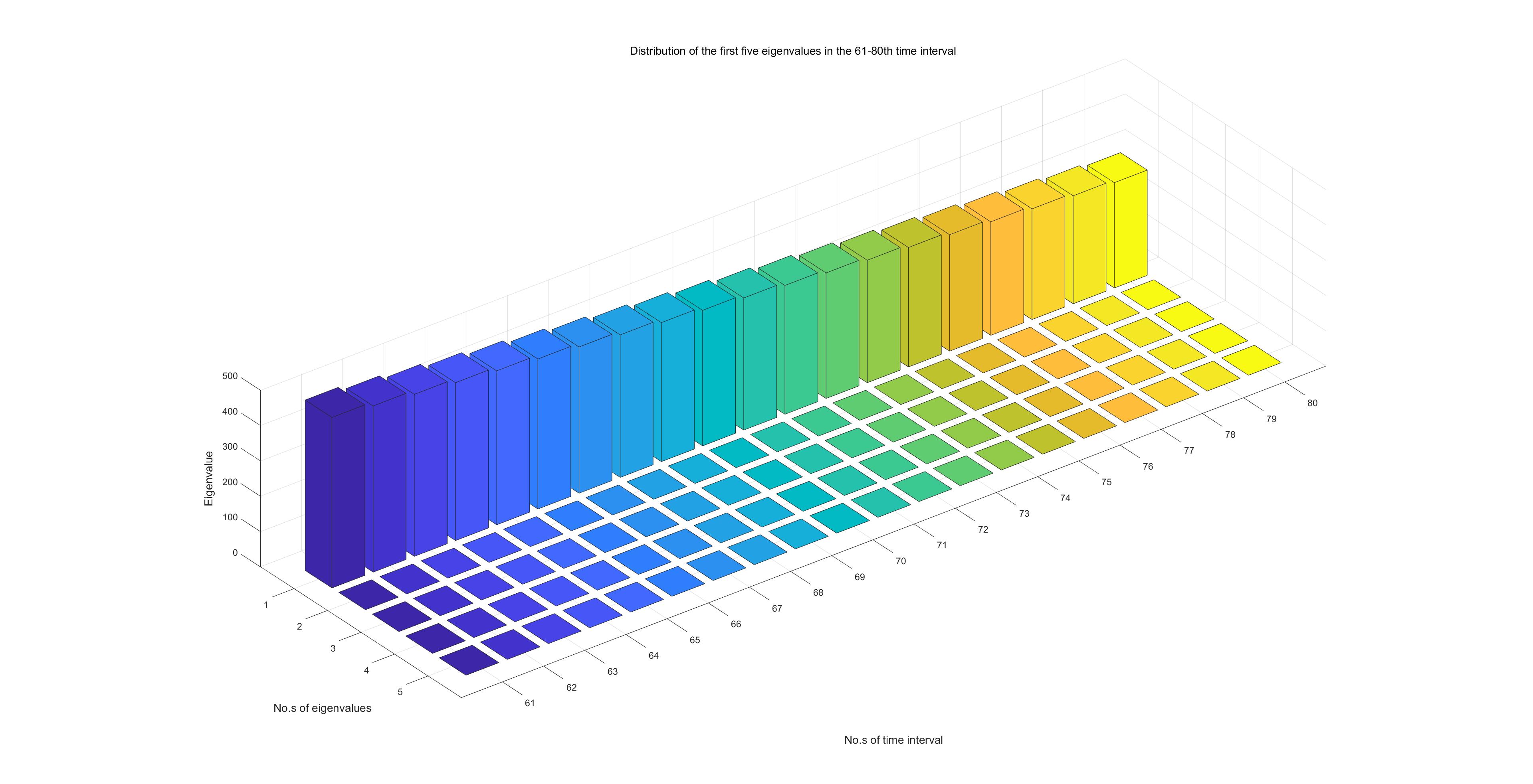}}

     \subfigure[  The 81th-100th time intervals]{\includegraphics[width=5cm, height =3cm]{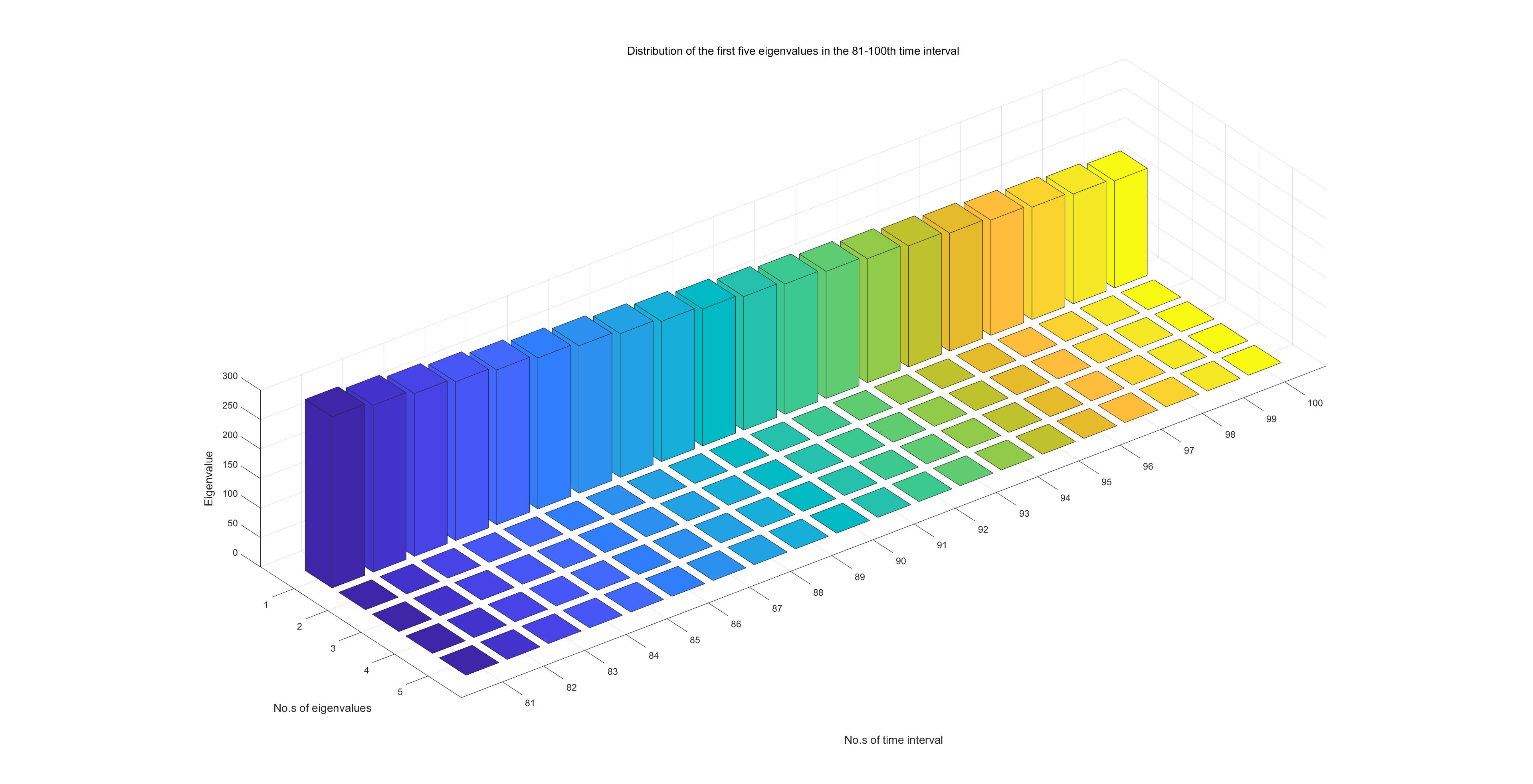}}
    \caption{The first 5 eigenvalues of $\mathfrak{U}_{i}^\top \mathfrak{U}_{i}$ (from big to small) for S2 with $f=xy$.}\label{2D-eigenvaluechanges2fxy}
\end{figure}
\begin{figure}[htbp]
    \centering
    \includegraphics[width=9cm, height =5cm]{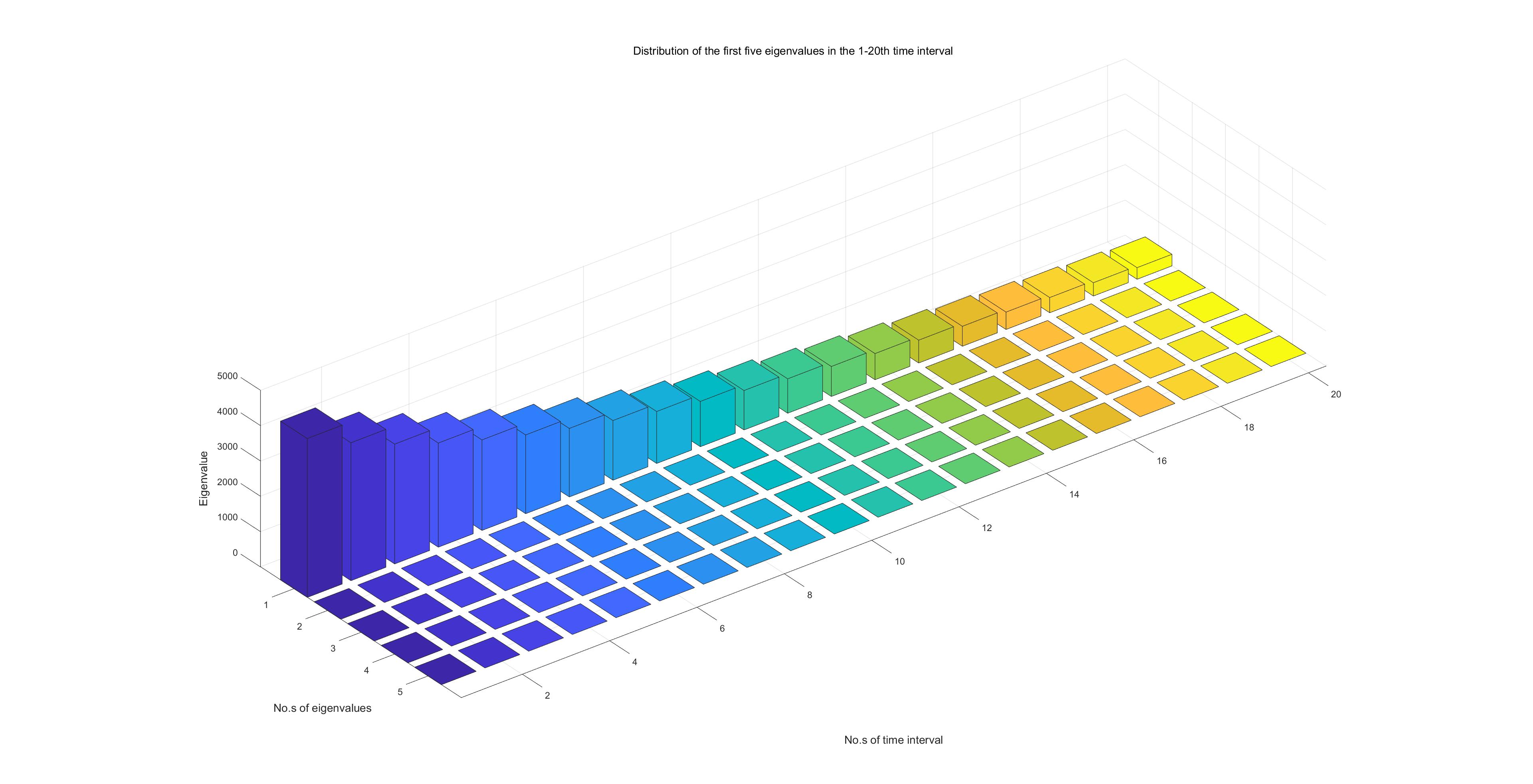}
    \caption{The first 5 eigenvalues of $\mathfrak{U}_{i}^\top \mathfrak{U}_{i}$ (from big to small) for S3 with $f=0$.}\label{2D-eigenvaluechanges3f0}
\end{figure}

\begin{figure}[htbp]
    \centering
   \subfigure[  $t = 0.25$]{\includegraphics[width=4cm, height =3cm]{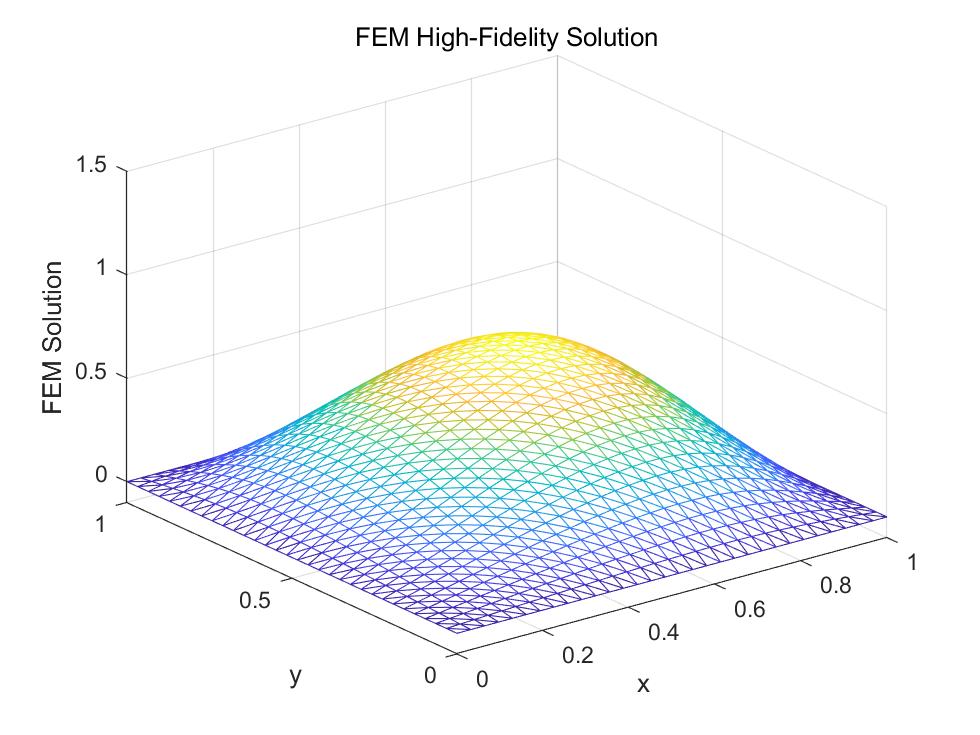}}
   \subfigure[  $t = 0.50$]{\includegraphics[width=4cm, height =3cm]{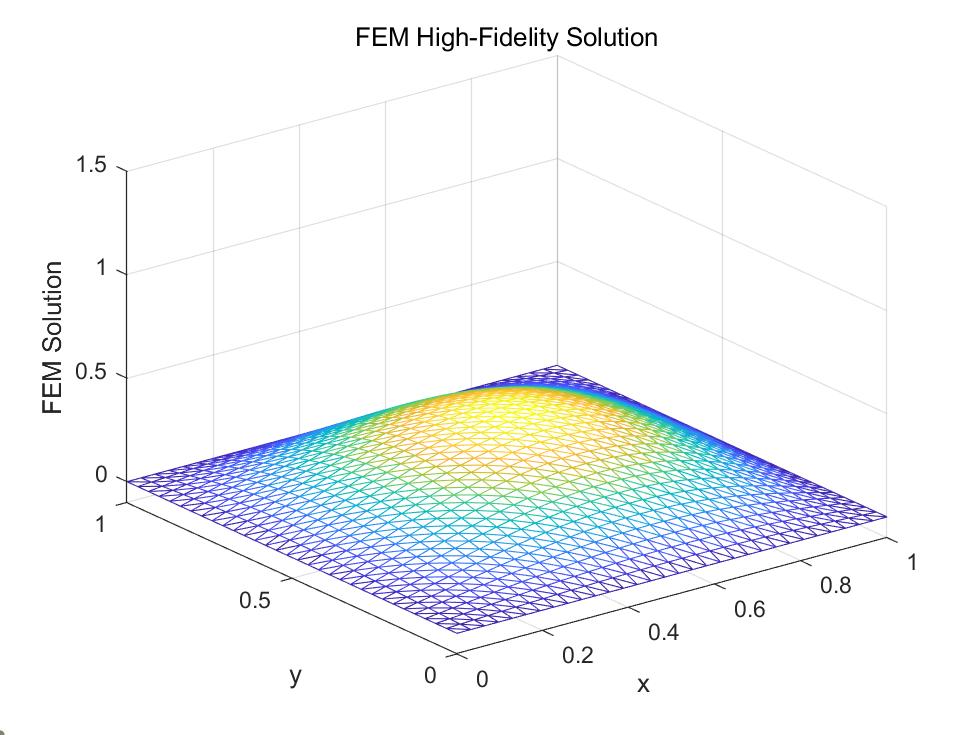}}
   
   \subfigure[  $t = 0.75$]{\includegraphics[width=4cm, height =3cm]{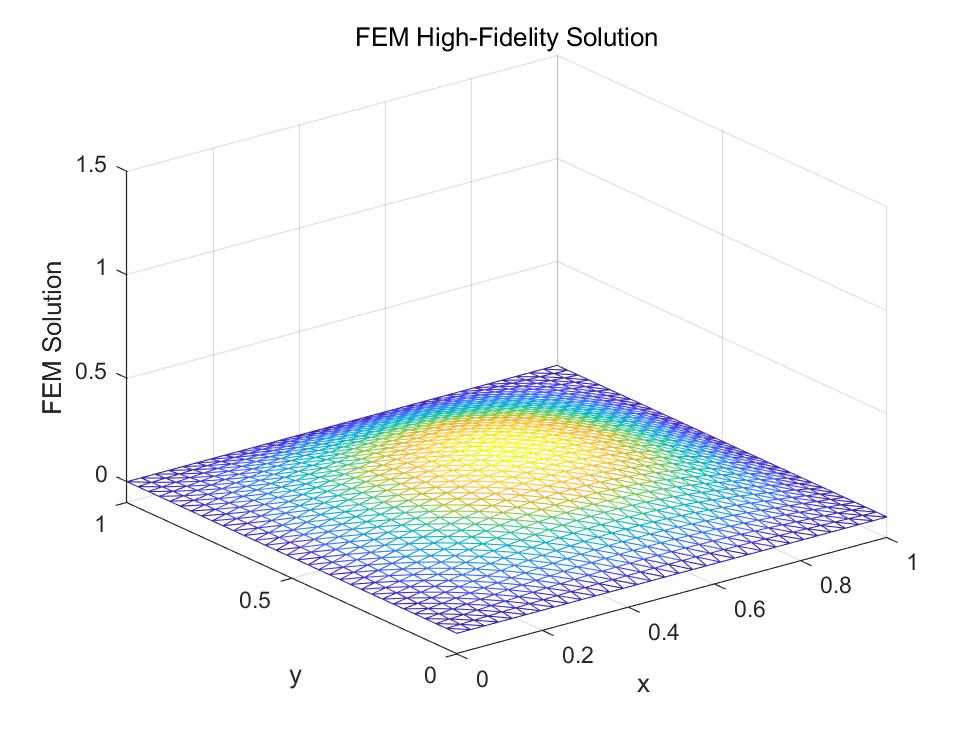}}
    \subfigure[ $t =1.00$]{\includegraphics[width=4cm, height =3cm]{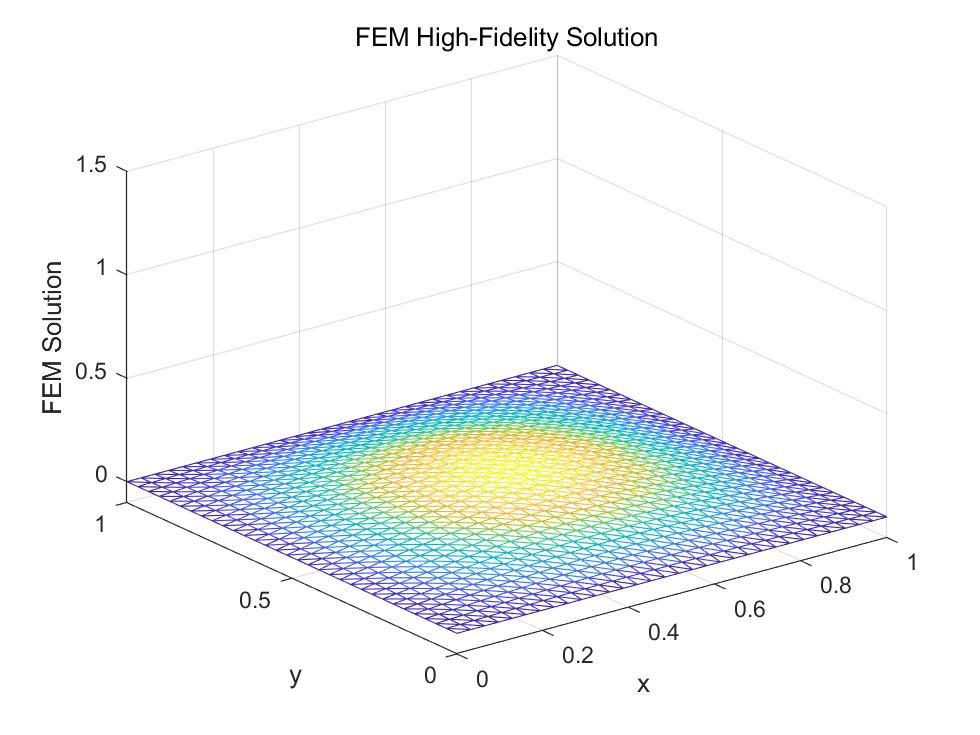}}\\
    High-fidelity numerical solution\\
    
     \subfigure[ $t =0.25$]{\includegraphics[width=4cm, height =3cm]{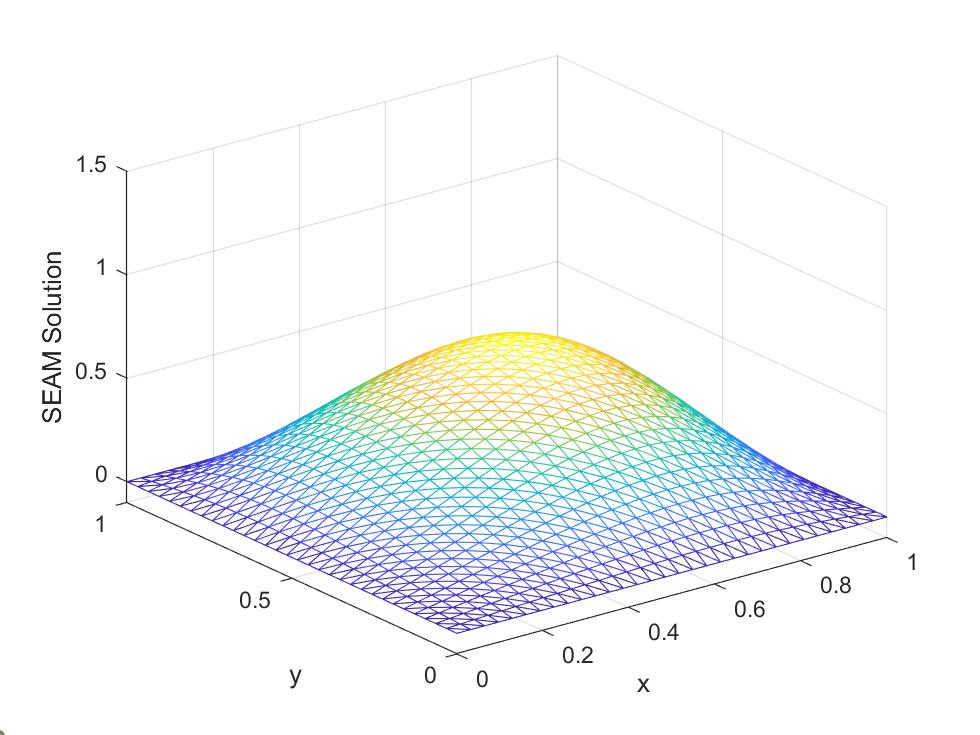}}
   \subfigure[ $t =0.50$]{\includegraphics[width=4cm, height =3cm]{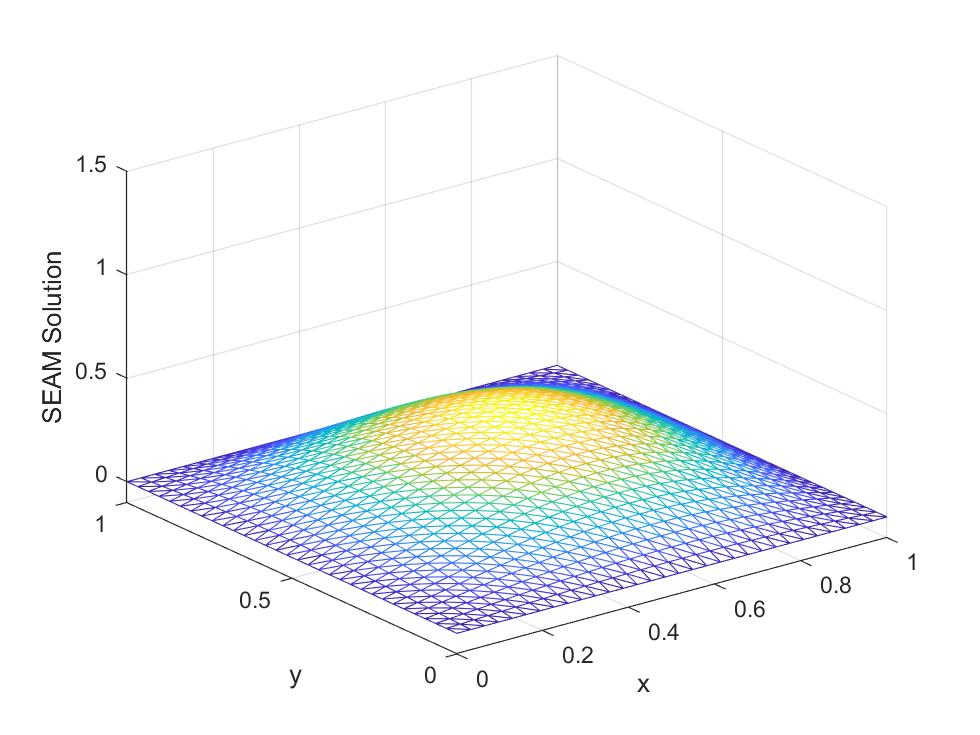}}
   
   \subfigure[  $t =0.75$]{\includegraphics[width=4cm, height =3cm]{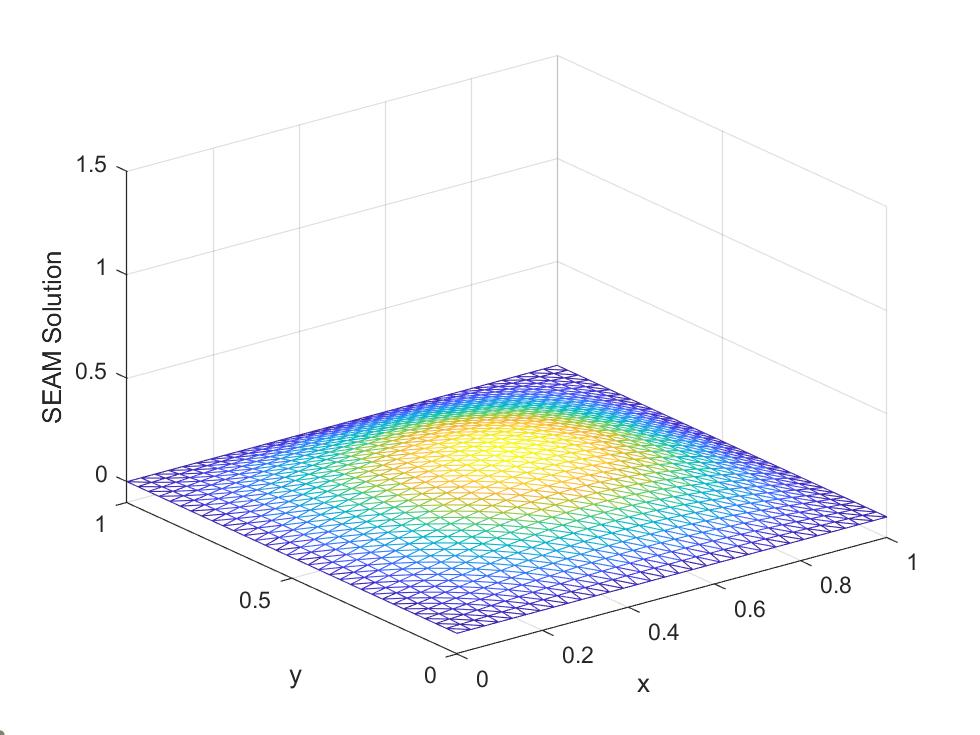}}
    \subfigure[ $t =1.00$]{\includegraphics[width=4cm, height =3cm]{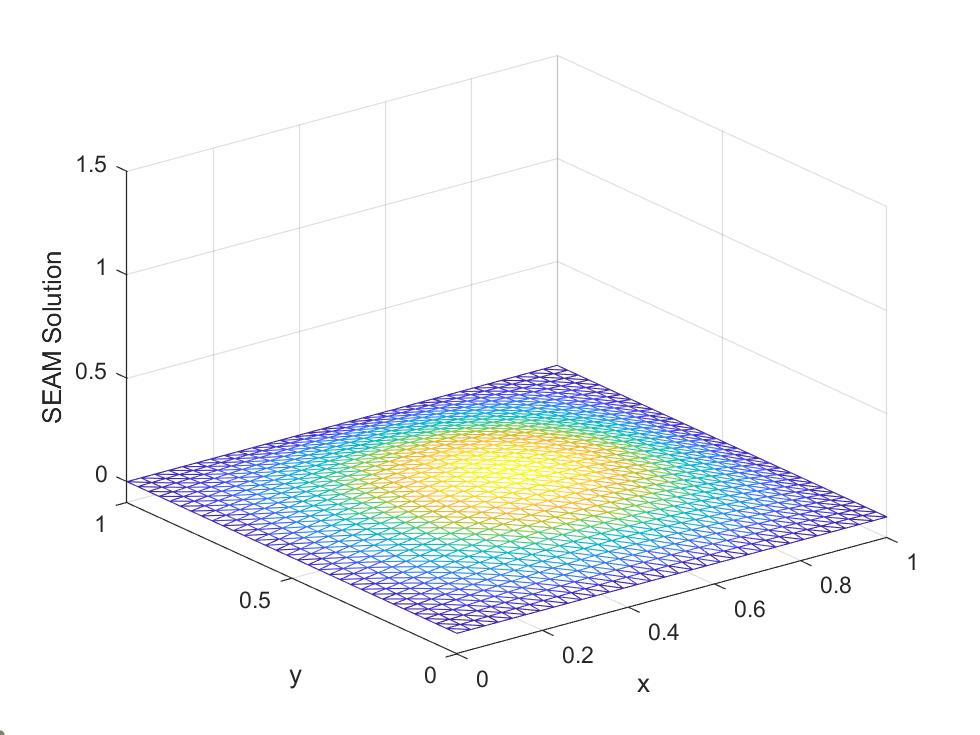}}\\
    SEAM solution\\
    \caption{High-fidelity / SEAM solutions of (\ref{nml2}) at different time for S1 with $f = 0$.}\label{seams1f0}
\end{figure}
\begin{figure}[htbp]
    \centering
   \subfigure[  $t = 0.25$]{\includegraphics[width=4cm, height =3cm]{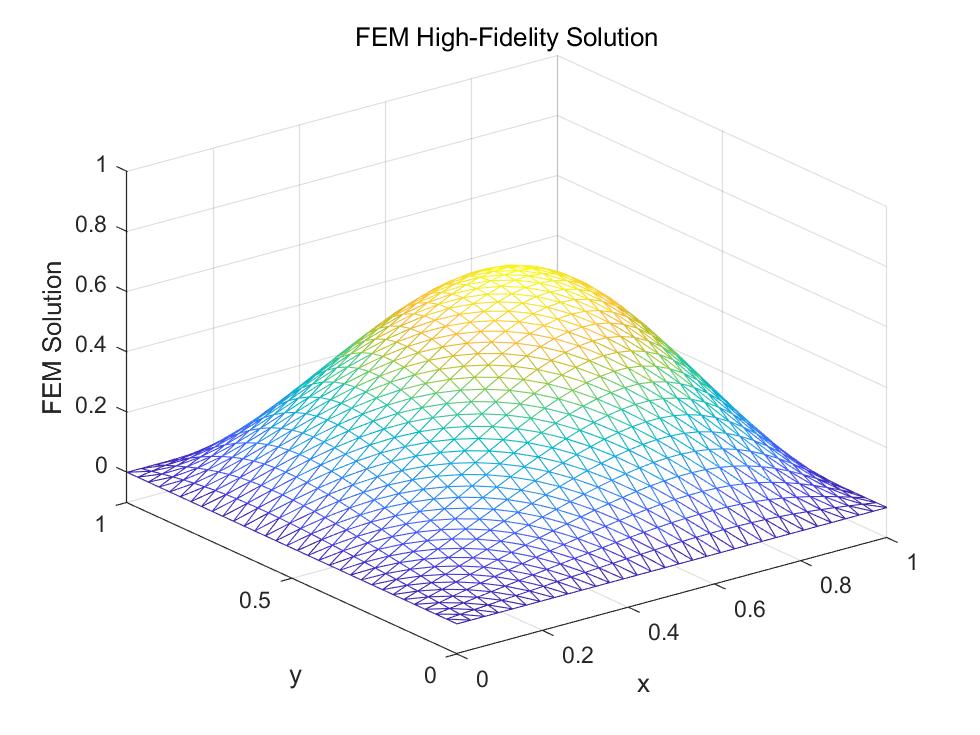}}
   \subfigure[  $t = 0.50$]{\includegraphics[width=4cm, height =3cm]{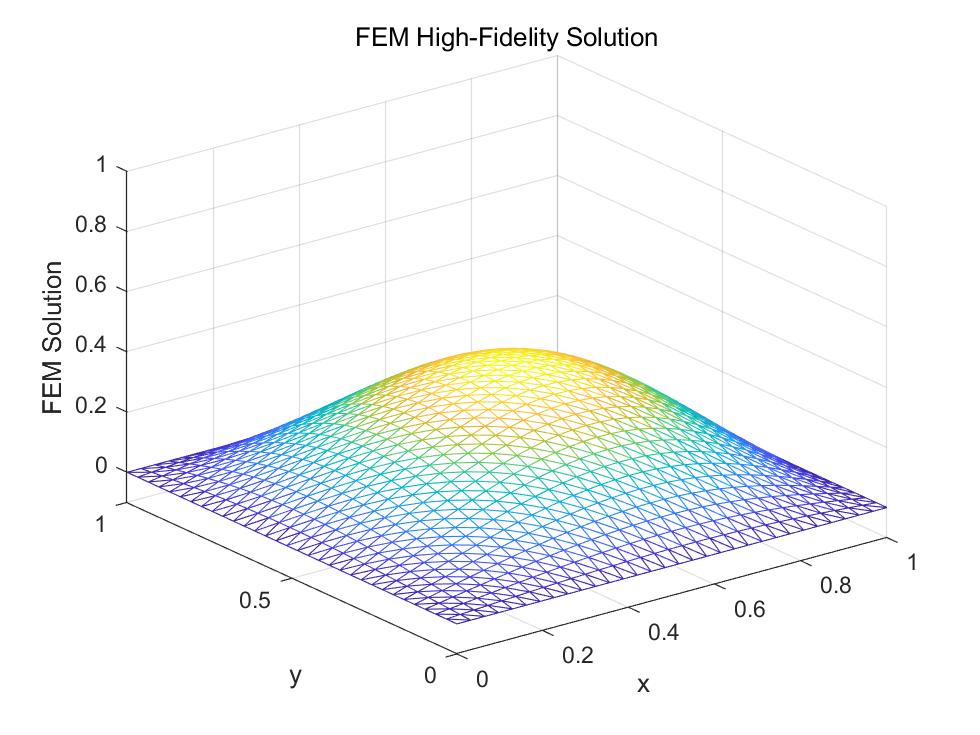}}
   
   \subfigure[  $t = 0.75$]{\includegraphics[width=4cm, height =3cm]{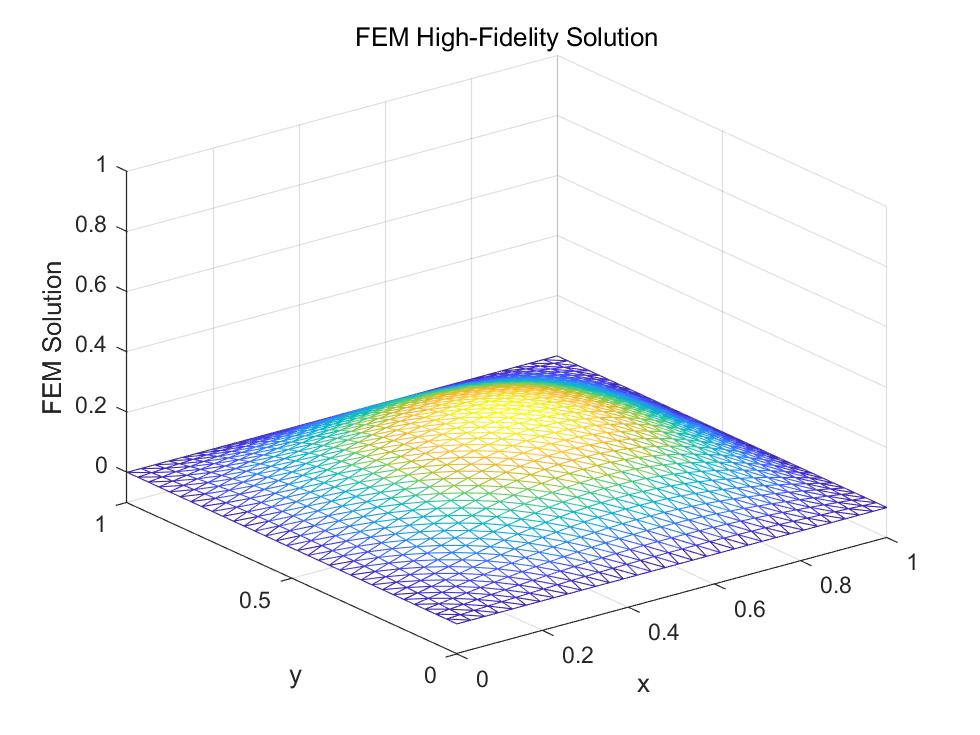}}
    \subfigure[ $t =1.00$]{\includegraphics[width=4cm, height =3cm]{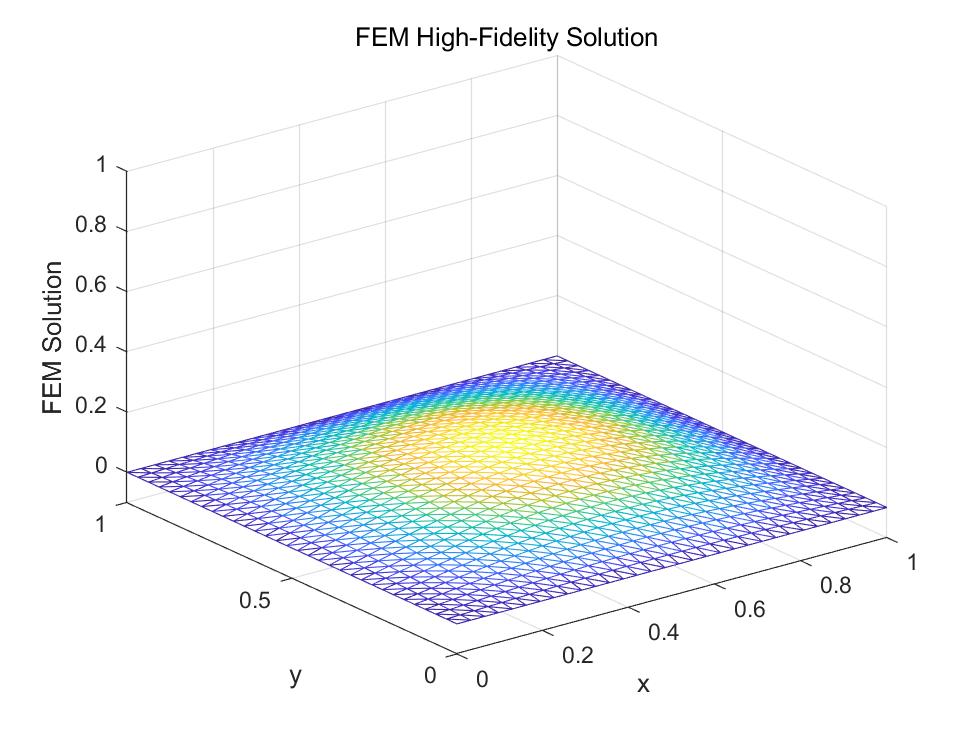}}\\
    High-fidelity numerical solution\\
    
     \subfigure[ $t =0.25$]{\includegraphics[width=4cm, height =3cm]{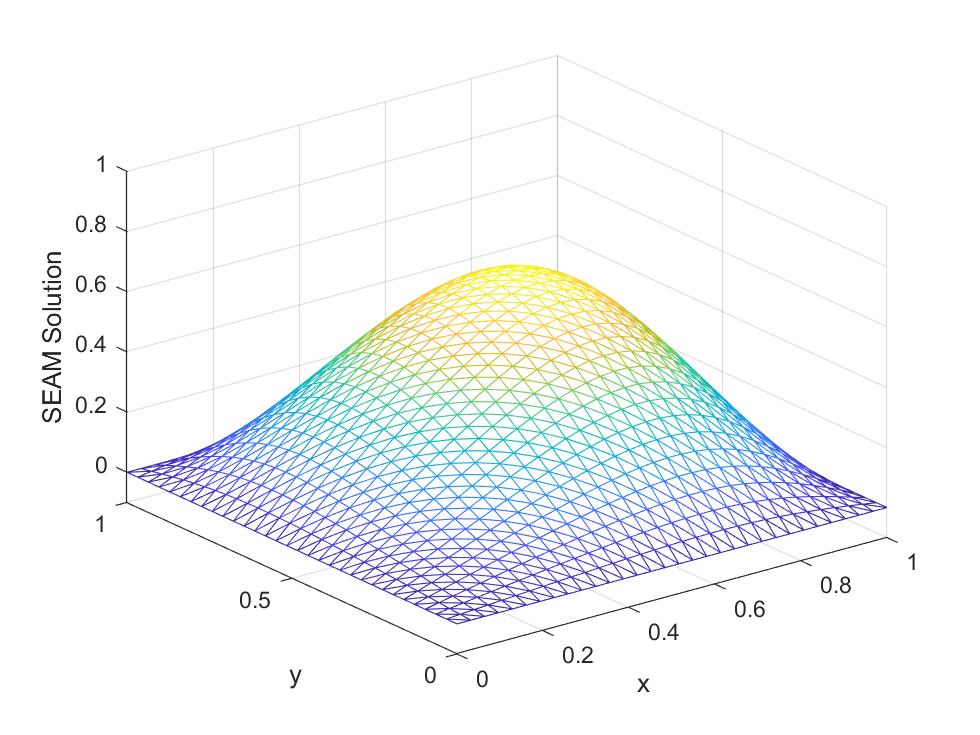}}
   \subfigure[ $t =0.50$]{\includegraphics[width=4cm, height =3cm]{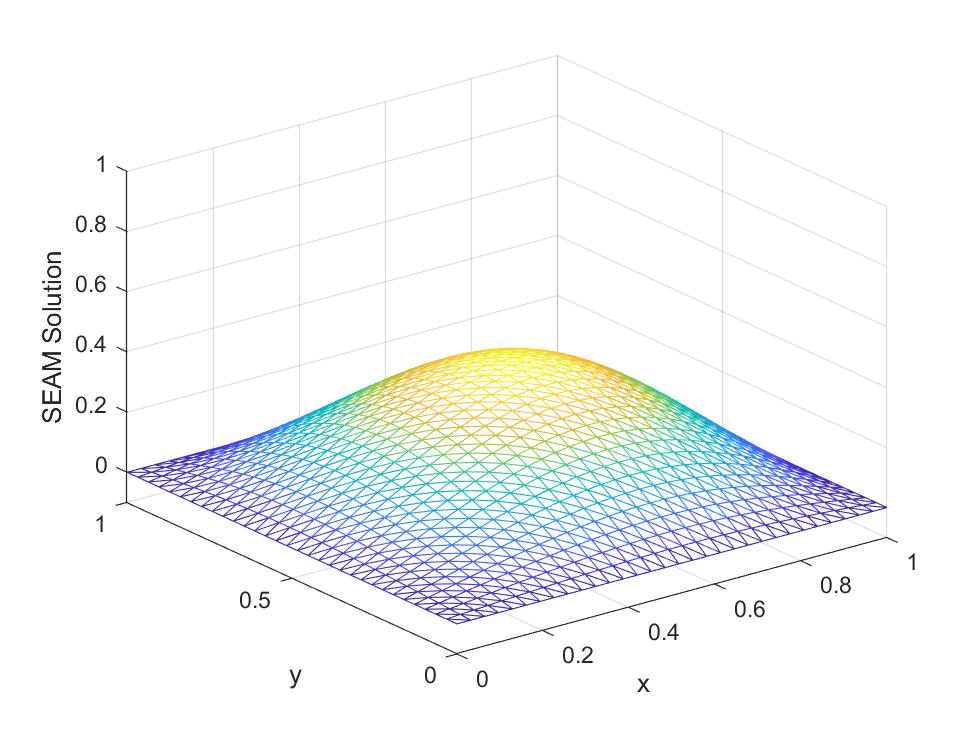}}
   
   \subfigure[  $t =0.75$]{\includegraphics[width=4cm, height =3cm]{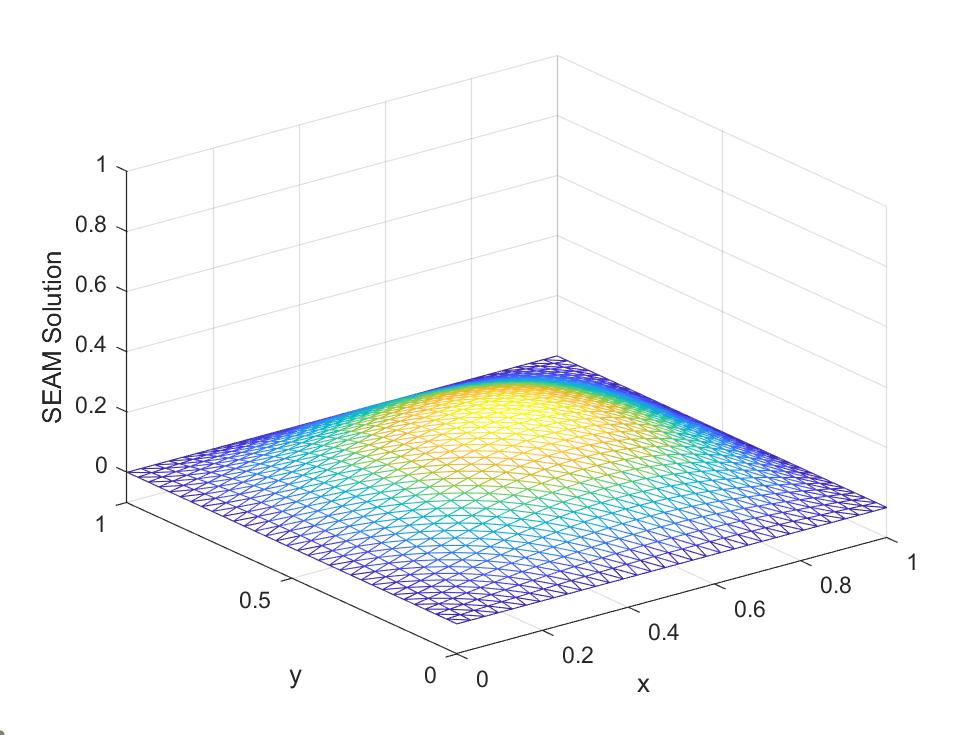}}
    \subfigure[ $t =1.00$]{\includegraphics[width=4cm, height =3cm]{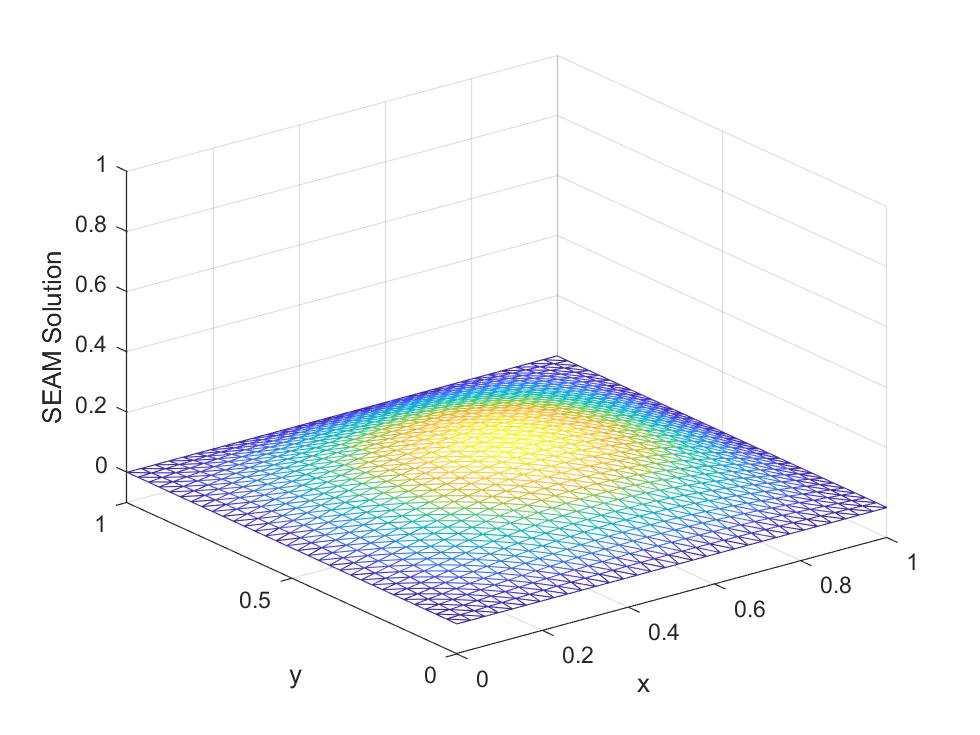}}\\
    SEAM solution\\
    \caption{High-fidelity / SEAM solutions of (\ref{nml2}) at different time for S1 with $f = xy$.}\label{seams1fxy}
\end{figure}
\begin{figure}[htbp]
    \centering
   \subfigure[  $t = 0.25$]{\includegraphics[width=4cm, height =3cm]{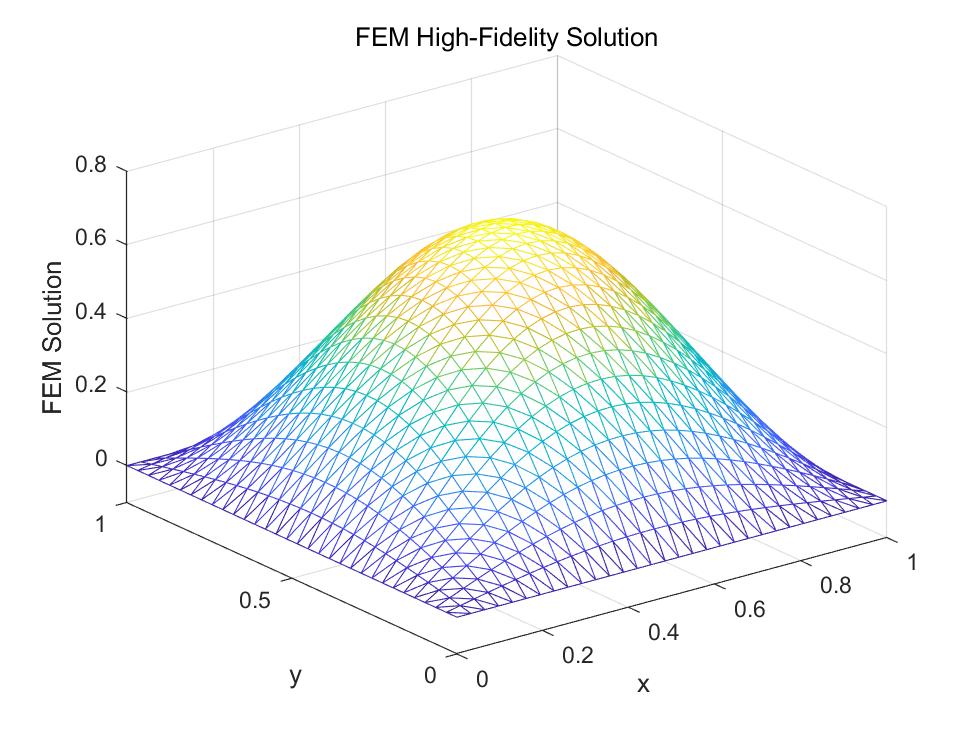}}
   \subfigure[  $t = 0.50$]{\includegraphics[width=4cm, height =3cm]{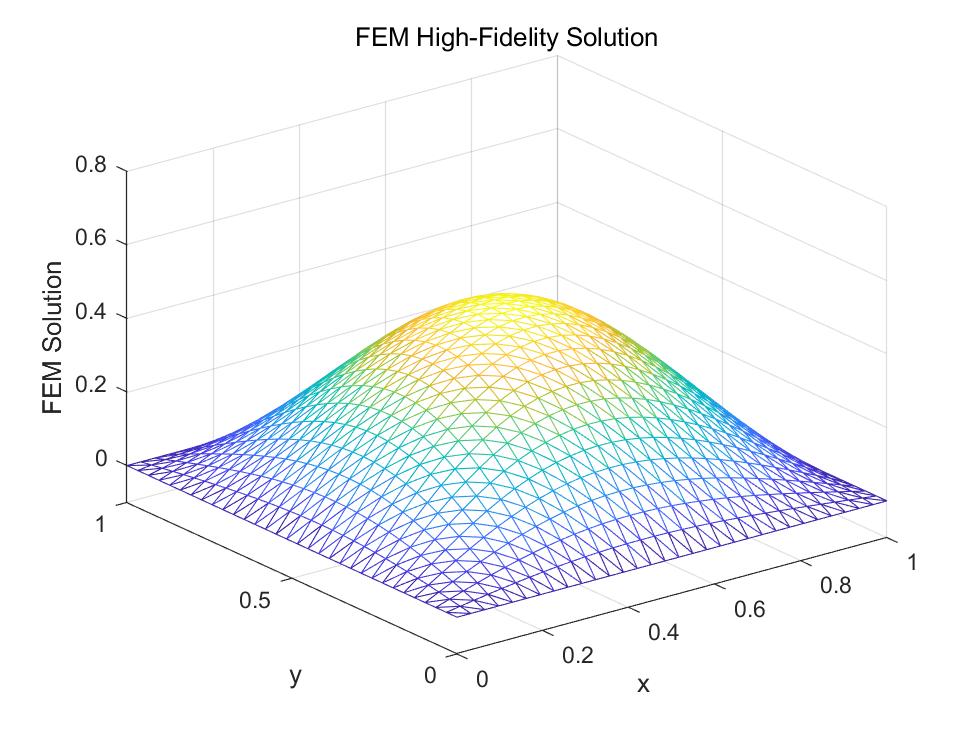}}
   
   \subfigure[  $t = 0.75$]{\includegraphics[width=4cm, height =3cm]{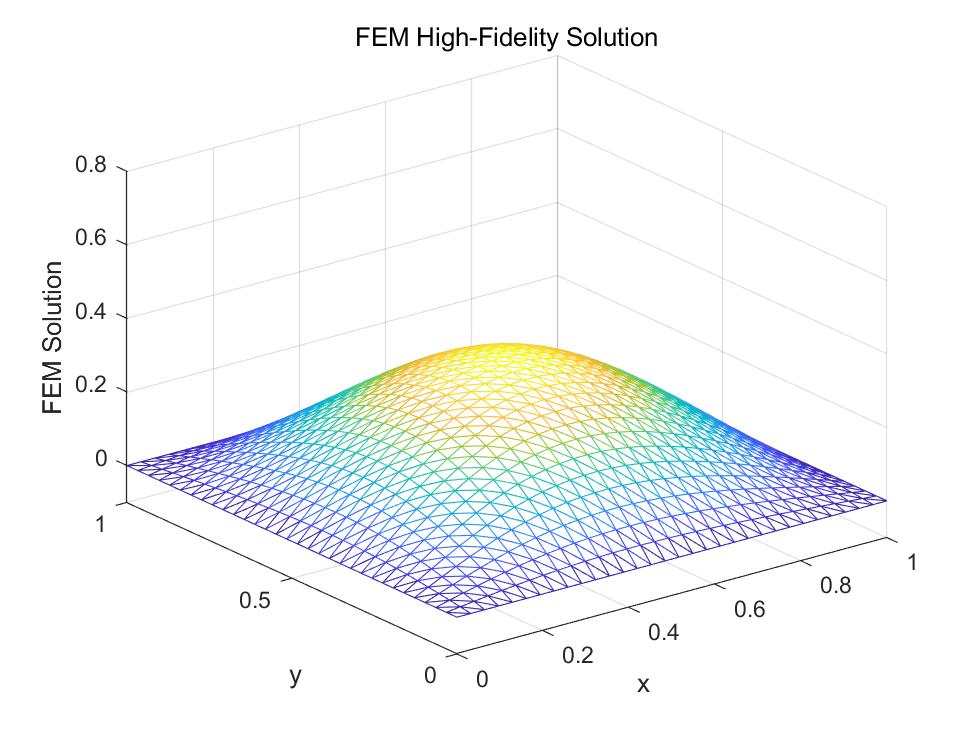}}
    \subfigure[ $t =1.00$]{\includegraphics[width=4cm, height =3cm]{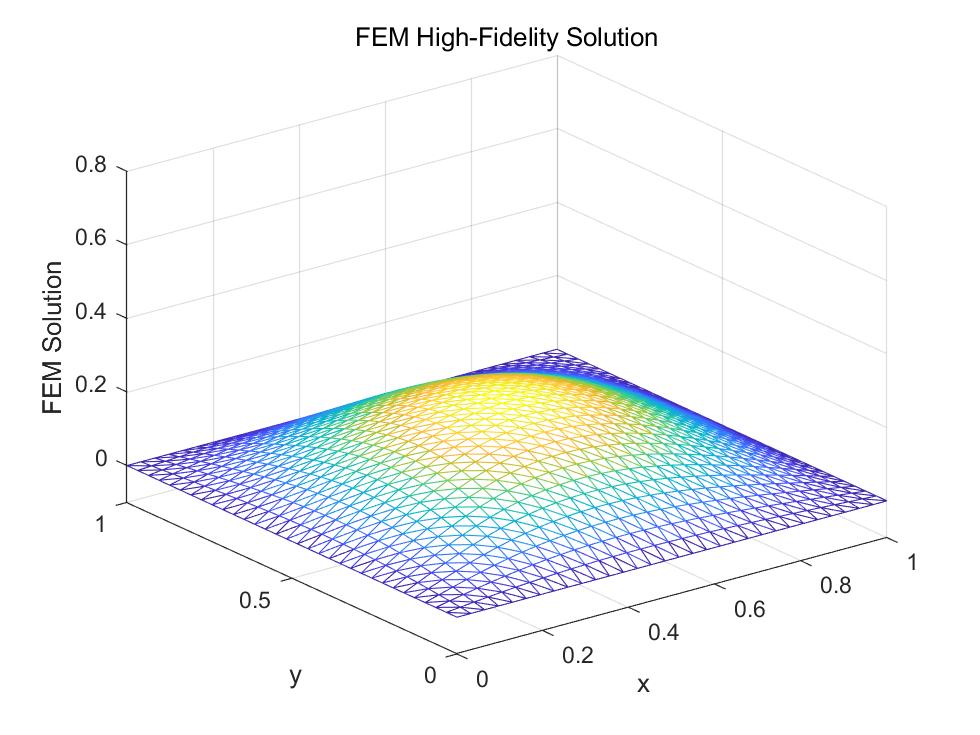}}\\
    High-fidelity numerical solution\\
    
     \subfigure[ $t =0.25$]{\includegraphics[width=4cm, height =3cm]{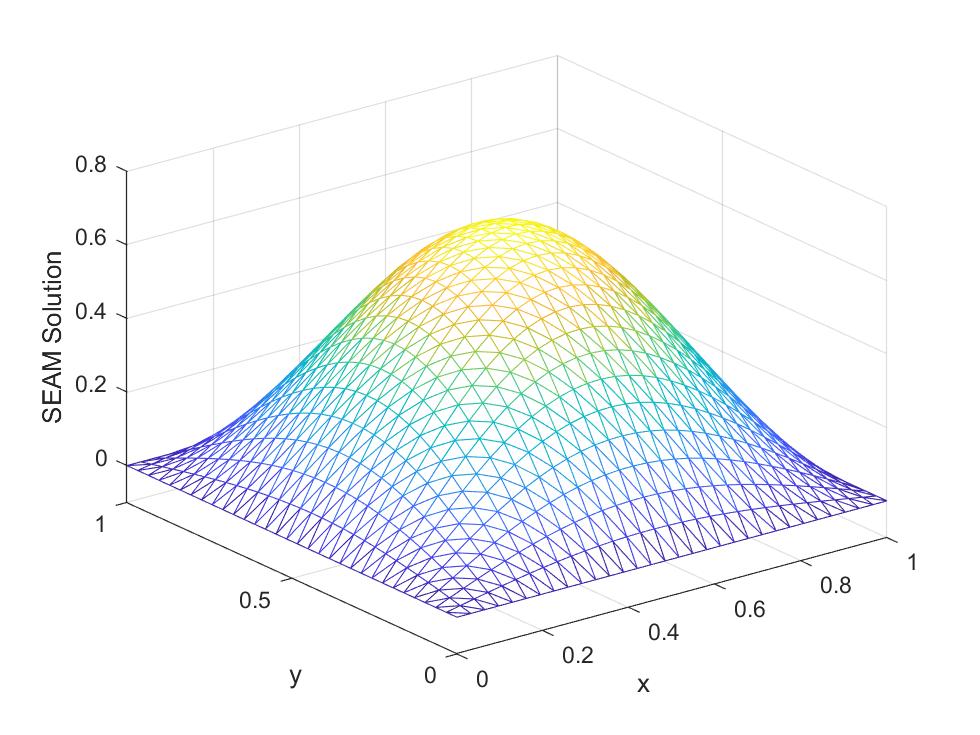}}
   \subfigure[ $t =0.50$]{\includegraphics[width=4cm, height =3cm]{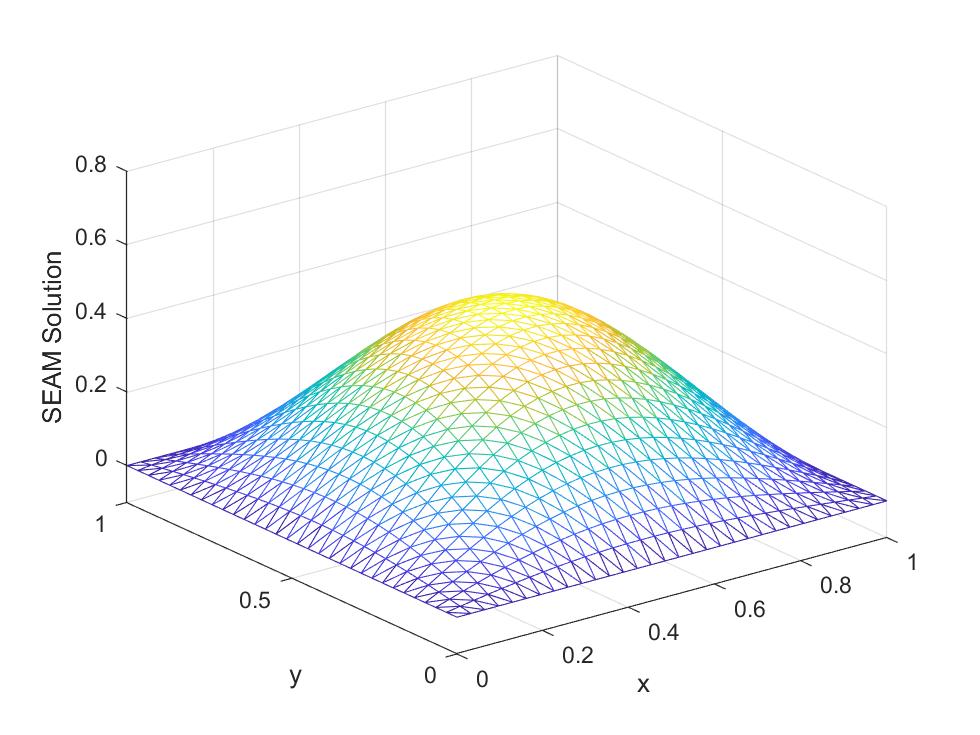}}
   
   \subfigure[  $t =0.75$]{\includegraphics[width=4cm, height =3cm]{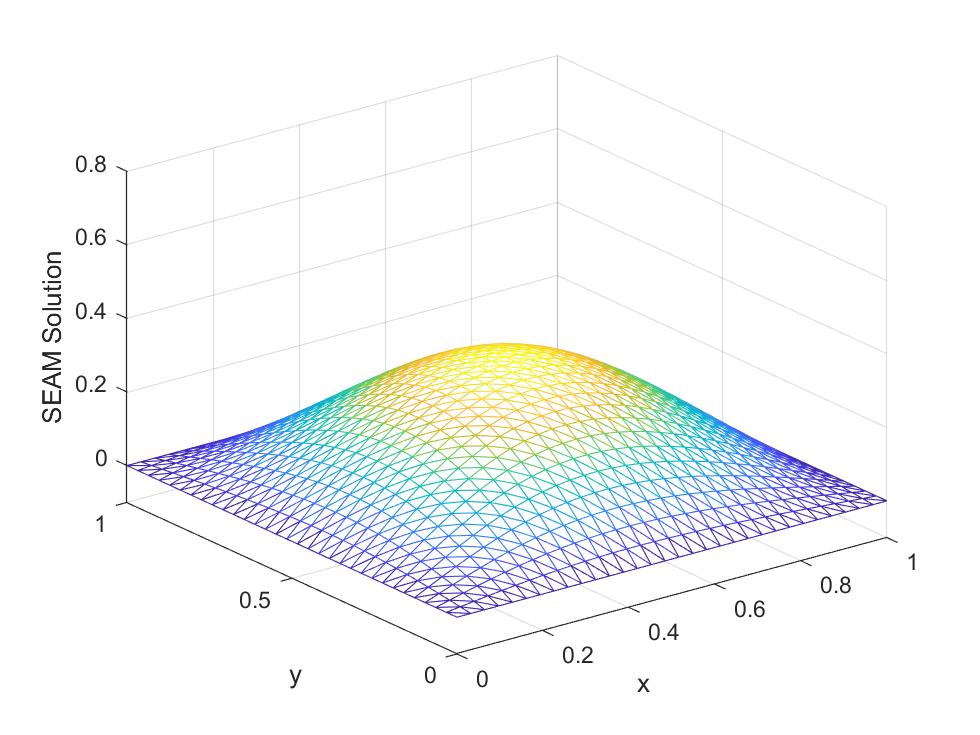}}
    \subfigure[ $t =1.00$]{\includegraphics[width=4cm, height =3cm]{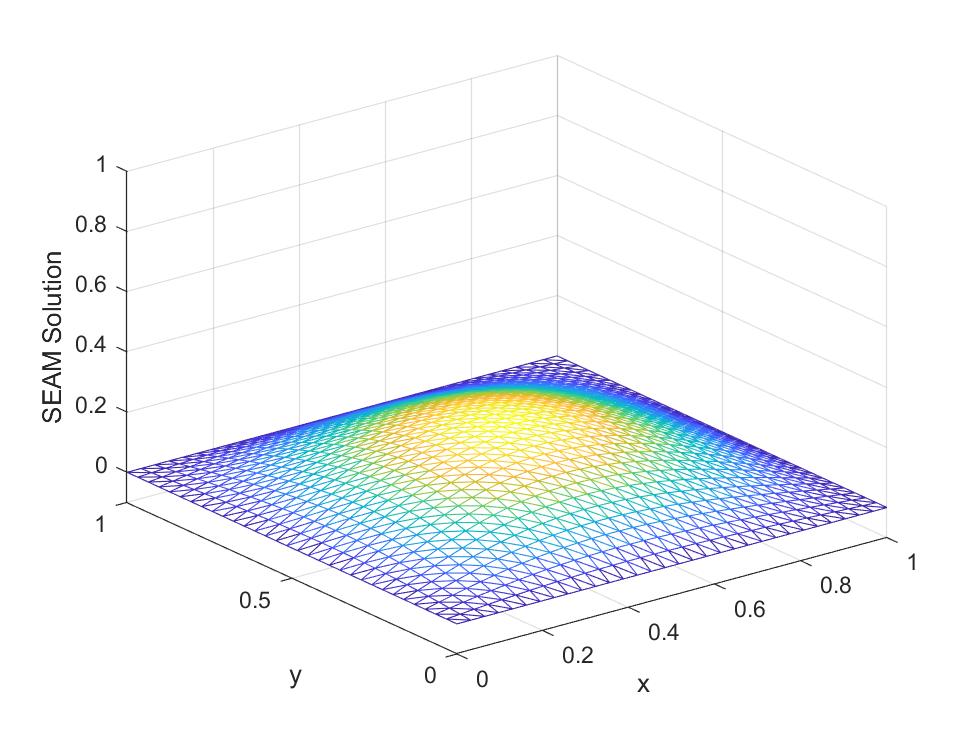}}\\
    SEAM solution\\
    \caption{High-fidelity / SEAM solutions of (\ref{nml2}) at different time   for S2 with $f = 0$.}\label{seams2f0}
\end{figure}
\begin{figure}[htbp]
    \centering
   \subfigure[  $t = 0.25$]{\includegraphics[width=4cm, height =3cm]{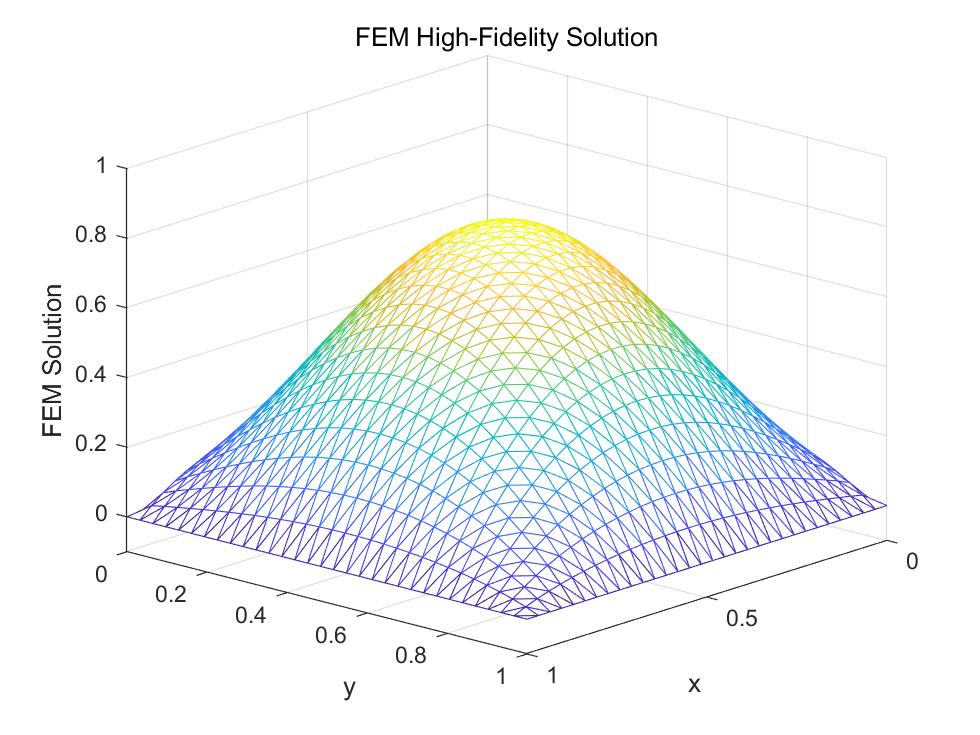}}
   \subfigure[  $t = 0.50$]{\includegraphics[width=4cm, height =3cm]{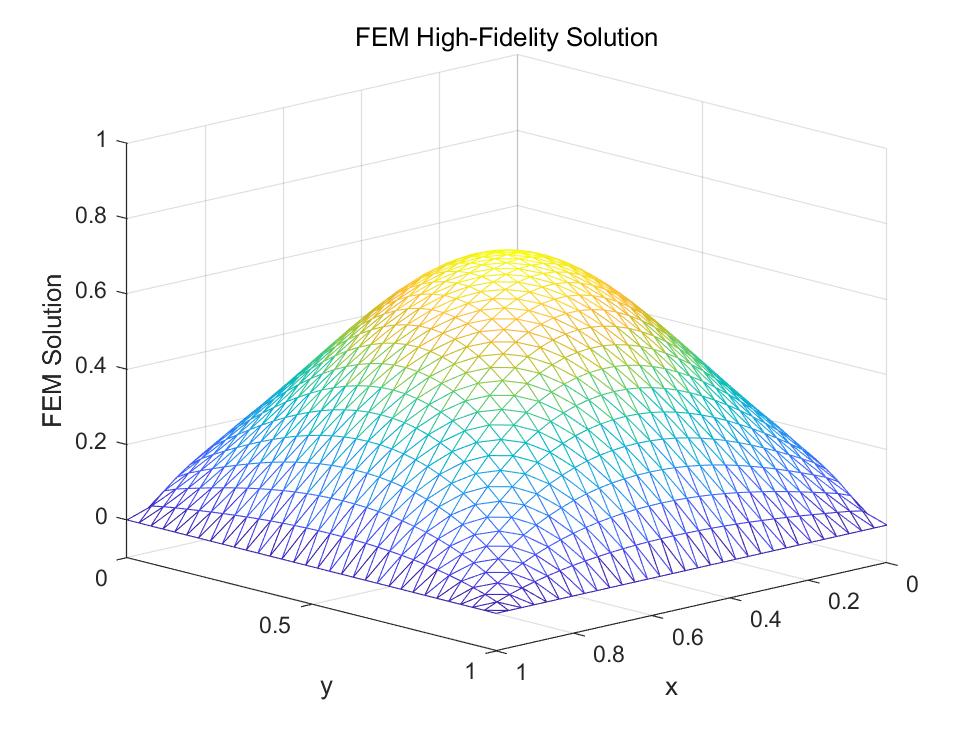}}
   
   \subfigure[  $t = 0.75$]{\includegraphics[width=4cm, height =3cm]{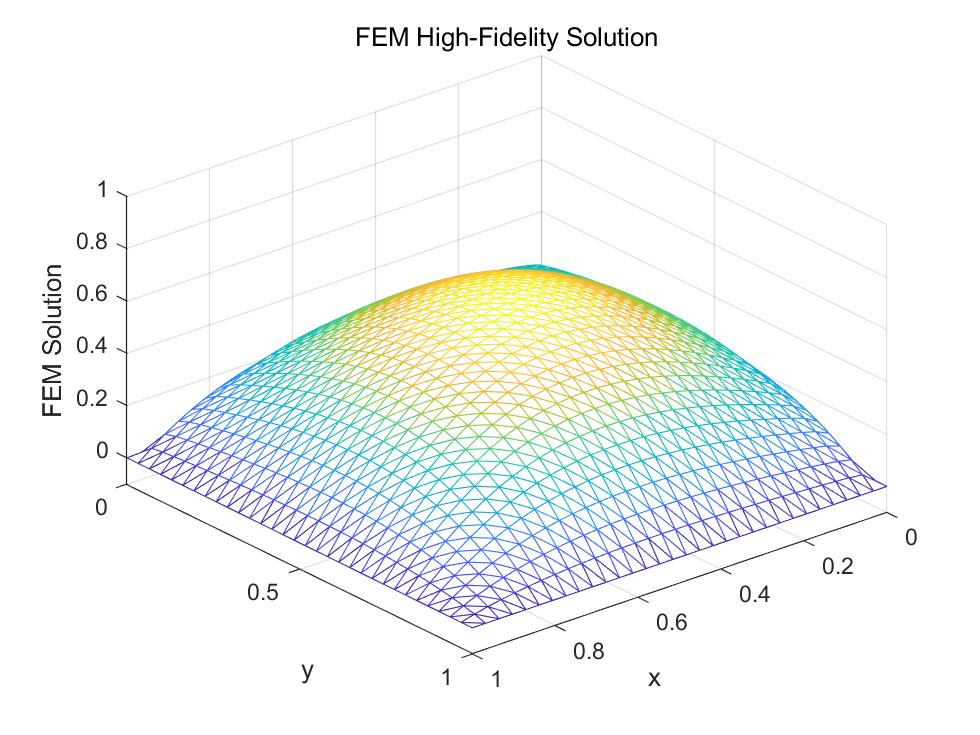}}
    \subfigure[ $t =1.00$]{\includegraphics[width=4cm, height =3cm]{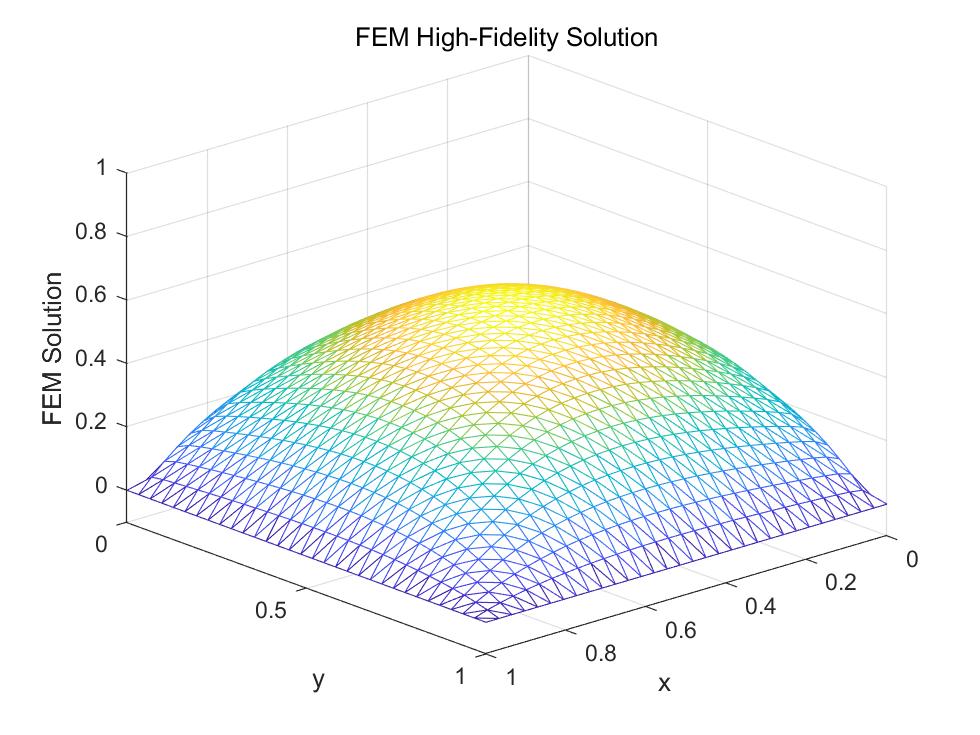}}\\
    High-fidelity numerical solution\\
    
     \subfigure[ $t =0.25$]{\includegraphics[width=4cm, height =3cm]{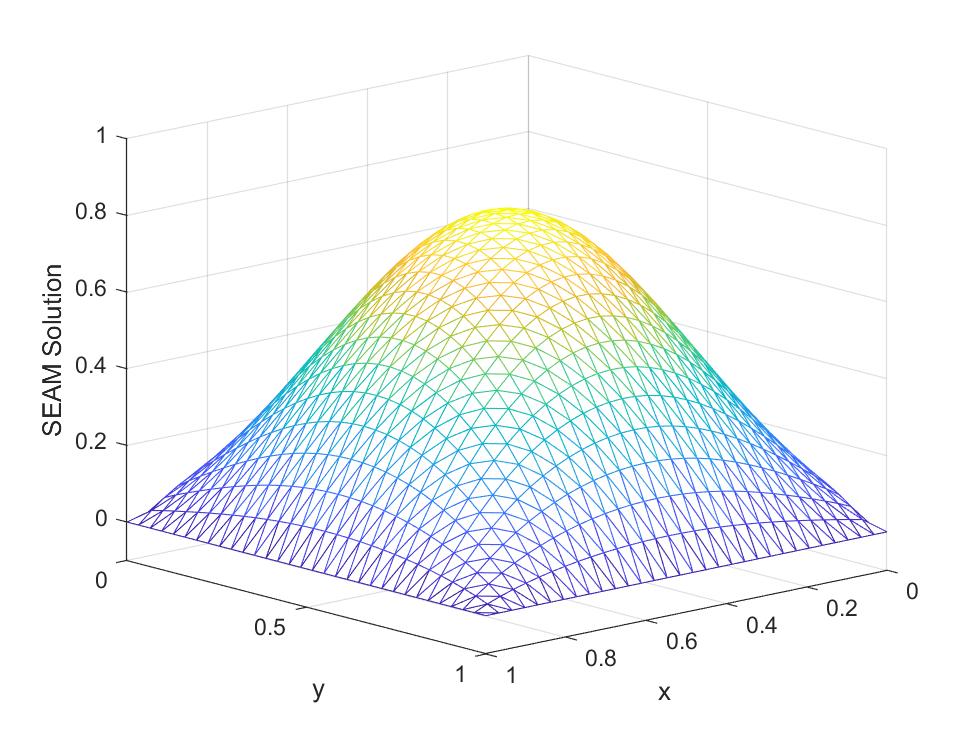}}
   \subfigure[ $t =0.50$]{\includegraphics[width=4cm, height =3cm]{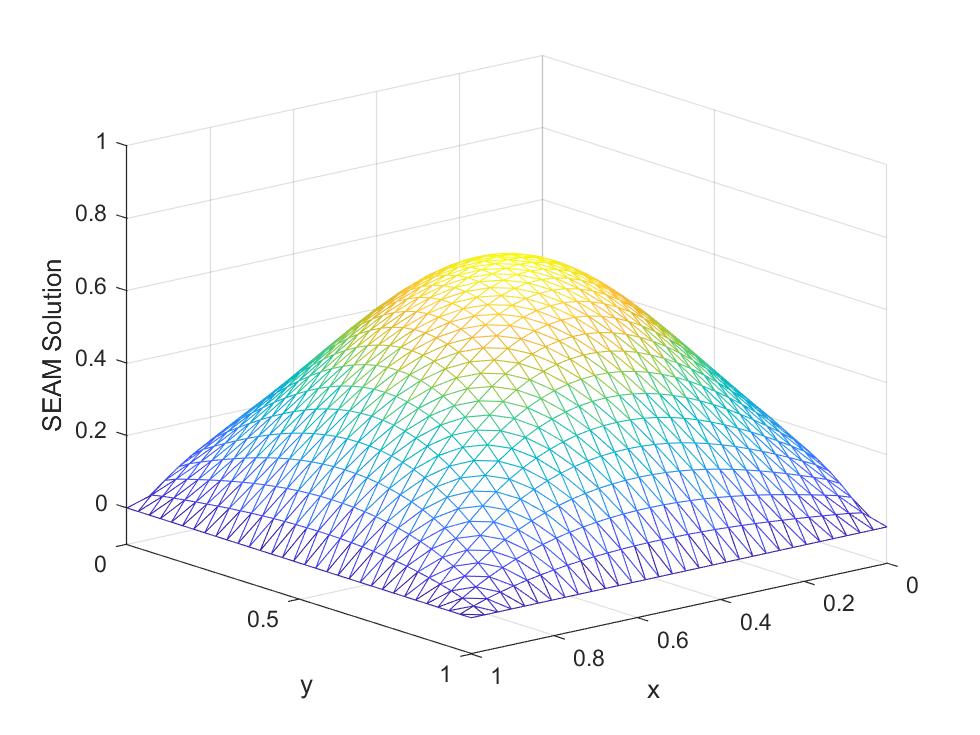}}
   
   \subfigure[  $t =0.75$]{\includegraphics[width=4cm, height =3cm]{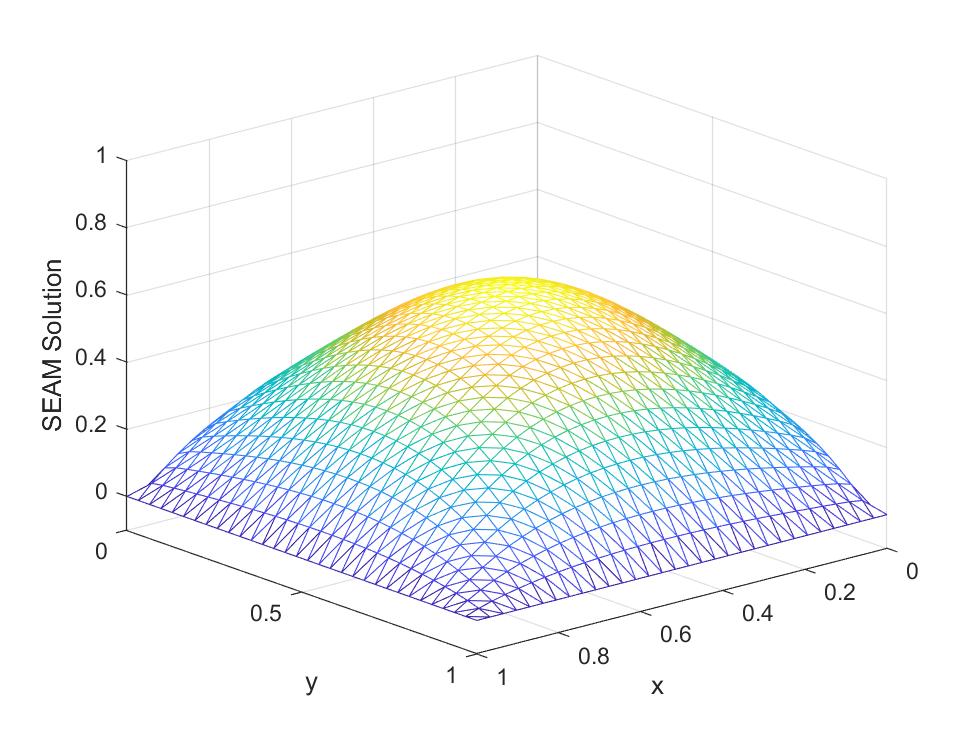}}
    \subfigure[ $t =1.00$]{\includegraphics[width=4cm, height =3cm]{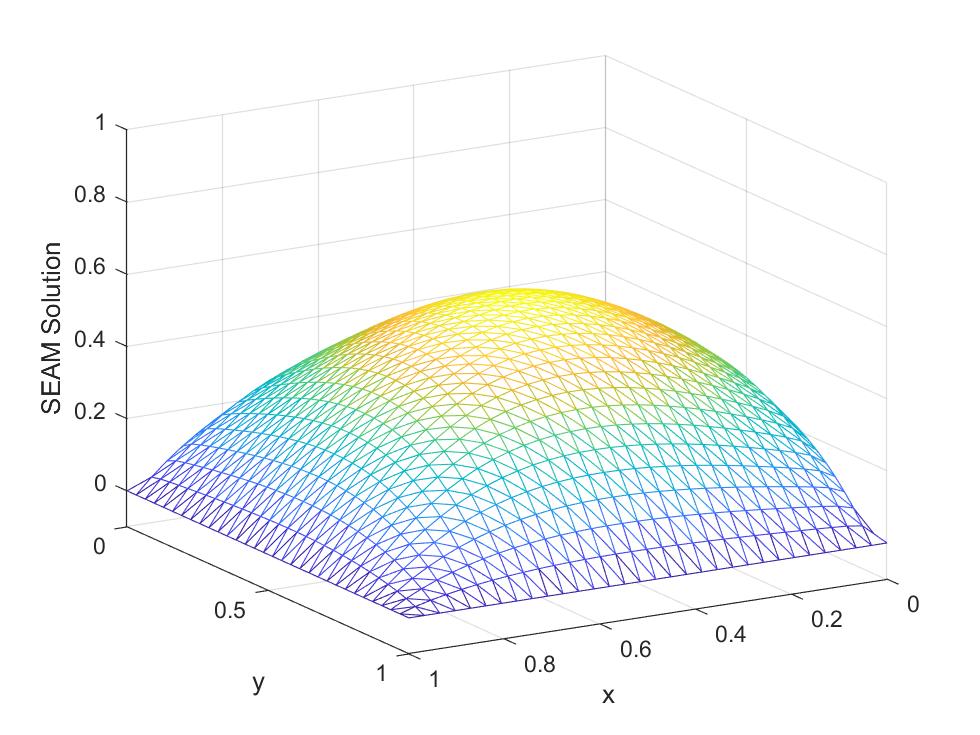}}\\
    SEAM solution\\
    \caption{High-fidelity / SEAM solutions of (\ref{nml2}) at different time for S2 with $f = 10$.}\label{seams2f10}
\end{figure}
\begin{figure}[htbp]
    \centering
   \subfigure[  $t = 0.25$]{\includegraphics[width=4cm, height =3cm]{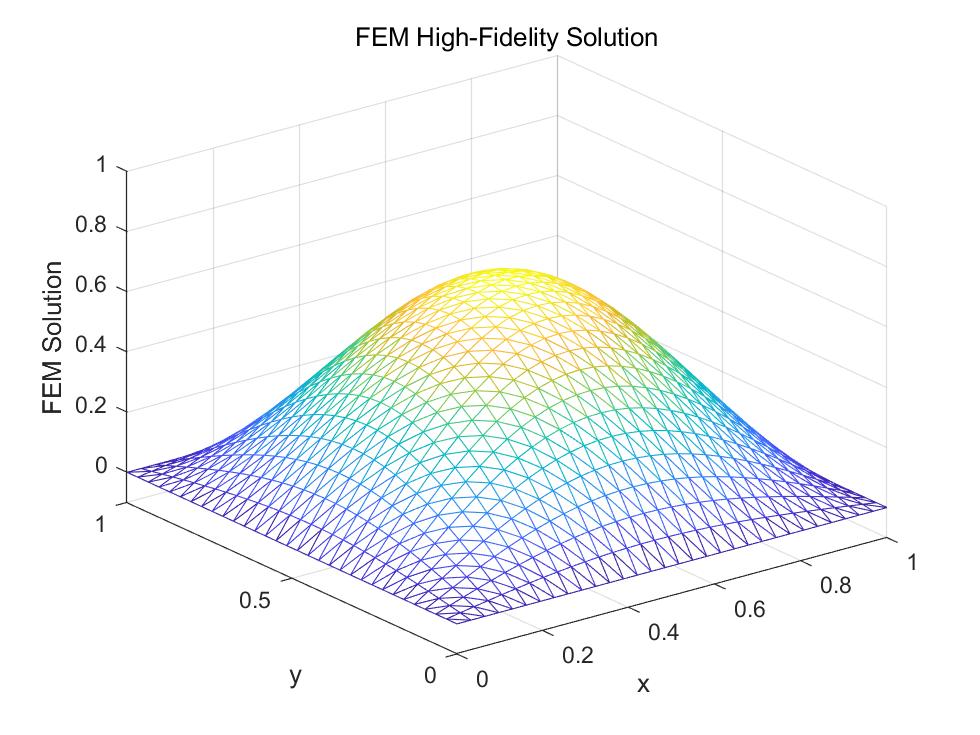}}
   \subfigure[  $t = 0.50$]{\includegraphics[width=4cm, height =3cm]{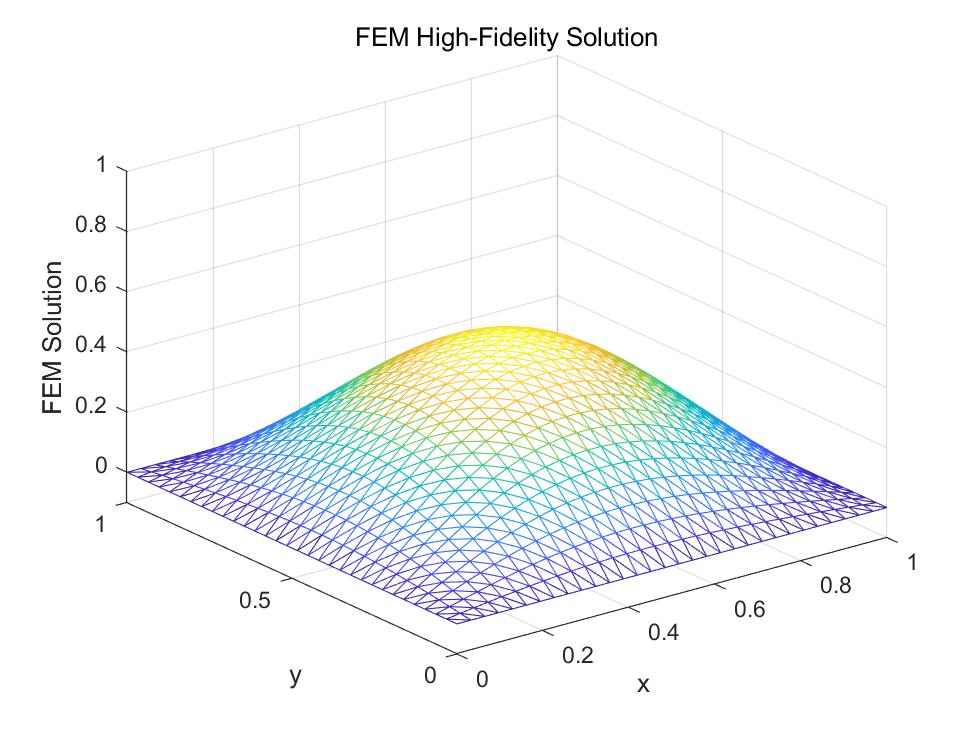}}
   
   \subfigure[  $t = 0.75$]{\includegraphics[width=4cm, height =3cm]{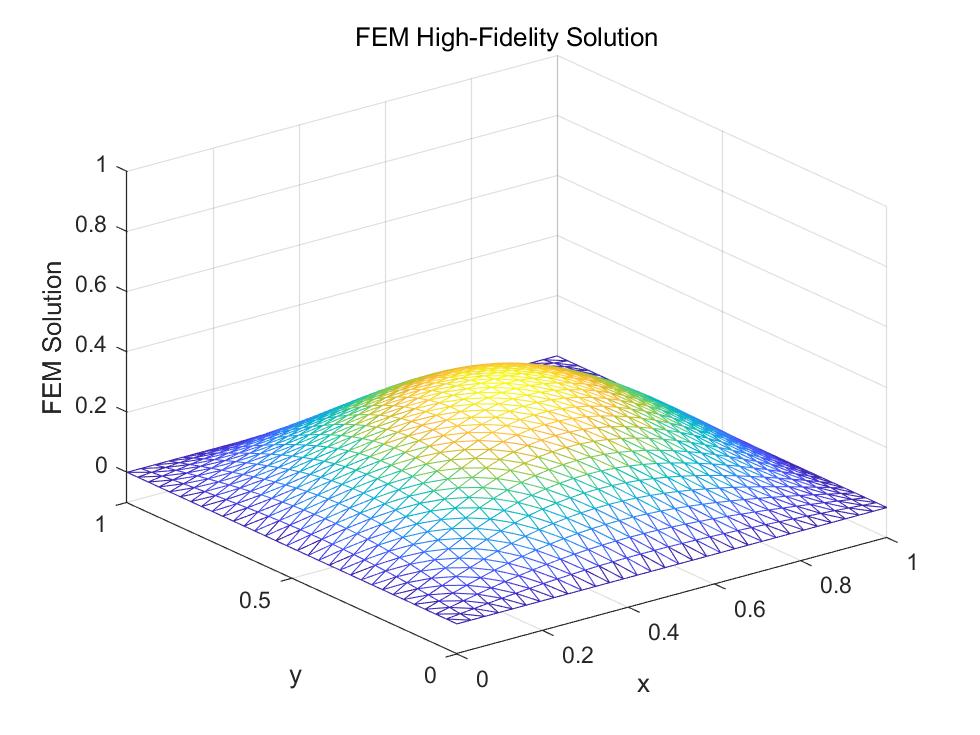}}
    \subfigure[ $t =1.00$]{\includegraphics[width=4cm, height =3cm]{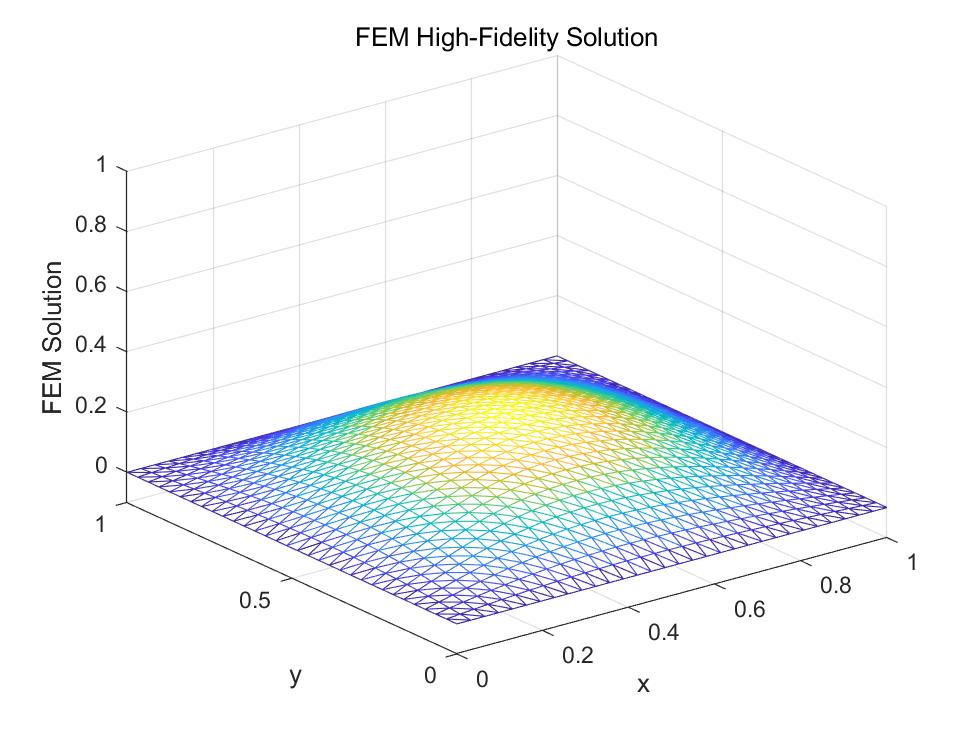}}\\
    High-fidelity numerical solution\\
    
     \subfigure[ $t =0.25$]{\includegraphics[width=4cm, height =3cm]{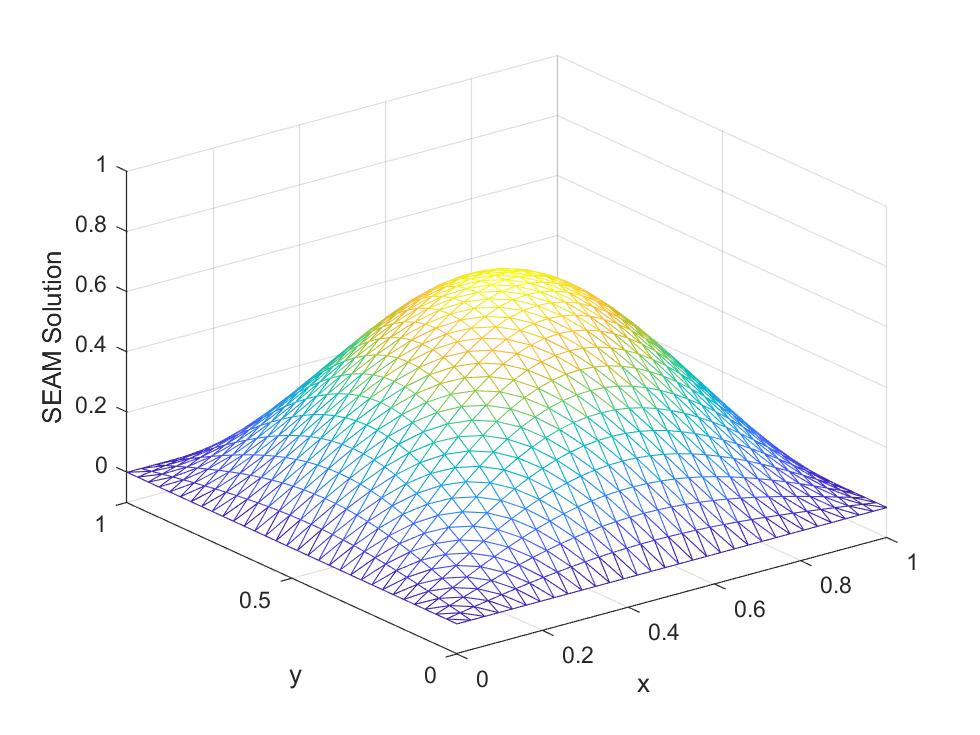}}
   \subfigure[ $t =0.50$]{\includegraphics[width=4cm, height =3cm]{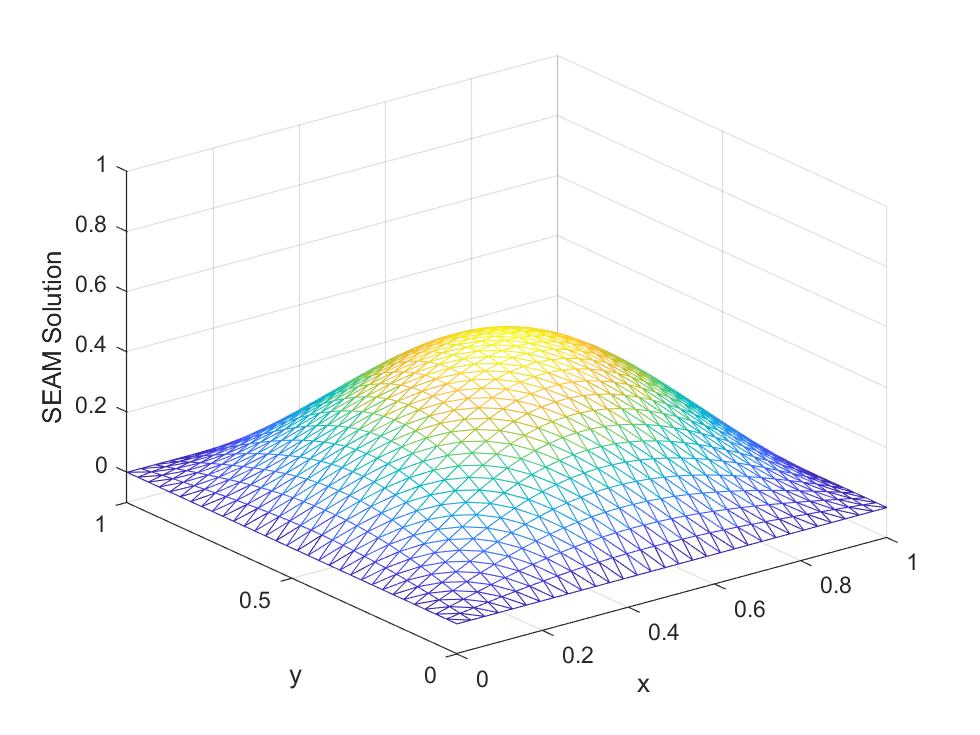}}
   
   \subfigure[  $t =0.75$]{\includegraphics[width=4cm, height =3cm]{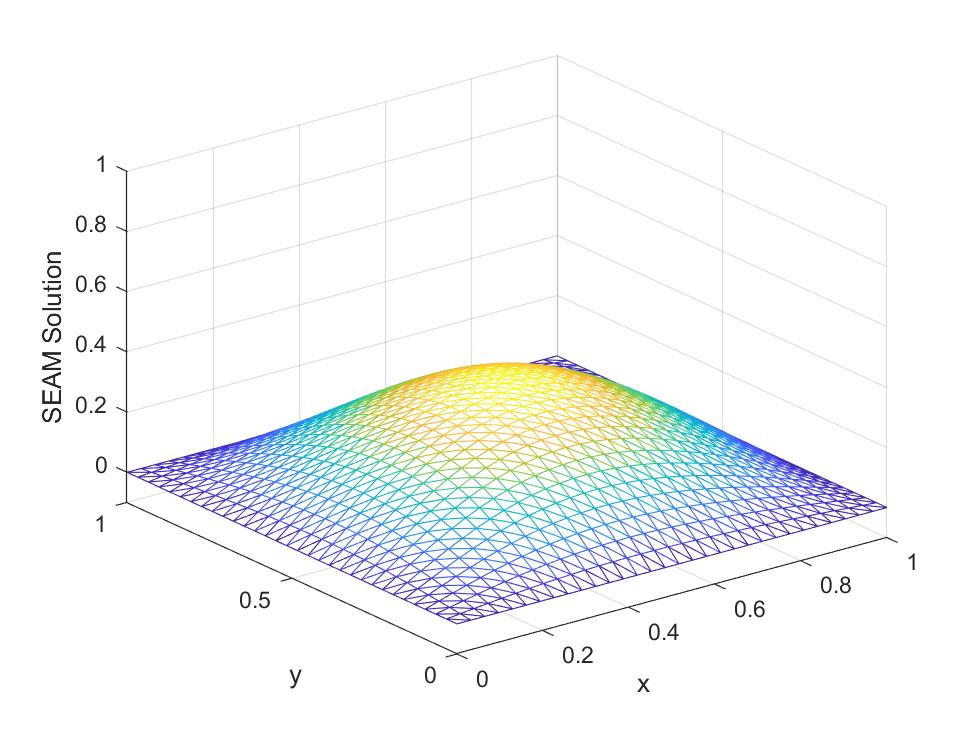}}
    \subfigure[ $t =1.00$]{\includegraphics[width=4cm, height =3cm]{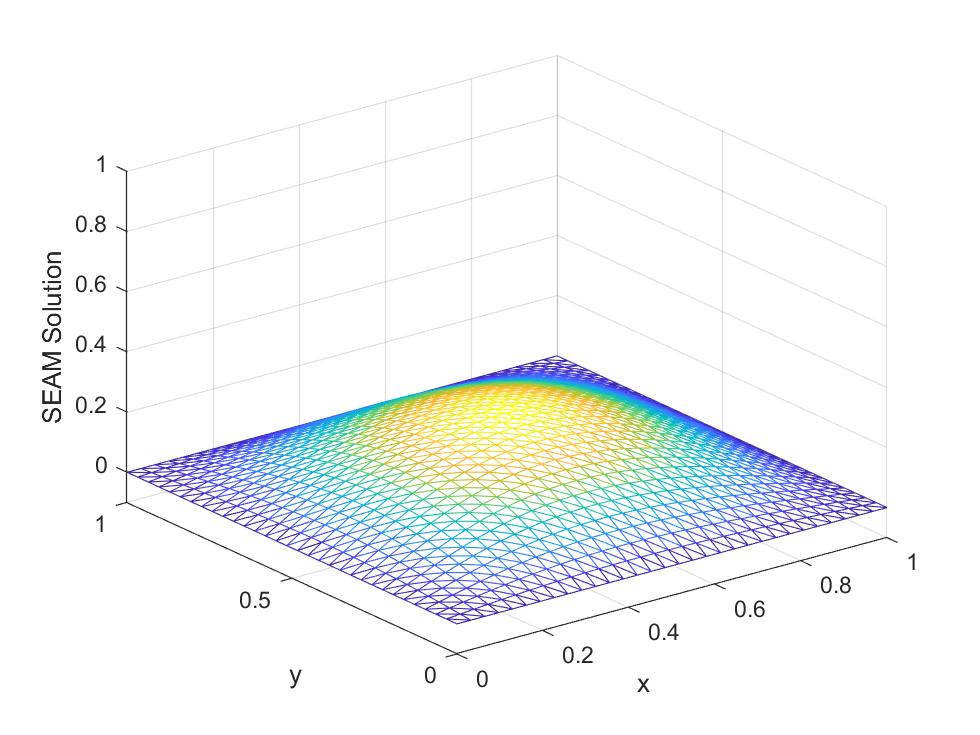}}\\
    SEAM solution\\
    \caption{High-fidelity / SEAM solutions of (\ref{nml2}) at different time for S2 with $f = xy$.}\label{seams2fxy}
\end{figure}
\begin{figure}[htbp]
    \centering
   \subfigure[  $t = 0.25$]{\includegraphics[width=4cm, height =3cm]{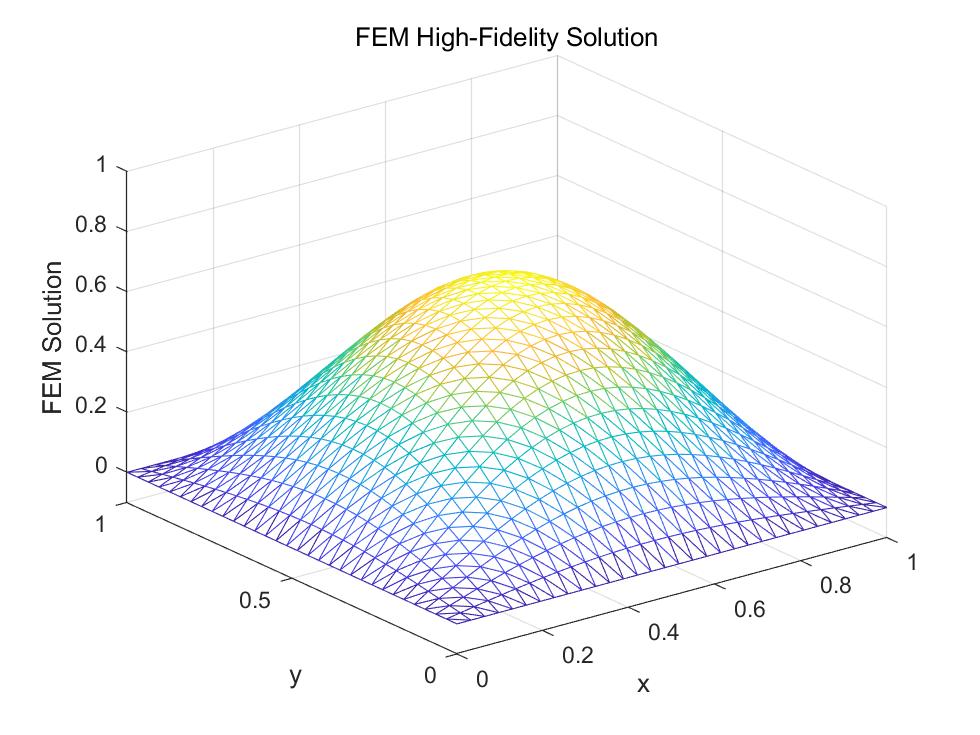}}
   \subfigure[  $t = 0.50$]{\includegraphics[width=4cm, height =3cm]{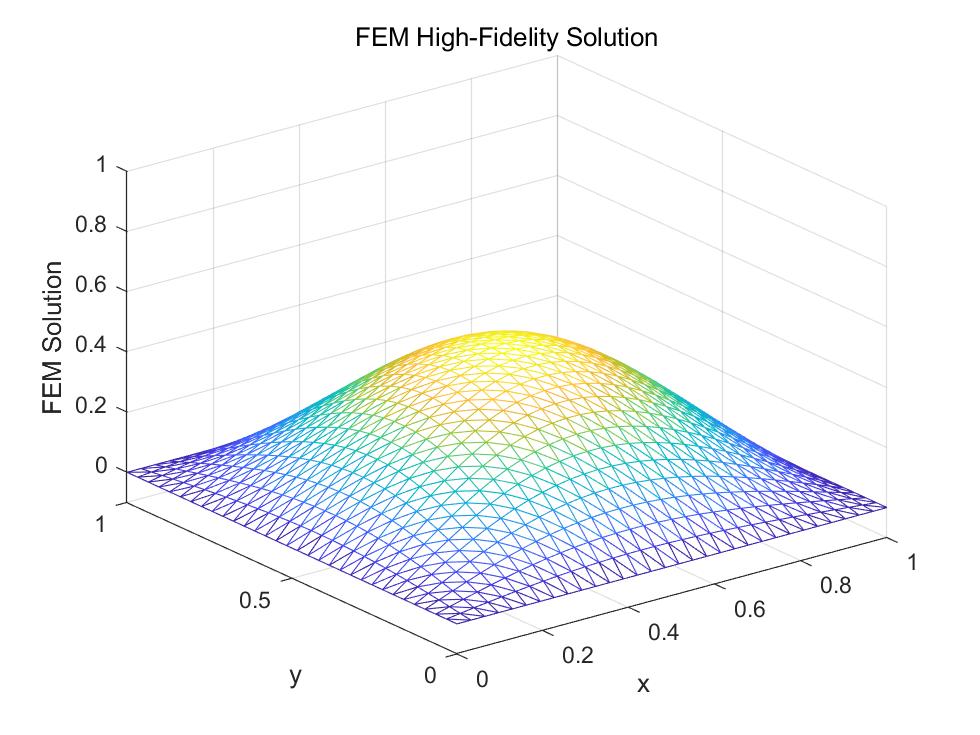}}
   
   \subfigure[  $t = 0.75$]{\includegraphics[width=4cm, height =3cm]{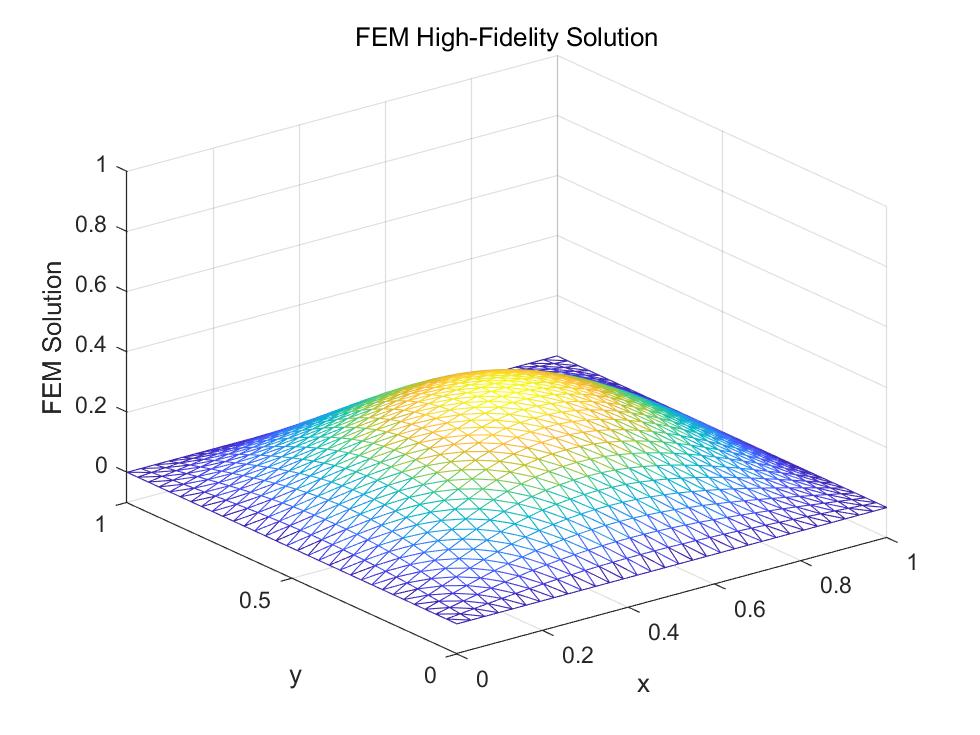}}
    \subfigure[ $t =1.00$]{\includegraphics[width=4cm, height =3cm]{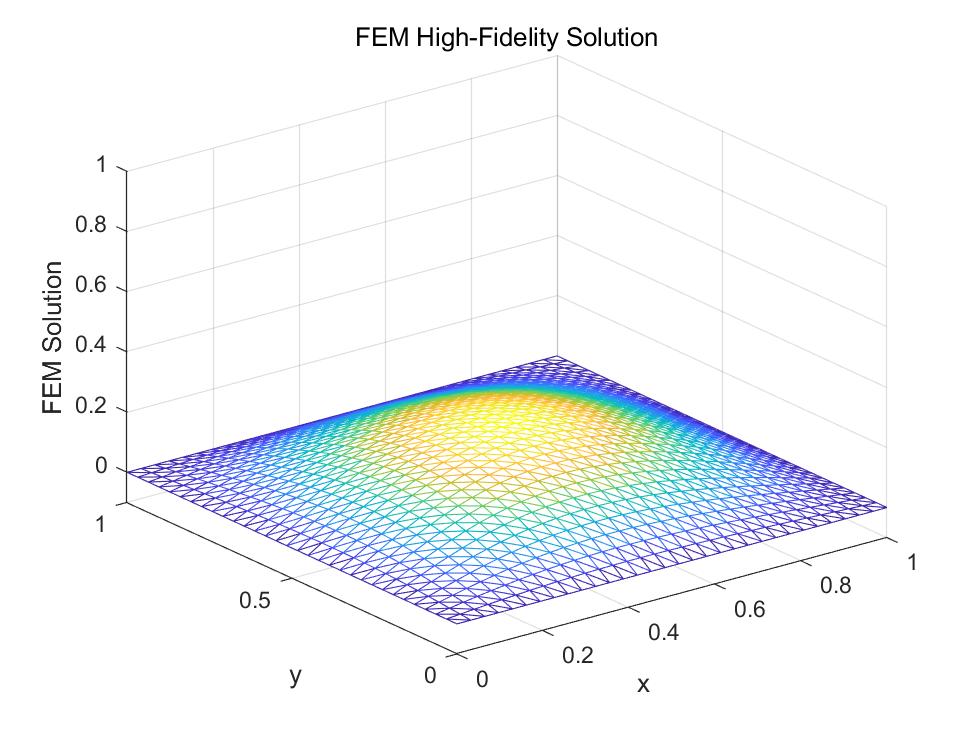}}\\
    High-fidelity numerical solution\\
    
     \subfigure[ $t =0.25$]{\includegraphics[width=4cm, height =3cm]{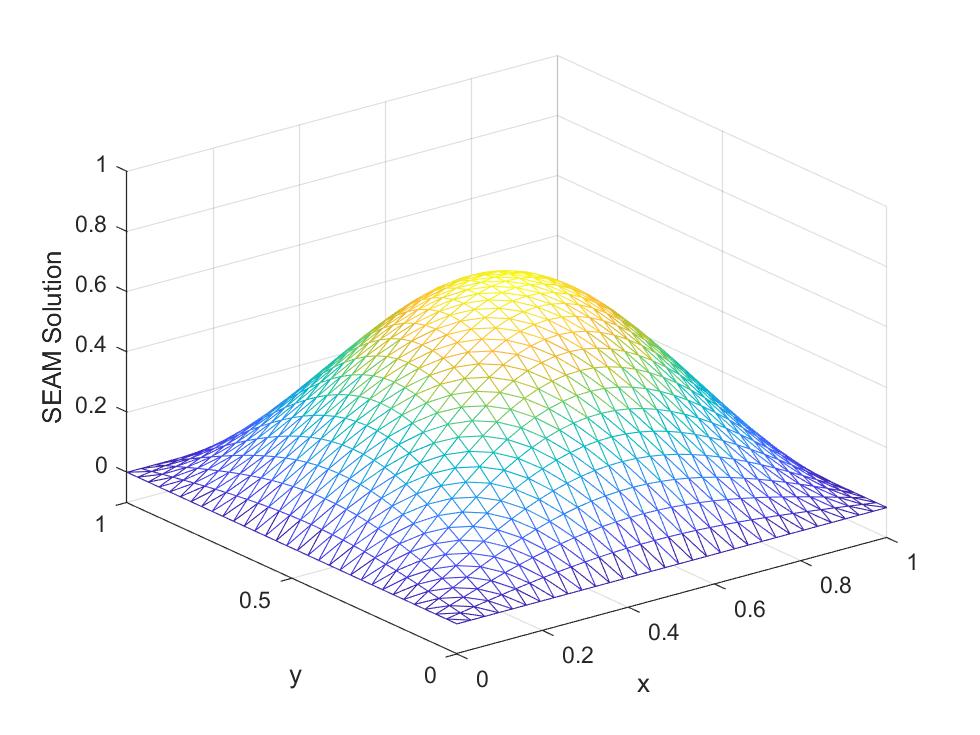}}
   \subfigure[ $t =0.50$]{\includegraphics[width=4cm, height =3cm]{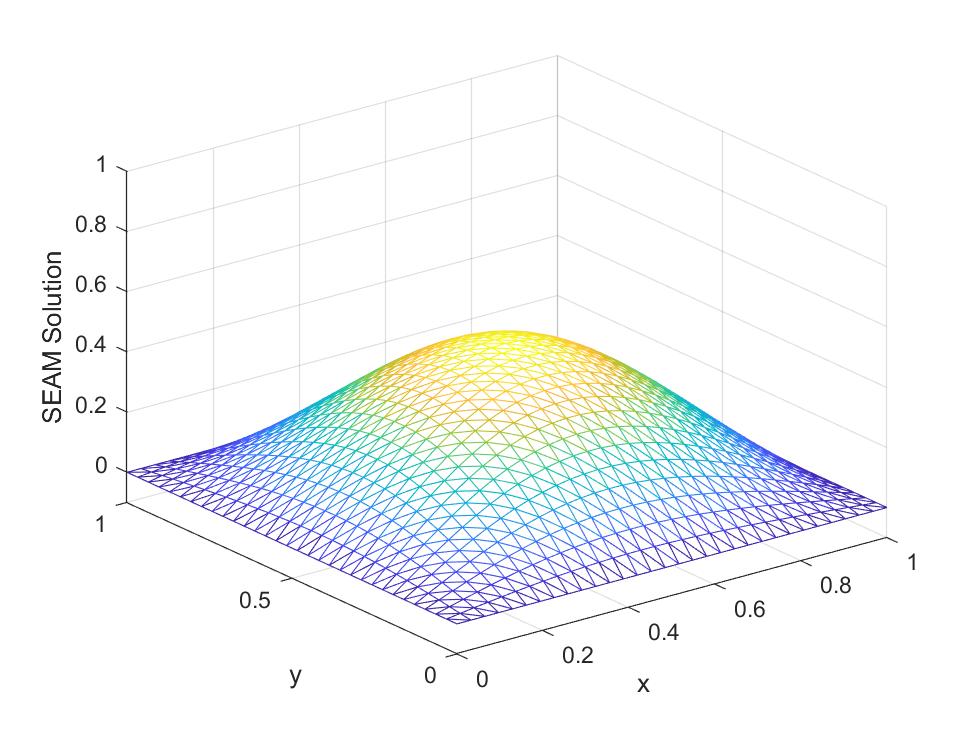}}
   
   \subfigure[  $t =0.75$]{\includegraphics[width=4cm, height =3cm]{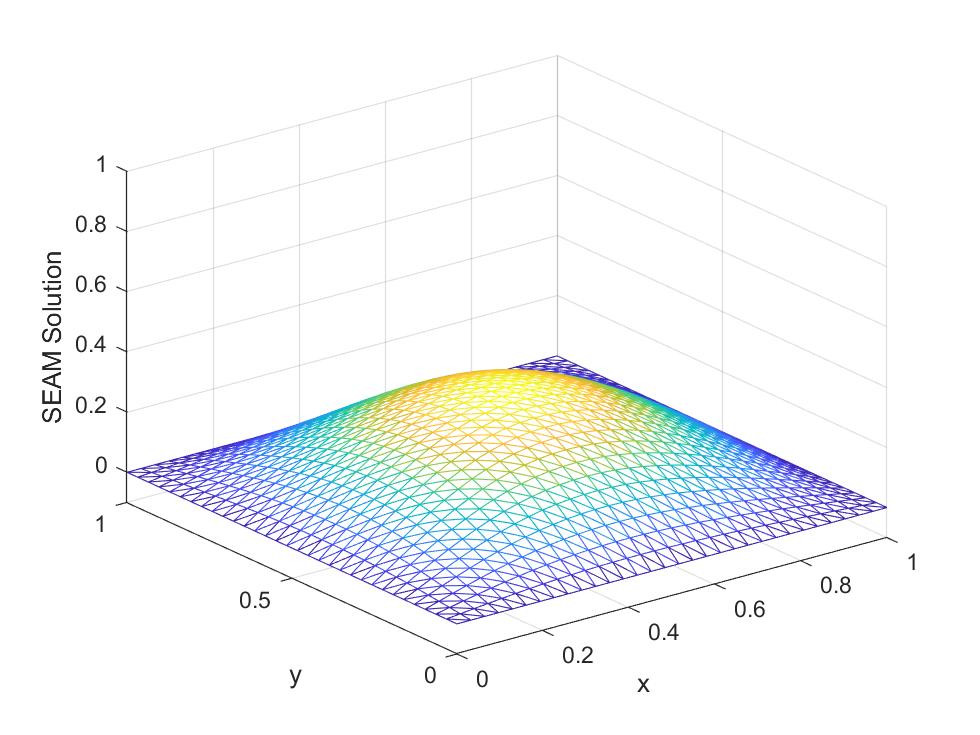}}
    \subfigure[ $t =1.00$]{\includegraphics[width=4cm, height =3cm]{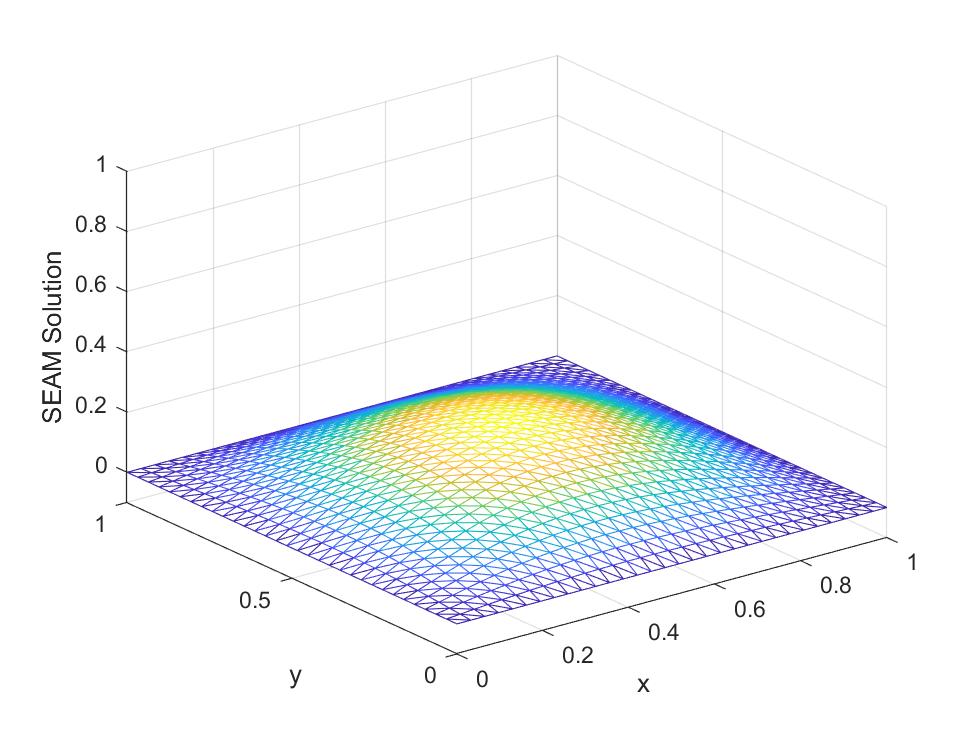}}\\
    SEAM solution\\
    \caption{High-fidelity / SEAM solutions of (\ref{nml2}) at different time for S3 with $f = 0$.}\label{seams3f0}
\end{figure}

\begin{table}[htbp]
    \centering
    \begin{tabular}{lr|lr}
    \hline
    \multicolumn{2}{l}{High-fidelity model}  &\multicolumn{2}{l}{SEAM}\\
    \hline
    \multirow{2}{*}{\makecell[l]{Number of FEM d.o.fs \\(each $t_n$)}}& \multirow{2}{*}{$M=961$ }    &  \multirow{2}{*}{\makecell[l]{Number of SEAM d.o.fs \\(each $t_n$)}} &\multirow{2}{*}{$p=1$}\\
    
    &&&\\
    \hline
             &      &     d.o.fs reduction & $961:1$\\
             \hline
    FE solution time ($S1,f=0$) &   8419s   &     SEAM solution time  & 12s\\
    FE solution time ($S1,f=xy$) &   8432s   &     SEAM solution time  & 14s\\
    FE solution time ($S2,f=0$) &   8510s   &     SEAM solution time  & 12s\\
    FE solution time ($S2,f=10$) &   8529s   &     SEAM solution time  & 11s\\
    FE solution time ($S2,f=xy$) &   8540s   &     SEAM solution time  & 13s\\  
    FE solution time ($S3,f=0$) &   301s   &     SEAM solution time  & 3s\\  
    \hline
    {SEAM-Error$_{L^{2}}$($S1,f=0$)} & &{$1.2e{-8}$}&\\
    {SEAM-Error$_{L^{2}}$($S1,f=xy$)} & & {$1.4e{-8}$}&\\
    {SEAM-Error$_{L^{2}}$($S2,f=0$)} & & {$2.1e{-8}$}&\\
    {SEAM-Error$_{L^{2}}$($S2,f=10$)} &  &{$1.6e{-8}$}&\\
    {SEAM-Error$_{L^{2}}$($S2,f=xy$)} &  &{$1.2e{-8}$}&\\
    {SEAM-Error$_{L^{2}}$($S3,f=0$)} &  &{$1.4e{-8}$}&\\
    \hline
    \end{tabular}
    \caption{Computational details for the high-fidelity / SEAM  solutions.}\label{details-2D}
\end{table}

\subsection{A three-dimensional test}
We consider the following three-dimensional problem: 
\begin{equation}\label{nml3}
\left\{\begin{array}{ll}
\frac{\partial u}{\partial t}=\frac{\partial^{2} u}{\partial x^{2}}+\frac{\partial^{2} u}{\partial y^{2}} + \frac{\partial^{2} u}{\partial z^{2}}, & \text{ in } D\times (0,T], \\
u=\sin (2 \pi x)\sin (2 \pi y)\sin(2 \pi z), &\text { on } D \times\{0\}, \\
u = 0,& \text { on } \partial D \times[0, T],
\end{array}\right.
\end{equation}
where $D = (0,1)\times (0,1)\times(0,1)$ and $T = 1$. 
To obtain  the high-fidelity numerical solution,  we first divide   the spatial domain $D$  uniformly   into $32 \times 32 \times 32=32768$ cubes, then divide each  cube into  $6$   tetrahedrons, and for the temporal subdivision we still take  $\tau= 2.5\times 10^{-3}$ and $n=\tilde{n} = 20$
 in  the full discretization \eqref{femscheme}.

 Figure \ref{eigen-3D} shows  the distribution of eigenvalues of  $\mathfrak{U}_{i}^\top \mathfrak{U}_{i}$, which is consistent with Corollary \ref{corr1}.
  Figure \ref{3D-solu} shows the  high-fidelity /  SEAM solutions of (\ref{nml3}),  and Table \ref{details-3D} gives some computational details as well as the relative errors between the  the  SEAM solution   and the high-fidelity numerical solution.
 We can see that  the SEAM algorithm yields  an accurate approximation solution and is much faster than  the high-fidelity numerical method,  due to the remarkable reduction of the size of the discrete model from $M=29397$ to $p=1$.

\begin{figure}[htbp]
    \centering
    \includegraphics[width=9cm, height =5cm]{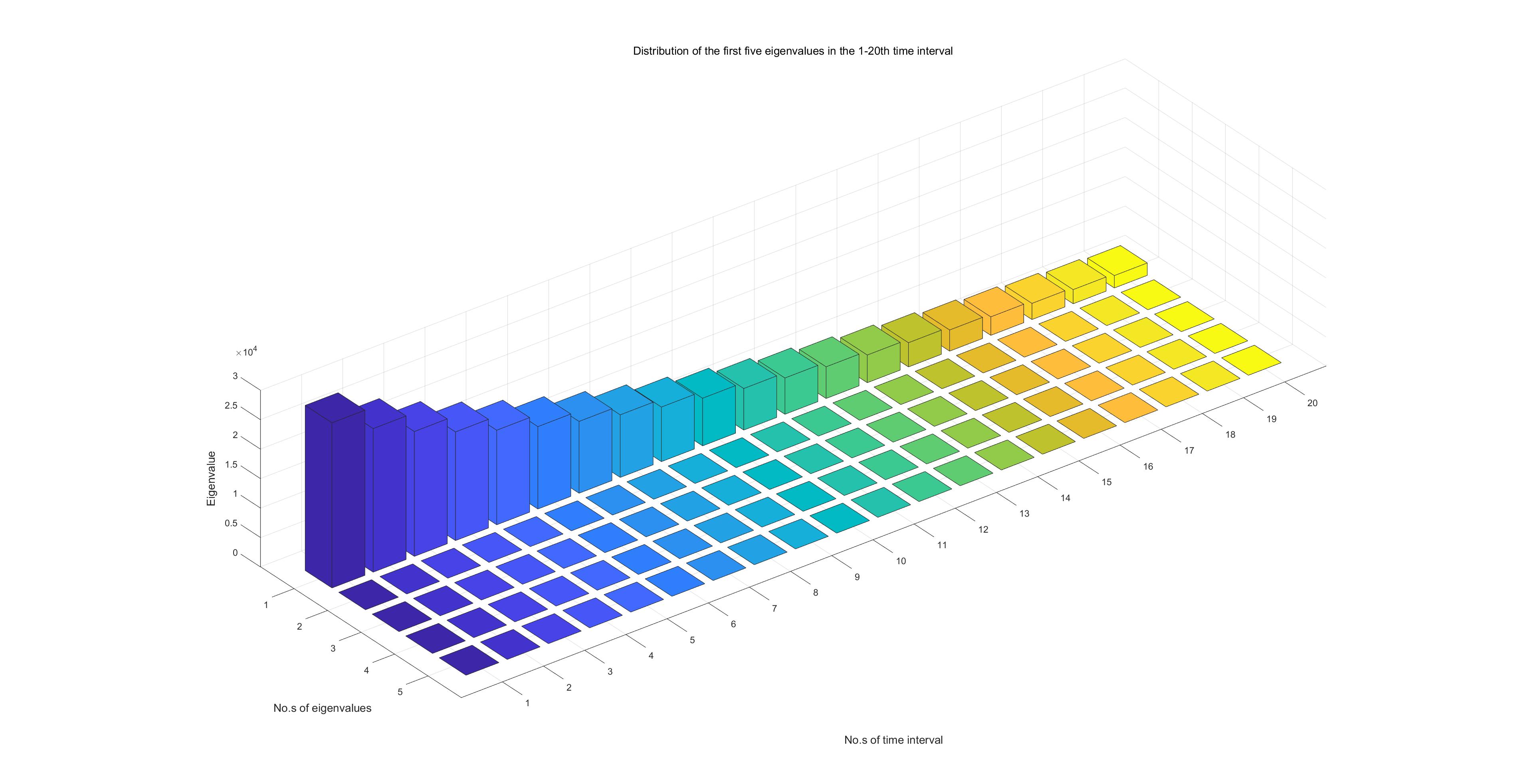}
    \caption{The first 5 eigenvalues of $\mathfrak{U}_{i}^\top \mathfrak{U}_{i}$ (from big to small) for \eqref{nml3}.}\label{eigen-3D}
\end{figure}

\begin{figure}[htbp]
    \centering
   \subfigure[  $t = 0.25$]{\includegraphics[width=5cm, height =3cm]{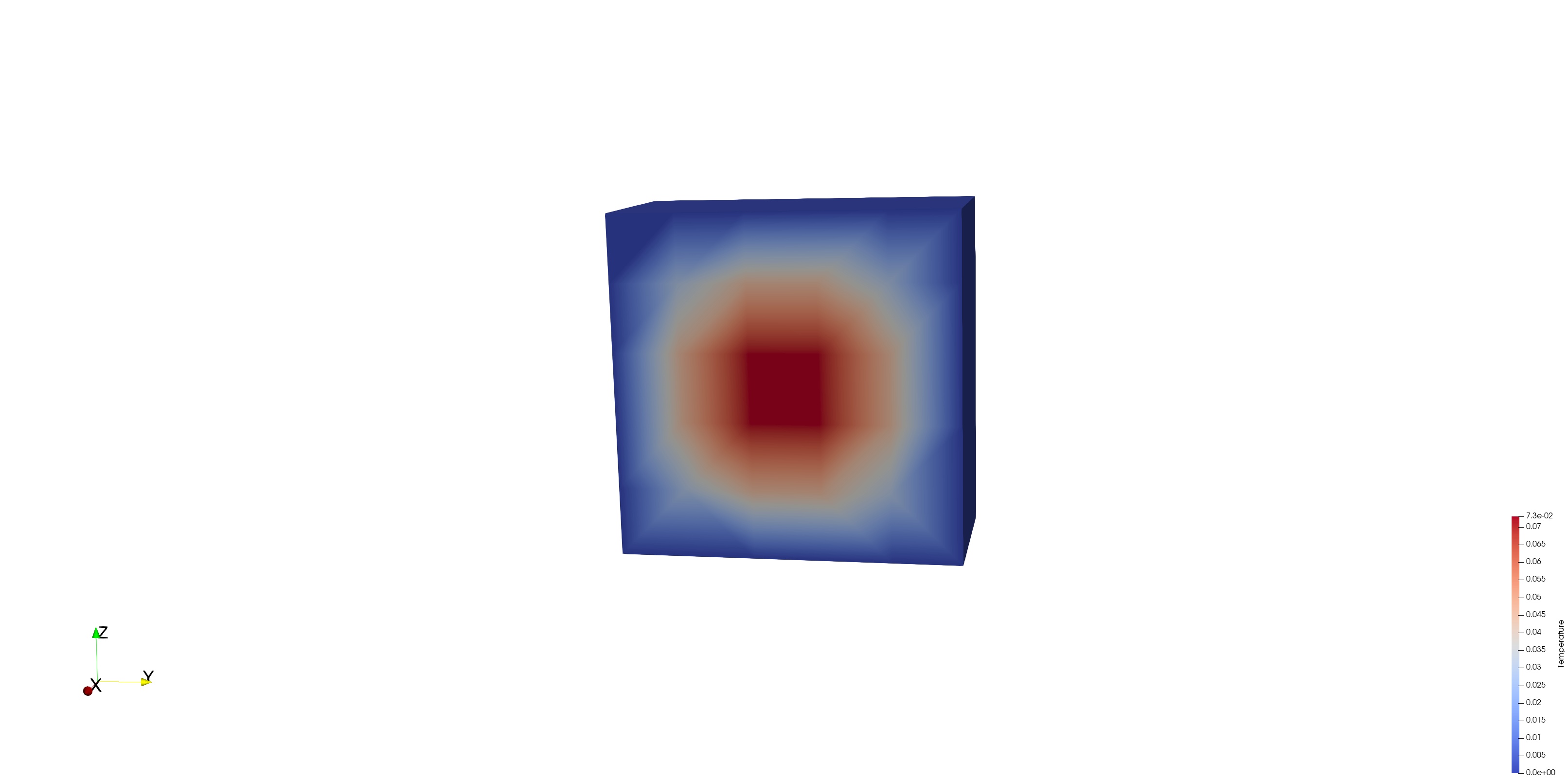}}
   \subfigure[  $t = 0.50$]{\includegraphics[width=5cm, height =3cm]{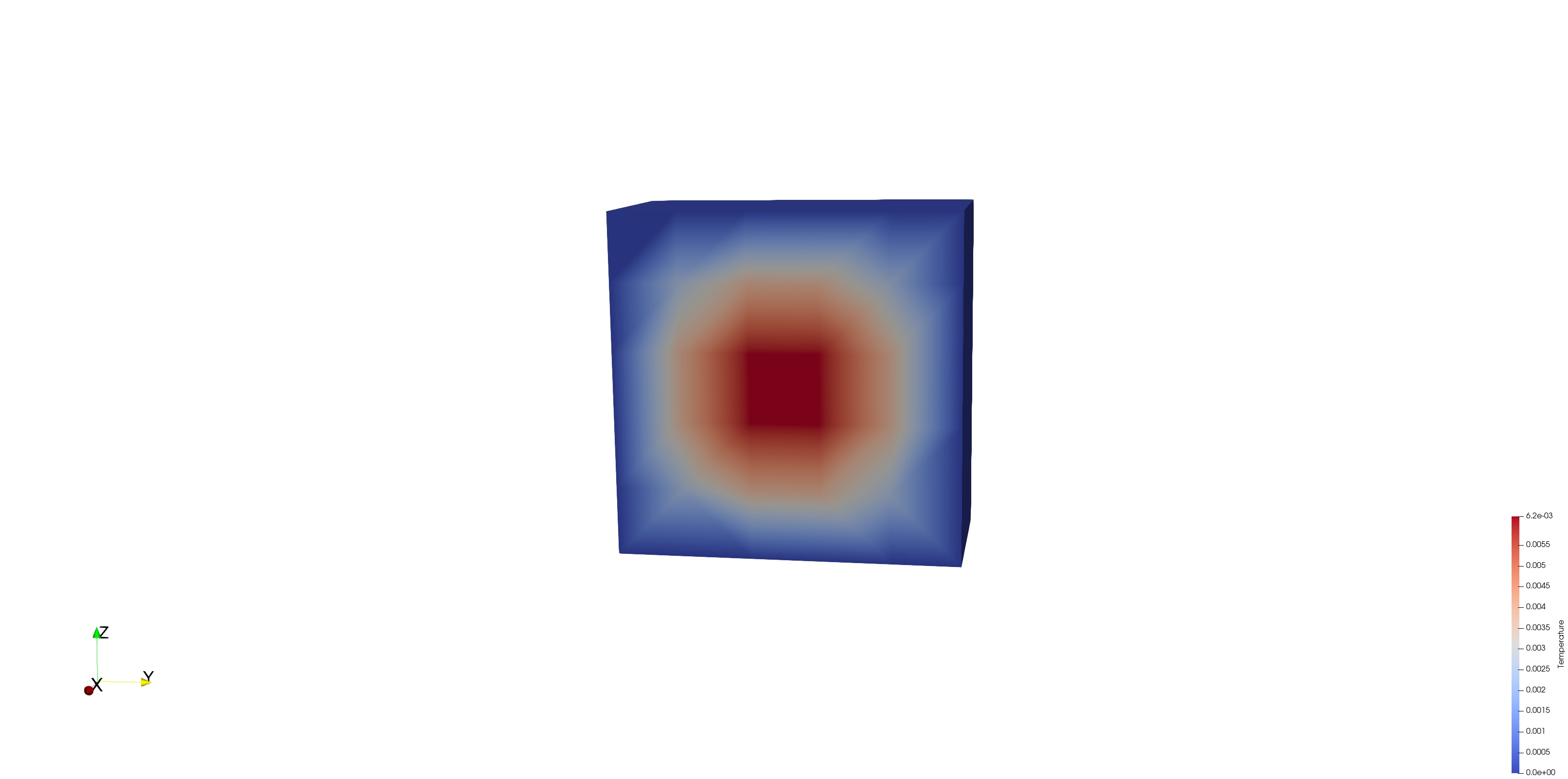}}
   
   \subfigure[  $t = 0.75$]{\includegraphics[width=5cm, height =3cm]{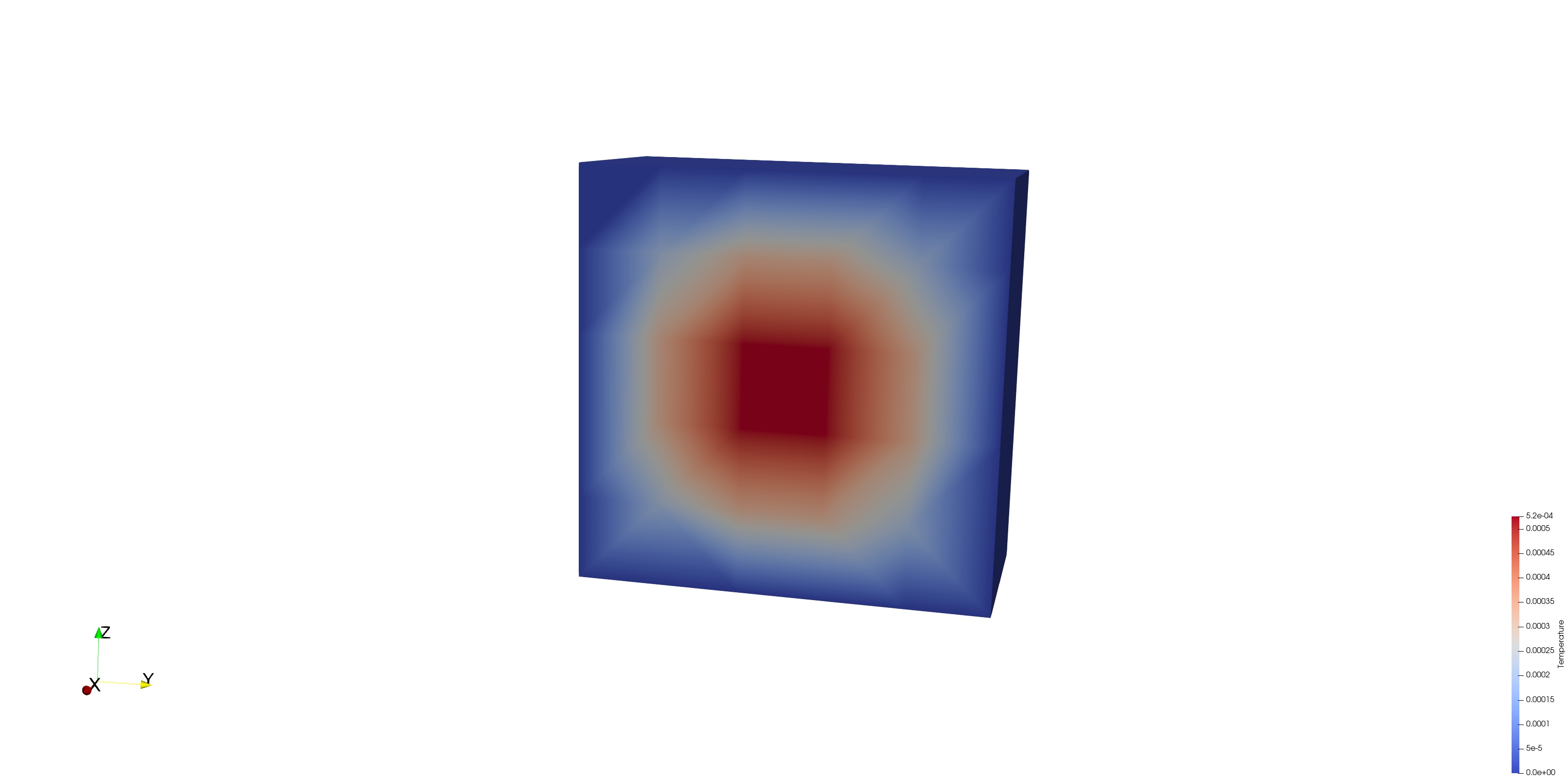}}
    \subfigure[ $t =1.00$]{\includegraphics[width=5cm, height =3cm]{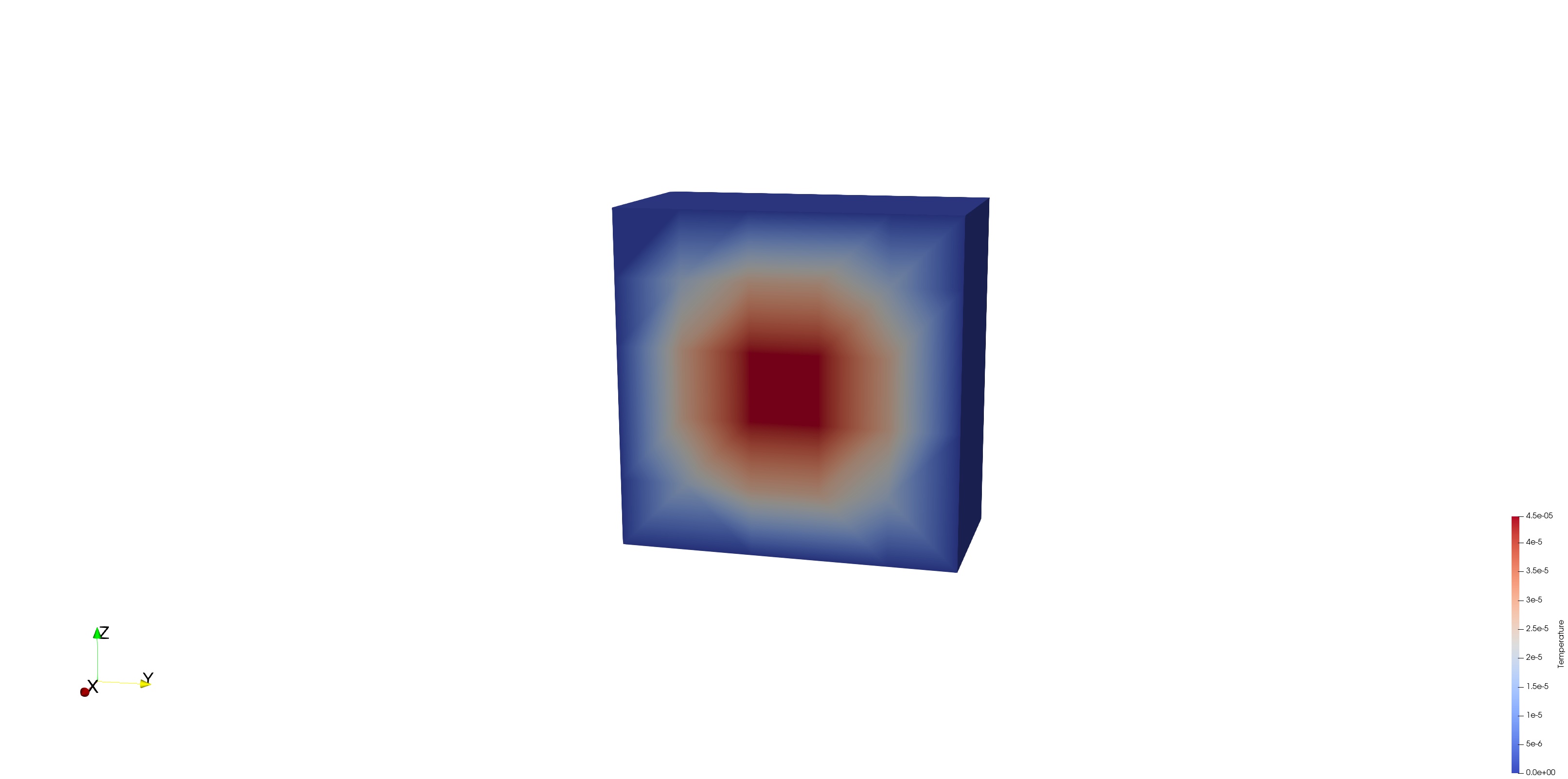}}\\
    High-fidelity numerical solution\\
    
     \subfigure[ $t =0.25$]{\includegraphics[width=5cm, height =3cm]{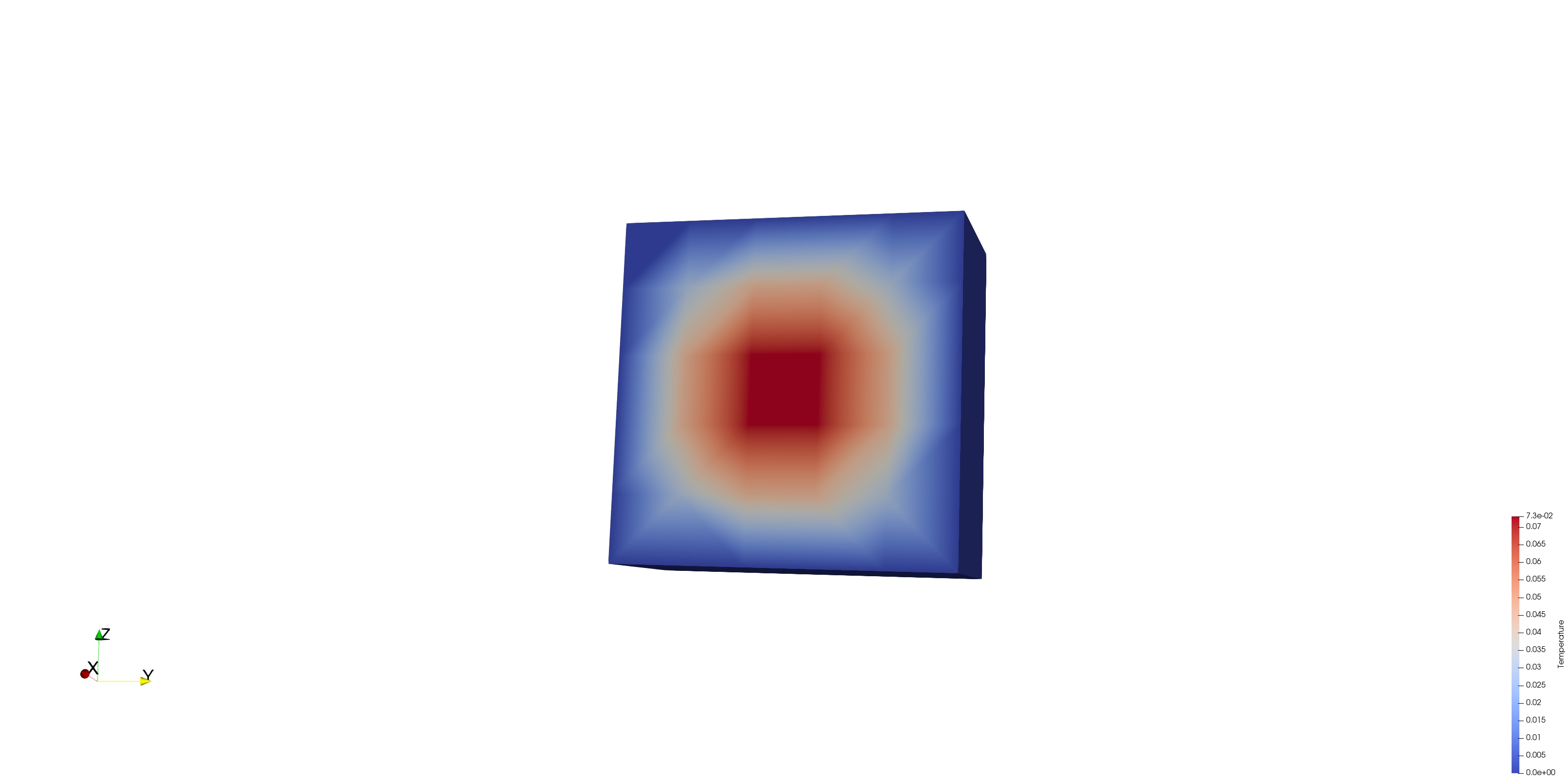}}
   \subfigure[ $t =0.50$]{\includegraphics[width=5cm, height =3cm]{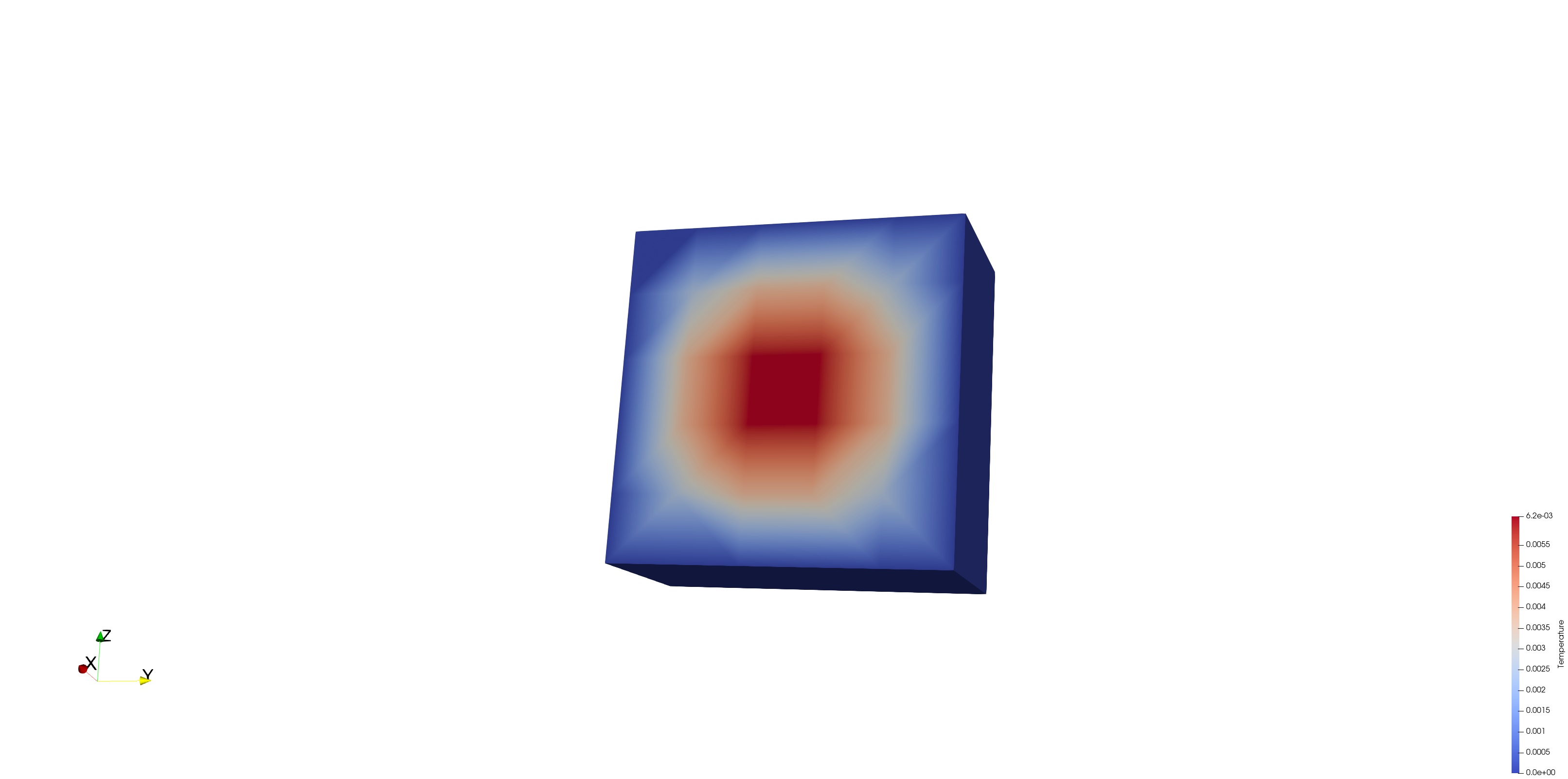}}
   
   \subfigure[  $t =0.75$]{\includegraphics[width=5cm, height =3cm]{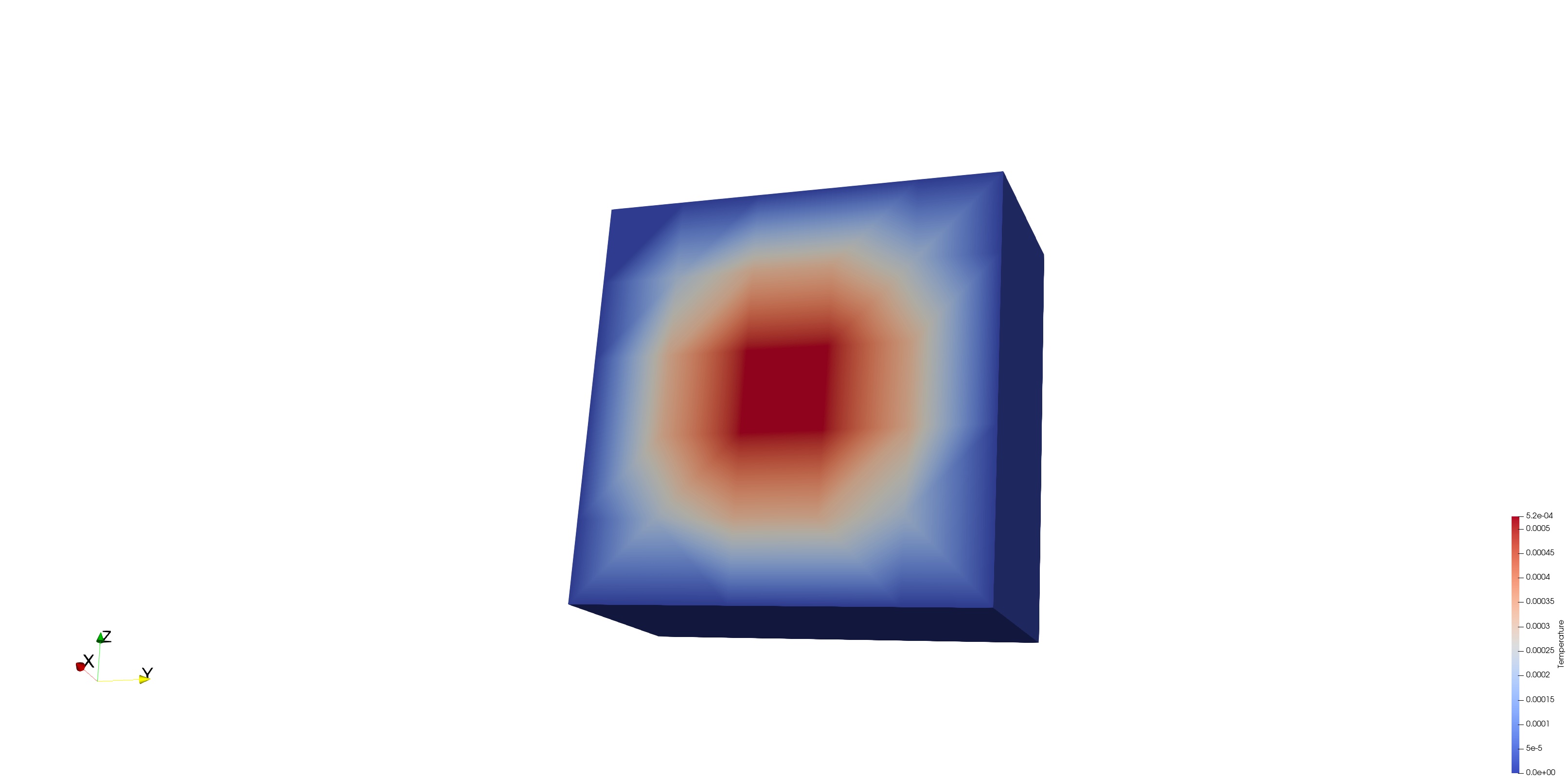}}
    \subfigure[ $t =1.00$]{\includegraphics[width=5cm, height =3cm]{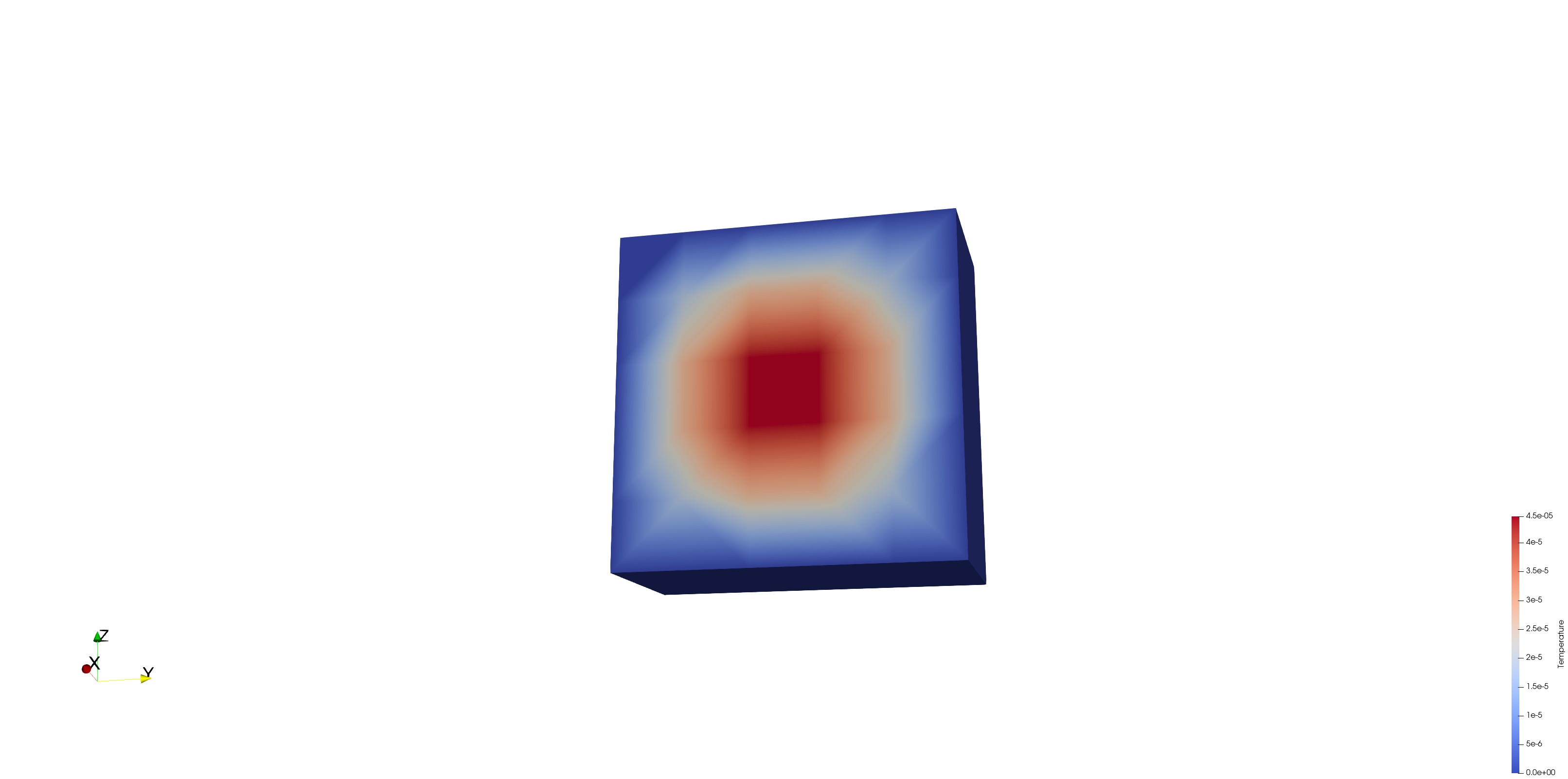}}\\
    SEAM solution\\
    \caption{High-fidelity / SEAM solutions of (\ref{nml3}) at different time.}\label{3D-solu}
\end{figure}

\begin{table}[htbp]
    \centering
    \begin{tabular}{lr|lr}
    \hline
    \multicolumn{2}{l}{High-fidelity model}  &\multicolumn{2}{l}{SEAM}\\
    \hline
    \multirow{2}{*}{\makecell[l]{Number of FEM d.o.fs \\(each $t_n$)}}& \multirow{2}{*}{$M=29397$ }    & \multirow{2}{*}{\makecell[l]{ Number of SEAM d.o.fs \\ (each $t_n$)}} &\multirow{2}{*}{$p=1$}\\
    &&&\\
    \hline
             &      &     d.o.fs reduction & $29397:1$\\
   \hline
    FE solution time &   40320s   &     SEAM solution time  &  8s\\
    \hline
    \multicolumn{2}{l}{$\text{SEAM-Error}_{L^{2}}$} &  \multicolumn{2}{l}{$1.2e{-8 }$}\\
    \hline
    \end{tabular}
    \caption{Computational details for the high-fidelity / SEAM solutions of (\ref{nml3}).}\label{details-3D}
\end{table}

\section{Conclusion}

We have  studied    the singular value distribution   of the numerical solution matrix corresponding to a full discretization of  second order parabolic equations that  uses the backward Euler scheme in time and continuous simplicial finite elements in space. Based on the property that the singular value distribution of the matrix formed by the high-fidelity solution is similar to a rank one matrix 
under the assumption that the time step size is sufficiently small, 
we have proposed  the SEAM algorithm by choosing  the eigenvector associated to the largest   singular value as the POD basis. The resulting RB scheme at each time step is a linear equation with only one unknown. 
Numerical experiments have demonstrated the remarkable efficiency of SEAM.

\section{Acknowledgement}
This work was supported in part by the National Natural Science Foundation of China (12171340) and National Key R\&D Program of China (2020YFA0714000).

\newpage
\bibliographystyle{siam}
\bibliography{ref}
\end{document}